\documentclass[10pt,a4paper]{article}

\usepackage[english]{babel}
\usepackage[utf8]{inputenc}
\usepackage{authblk}

\usepackage[T1]{fontenc}

\usepackage[absolute]{textpos}

\usepackage[vlined,ruled]{algorithm2e}
\usepackage[left=2.8cm,right=2.8cm,top=2.5cm,bottom=2.5cm]{geometry}
\SetKwInOut{Input}{input}
\SetKwInOut{Output}{output}
\SetKwComment{Comment}{}{}

\SetCommentSty{mycmtsty}

\usepackage[vlined,ruled]{algorithm2e}
\usepackage{graphicx}
\usepackage{subfigure}
\usepackage{amsmath, amsthm}
\usepackage{hyperref}
\usepackage{amsfonts}
\usepackage{float}
\usepackage{pstricks,pst-plot,pst-text,pst-tree,pst-eps,pst-fill,pst-node,pst-math}
\usepackage[justification = centering]{caption}
\usepackage{tikz}
\usepackage{setspace}
\usepackage{enumitem}
\usepackage{esint}

\usepackage{titlesec}
\titleformat{\chapter}[display]   
{\normalfont\huge\bfseries}{\chaptertitlename\ \thechapter}{20pt}{\Huge}   
\titlespacing*{\chapter}{0pt}{-30pt}{30pt}

\usepackage{fancyvrb}
\VerbatimFootnotes 

\usepackage{amsthm}
\usepackage{subfigure}
\usepackage{appendix}
\usepackage{bigints}

\usepackage{mathtools}
\mathtoolsset{showonlyrefs,showmanualtags}

\newtheorem{theorem}{Theorem}

\newtheorem{lemme}{Lemma}
\newtheorem{remark}{Remark}
\newtheorem{prop}{Proposition}
\newtheorem{corol}{Corollary}

\begin{document}

\setlength{\abovedisplayshortskip}{5.5pt}
\setlength{\belowdisplayshortskip}{5.5pt}
\setlength{\abovedisplayskip}{5.5pt}
\setlength{\belowdisplayskip}{5.5pt}

\title{\Huge Elliptic homogenization with almost translation-invariant coefficients}

\author{R\'emi Goudey}

\affil{  CERMICS, Ecole des Ponts and MATHERIALS project-team, INRIA, \\ 6 \& 8, avenue Blaise Pascal, 77455 Marne-La-Vall\'ee Cedex 2, FRANCE.\\ \textbf{remi.goudey@enpc.fr}}

\date{}

\maketitle


\begin{abstract}
We consider an homogenization problem for the second order elliptic equation  $-\operatorname{div}\left(a(./\varepsilon) \nabla u^{\varepsilon} \right)$ $=f$ when the coefficient $a$ is almost translation-invariant at infinity and models a geometry close to a periodic geometry. This geometry is characterized by a particular discrete gradient of the coefficient $a$ that belongs to a Lebesgue space $L^p(\mathbb{R}^d)$ for $p\in[1,+\infty[$. When $p<d$, we establish a discrete adaptation of the Gagliardo-Nirenberg-Sobolev inequality in order to show that the coefficient $a$ actually belongs to a certain class of periodic coefficients perturbed by a local defect. We next prove the existence of a corrector and we identify the homogenized limit of $u^{\varepsilon}$. When $p\geq d$, we exhibit admissible coefficients $a$ such that $u^{\varepsilon}$ possesses different subsequences that converge to different limits in $L^2$. 
\end{abstract}

\section{Introduction}

Our purpose is to address an homogenization problem for a second order elliptic equation in divergence form with highly oscillatory coefficients : 
\begin{equation}
\label{equationepsilon_new}
\left\{
\begin{array}{cc}
   -\operatorname{div}\left(a\left(./\varepsilon\right) \nabla u ^{\varepsilon}\right) = f   & \text{on } \Omega, \\
    u^{\varepsilon} = 0 & \text{on } \partial \Omega,
\end{array}
\right.
\end{equation}
where $\Omega$ is a bounded domain of $\mathbb{R}^d$ ($d \geq 1$), $f$ is a function in $L^2(\Omega)$ and $\varepsilon>0$ is a small scale parameter. The (matrix-valued) coefficient $a$ is assumed to model a perturbed periodic geometry and to satisfy an almost translation invariance at infinity. Such a property, which will be formalized in the sequel, ensures that the coefficient describes a non-periodic medium with a structure close to that of a periodic medium at infinity. The present work follows up on several previous works \cite{blanc2018precised, blanc2018correctors, blanc2015local, blanc2012possible} where the authors have studied the homogenization of problem \eqref{equationepsilon_new} for non-periodic geometries characterized by a known periodic background perturbed by certain local defects. This structure was generically modeled using a particular class of coefficients of the form $a= a_{per} + \Tilde{a}$ where $a_{per}$ is a periodic coefficient and $\Tilde{a}$ is a perturbation that in some formal sense vanishes at infinity since it belongs to a Lebesgue space $\left(L^p(\mathbb{R}^d)\right)^{d \times d}$ for $p \in ]1, \infty[$. In this paper, we adopt a somewhat more general approach for the study of problem \eqref{equationepsilon_new} in a context of a perturbed periodic geometry, without postulating the specific structure "$a=a_{per} + \Tilde{a}$" for the coefficient $a$. The only assumption that we make on the ambient background, which is the starting point of our study, is an assumption of almost $Q$-translation invariance at infinity (where $Q = ]0,1[^d$ denotes the $d$-dimensional unit cube) satisfied by $a$. Typically, such an assumption in dimension $d=1$ will be expressed as the integrability of the function $\delta a := a(.+1) - a$ at infinity. 

To start with, the coefficients $a$ considered are assumed to be elliptic, uniformly bounded and uniformly $\alpha$-H\"older continuous on $\mathbb{R}^d$ : 
\begin{align}
\label{hypothèses1}
& \exists \lambda> 0, \text{  $\forall x$, $\xi \in \mathbb{R}^d$}, \quad  \lambda |\xi|^2 \leq \langle a(x)\xi, \xi\rangle, \\
& \label{hypothèses2}
a \in \left(L^{\infty}(\mathbb{R}^d)\right)^{d \times d},\\
& \label{hypothèses22}
a \in \left(\mathcal{C}^{0,\alpha}(\mathbb{R}^d)\right)^{d \times d}, \qquad \text{ for $\alpha \in ]0,1[$},
\end{align}
where $\mathcal{C}^{0,\alpha}(\mathbb{R}^d)$ is the space of functions which are both uniformly bounded and uniformly $\alpha$-H\"older continuous on $\mathbb{R}^d$, defined by 
$$ \mathcal{C}^{0,\alpha}(\mathbb{R}^d) = \left\{f \in L^{1}_{loc}(\mathbb{R}^d) \ \middle| \ \|f\|_{\mathcal{C}^{0,\alpha}(\mathbb{R}^d)} < \infty \right\},$$ 
where $\displaystyle \|f\|_{\mathcal{C}^{0,\alpha}(\mathbb{R}^d)} = \sup_{x\in \mathbb{R}^d} |f(x)| + \sup_{x,y\in \mathbb{R}^d, \ x\neq y}\frac{|f(x) - f(y)|}{|x-y|^{\alpha}}$.
Assumptions \eqref{hypothèses1} and \eqref{hypothèses2} are standard for the study of the homogenization problem~\eqref{equationepsilon_new}. Assumption \eqref{hypothèses22} is an additional assumption which is required in our approach to apply some results of elliptic regularity and to use pointwise estimates satisfied by the Green functions associated with equations in divergence form (see for instance \cite{avellaneda1987compactness,avellaneda1991lp} in which these assumptions are already made in the case of periodic coefficients). Since assumptions \eqref{hypothèses1} and \eqref{hypothèses2} imply that $a(./\varepsilon)$ is uniformly elliptic and uniformly bounded in $L^{\infty}(\Omega)$ with respect to~$\varepsilon$, the general homogenization theory of second order elliptic equations in divergence form \eqref{equationepsilon_new} (see \cite[Chapter 6, Chapter 13]{tartar2009general}) shows the existence of an extraction $\varphi$ such that $u_{\varphi(\varepsilon)}$ converges, strongly in $L^2(\Omega)$ and weakly in $H^1(\Omega)$, to a function $u^*$ when $\varepsilon$ converges to 0. The limit function is a solution to an homogenized problem, which is also a second order elliptic equation in divergence form,  
\begin{equation}
\label{equationhomog_new}
\left\{
\begin{array}{cc}
   -\operatorname{div}(a^* \nabla u^{*}) = f   & \text{on } \Omega, \\
    u^{*}(x) = 0 & \text{on } \partial \Omega,
\end{array}
\right.
\end{equation}
for some matrix-valued coefficient $a^*$ to be determined. In the periodic case, that is  \eqref{equationepsilon_new} when $a = a_{per}$ is periodic, it is well-known (see \cite{bensoussan2011asymptotic, jikov2012homogenization}) that the whole sequence $u^{\varepsilon}$ converges to $u^*$ and $(a_{per})^*$ is a constant matrix. The convergence in the $H^1(\Omega)$ norm can be obtained upon introducing a corrector $w_{per,q}$ defined for all $q$ in $\mathbb{R}^d$ as the periodic solution (unique up to the addition of a constant) to :
\begin{equation}
\label{problemeperiodique}
    -\operatorname{div}(a_{per}(\nabla w_{per,q} + q)) = 0 \quad \text{on } \mathbb{R}^d. 
\end{equation}
This corrector allows to both make explicit the homogenized coefficient  
\begin{equation}
\label{homogenizedcoeff_ap}   
    ((a_{per})^*)_{i,j} = \int_{Q} e_i^T a_{per}(y) \left( e_j + \nabla w_{per,e_j} \right) dy, 
\end{equation}
(where $(e_i)_{\{1,...,d\}}$ denotes the canonical basis of $\mathbb{R}^d$) and define an approximation 
\begin{equation}
\label{approximatesequence}
u^{\varepsilon,1} = u^*(.) + \varepsilon \displaystyle \sum_{i=1}^d \partial_{i}u^*(.) w_{per,e_i}(./ \varepsilon),
\end{equation}
such that $u^{\varepsilon,1} - u^{\varepsilon}$ strongly converges to $0$ in $H^1(\Omega)$ (see \cite{allaire1992homogenization} for more details). 

Our purpose here is to study the possibility to extend the above results to the setting of the non-periodic problem~\eqref{equationepsilon_new} when $a$ satisfies assumptions \eqref{hypothèses1}-\eqref{hypothèses2}-\eqref{hypothèses22} and is almost translation invariant at infinity. In our non-periodic case, the main difficulty is that, analogously to the periodic context, the behavior of $u^{\varepsilon}$ is closely linked to the existence of a corrector satisfying a property of strict sub-linearity at infinity, that is a solution, for $q\in \mathbb{R}^d$ fixed, to the corrector equation 
\begin{equation}
\label{correctorequationap}
    \left\{
\begin{array}{cc}
     -\operatorname{div}\left(a \left(\nabla w_q +q \right) \right)=0 & \text{on } \mathbb{R}^d,  \vspace{4pt} \\
    \displaystyle \lim_{|x| \to \infty} \dfrac{|w_q(x)|}{1+|x|} = 0. &
\end{array}
    \right.
\end{equation}
Here the corrector equation, formally obtained by a two-scale expansion (see again \cite{allaire1992homogenization} for the details), is defined on the whole space $\mathbb{R}^d$ and cannot be reduced to an equation posed on a bounded domain, as is the case in periodic context in particular. This prevents us from using classical techniques.

\subsection{Mathematical setting and preliminary approach}

\label{Subsect1_1}

In order to formalize our setting of non-periodic coefficients satisfying an almost translation invariance at infinity, we introduce, for every function $g\in L^1_{loc}(\mathbb{R}^d)$, the discrete gradient of $g$ denoted by $\delta g$. It is a vector-valued function defined by 
\begin{equation}
\label{def_gradient_discret}
    \delta g := \left(\delta_i g \right)_{i \in \{1,...,d\}} := \left( g(.+e_i) - g \right)_{i \in \{1,...d\}}.
\end{equation}
For every $p\in[1,+\infty[$, we also define the set $\mathbf{A}^p$ of locally integrable functions with a discrete gradient in $\left(L^p(\mathbb{R}^d) \right)^d$ : 
\begin{equation}
\label{circ_AP}
     \mathbf{A}^p = \left\{g \in L^1_{loc}(\mathbb{R}^d) \ \middle| \ \delta g \in \left(L^p(\mathbb{R}^d)\right)^d \right\}.
\end{equation}
Defined as above, the operator $\delta$ measures the deviation of a function $g$ from a $Q$-periodic function. This discrete gradient has been already used in the literature to study the behavior of solutions to elliptic equations posed in a periodic background, particularly to establish Liouville-type properties in \cite{moser1992liouville} and, more recently, to establish some regularity results satisfied by the solution to $-\operatorname{div}(a_{per} \nabla u ) = 0$ in \cite[Lemma 3.1]{armstrong2020large}. In our study, the class of coefficients $a$ we consider to model an asymptotically $Q$-periodic geometry is assumed to satisfy :
\begin{equation}
  \label{hypothèses3}
\exists p \in[1,+\infty[, \ \forall i,j \in \{1,...,d\},\quad \delta a_{i,j} \in \mathbf{A}^p.  
\end{equation}
Such an assumption ensures in a certain sense that $\delta a$ vanishes at infinity and, consequently, that the behavior of $a$ is close to that of a $Q$-periodic coefficient far from the origin (see Figure \ref{figf1ap} for examples in dimension $d=1$ and $d=2$). In addition, although we choose here to consider a specific case in which the coefficient is characterized by a "$\mathbb{Z}^d$-periodicity" at infinity, the results of the present paper can be easily adapted in a context of "$T$-periodicity" at infinity for any period $T$ (see Remark \ref{remarque_ap_period_T}). From a practical point of view, for a given medium modeled by a coefficient $a$, we are aware that the main difficulty is actually to identify the underlying period $T$ that characterizes the behavior of $a$ at infinity. A possible approach to overcome this difficulty consists in performing a spectrum analysis in order to identify the frequency of occurrence of the Dirac delta functions in the Fourier transform of $a$. 
We additionally note that, adapting the definition of the discrete gradient \eqref{def_gradient_discret}, similar questions to those addressed in the present article may be studied for random coefficients that are stationary at infinity in a sense that has to be made precise. We hope to return to this alternative setting in a future publication and we refer to \cite{goudey} for more details. 


\begin{figure}[h!]
\centering
\includegraphics[scale=0.26]{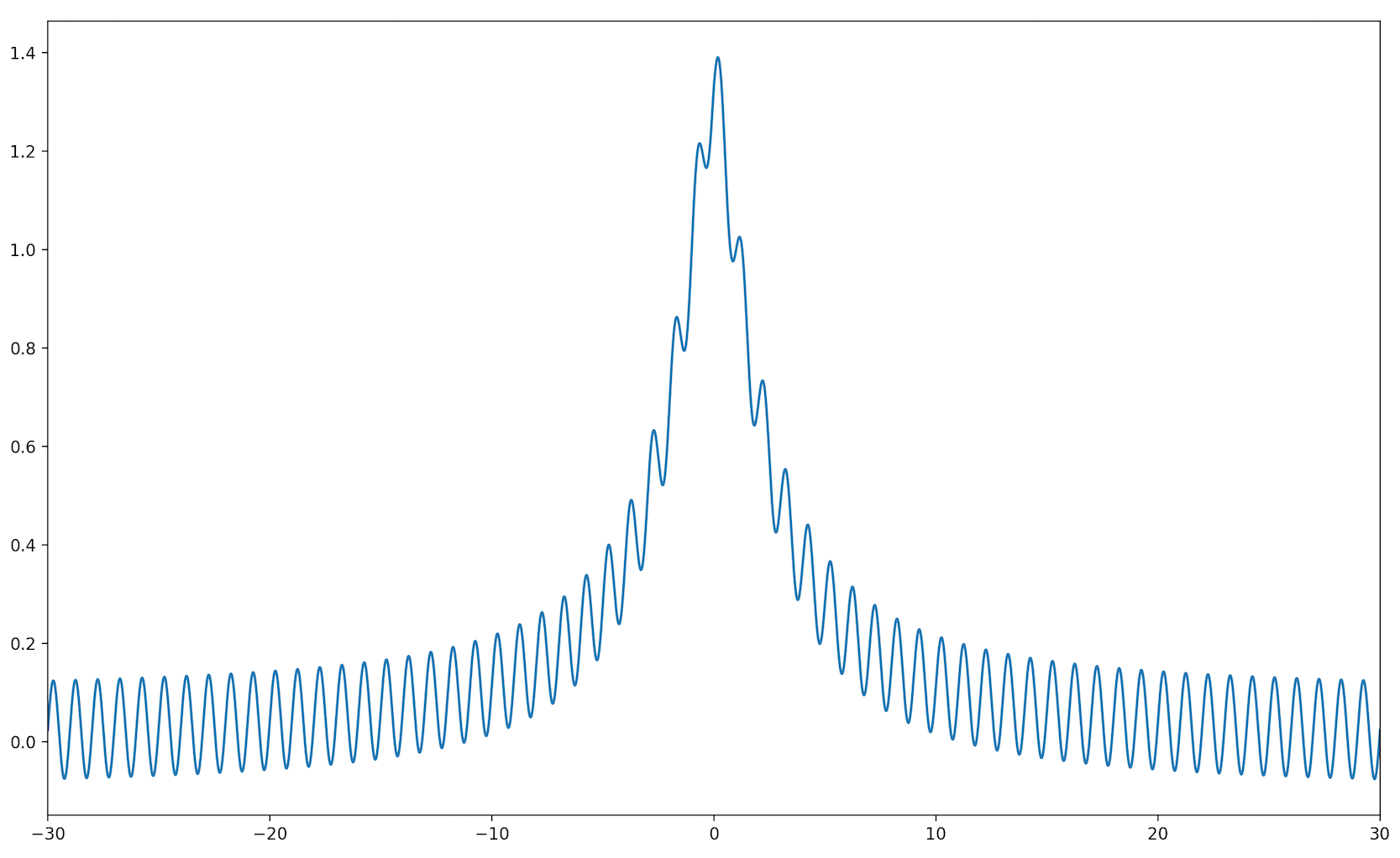}
\includegraphics[scale=0.24]{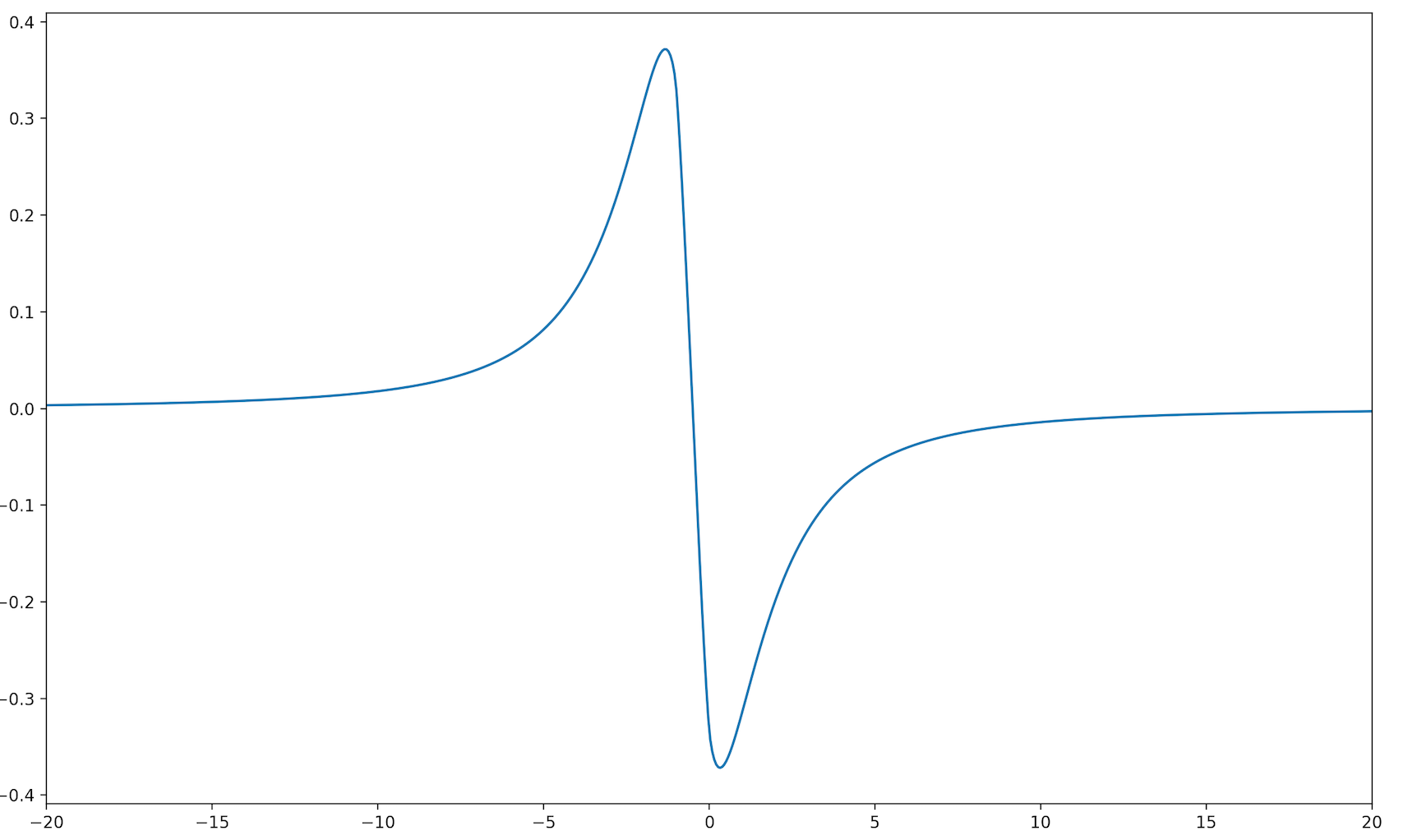}
\includegraphics[scale=0.265]{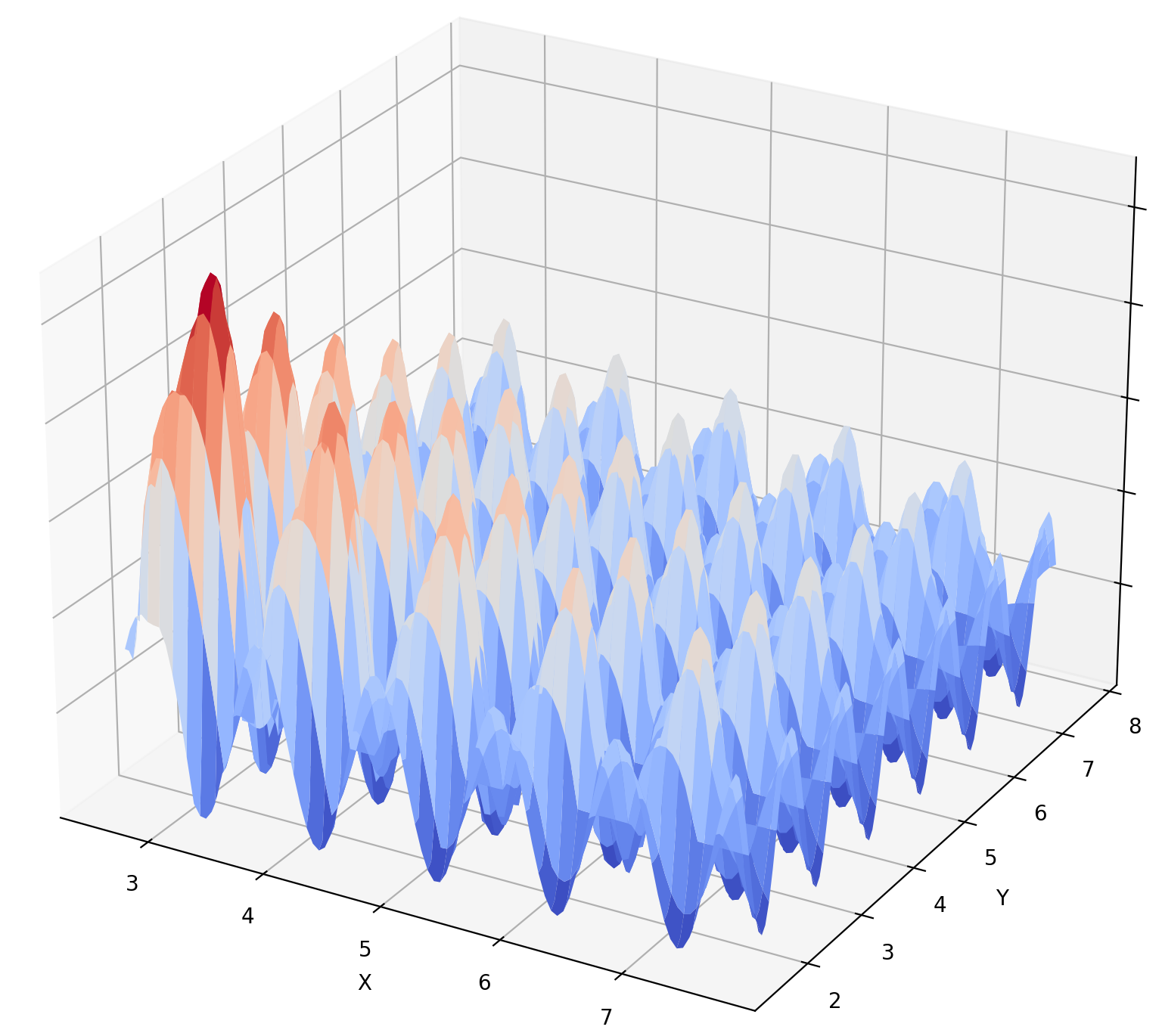}
\includegraphics[scale=0.265]{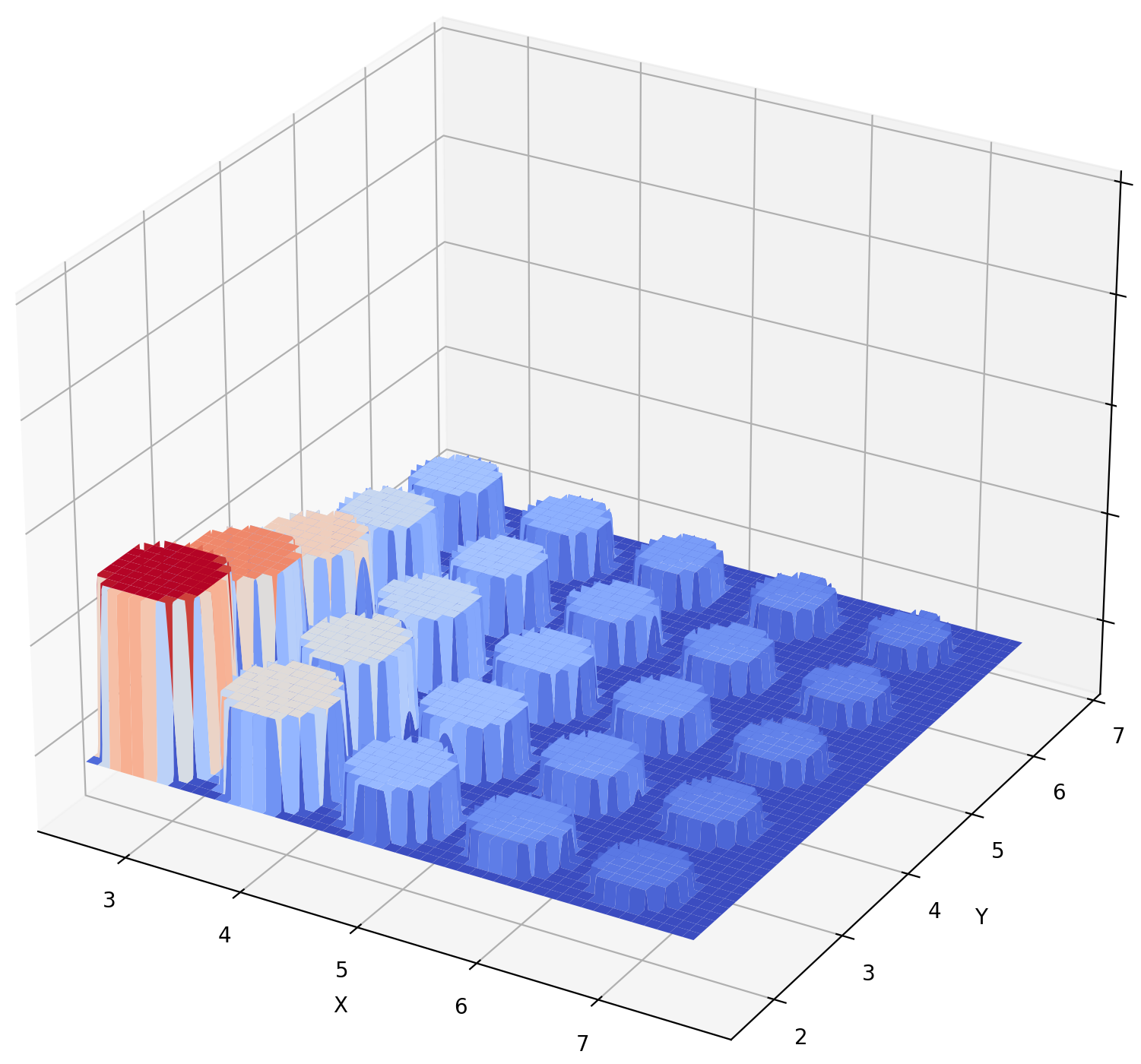}
\caption
{Examples of coefficients $a$ satisfying assumption \eqref{hypothèses3} in dimension $d=1$ (left: $a(x)$ ; right: $\delta a(x)$) and $d=2$ (left: $a(x,y)$ ; right: $|\delta_2 a(x,y)|$).}
\label{figf1ap}
\end{figure}



A preliminary approach to understand the behavior of $u^{\varepsilon}$ and to prove existence of an adapted corrector in our context is to consider a continuous version of \eqref{hypothèses3} for which
\begin{equation}
\label{hypothese_continue_ap}
   \exists p \in [1,+\infty[,  \forall i,j \in \{1,...,d\}, \quad  \nabla a_{i,j} \in \left(L^p(\mathbb{R}^d)\right)^d.
\end{equation}
Since 
$ \displaystyle \delta_{k} a_{i,j} = \int_{0}^{1} \nabla a_{i,j}(. + t e_k).e_k \ dt$  
for every $i,j,k$ in $\{1,...,d\}$, the H\"older inequality actually allows to show that assumption \eqref{hypothese_continue_ap} is stronger than \eqref{hypothèses3}. Such an assumption implies the convergence to $0$ of $\nabla a$ at infinity, that is to say, it models a medium close to an homogeneous medium at infinity. For this particular setting, we can distinguish two cases that depend on the value of the ratio $\dfrac{p}{d}$ : 

1. \underline{The case $p<d$} : If we denote by $p^* = \dfrac{pd}{d-p}$ the Sobolev exponent associated with $p$, a consequence of the Gagliardo-Nirenberg-Sobolev inequality (see for instance \cite[Section 5.6.1]{evans10}) gives the existence of a constant $c\in \mathbb{R}^{d\times d}$ such that $a-c \in \left(L^{p^*}(\mathbb{R}^d)\right)^{d\times d}$ and the following inequality holds~:  
    $$\|a-c\|_{L^{p^*}(\mathbb{R}^d)} \leq M \|\nabla a\|_{L^p(\mathbb{R}^d)},$$
    where $M$ is a constant independent of $a$. It is therefore possible to split the coefficient $a$ as the sum of a constant and a "local" perturbation. Precisely,
    \begin{equation}
    \label{Decompistion_homog_coeff_ap}
      a = c + a-c = c + \Tilde{a},  
    \end{equation}
    where $\Tilde{a}$ belongs to $\left(L^{p^*}(\mathbb{R}^d)\right)^{d\times d}$ and the setting is that of a periodic geometry (actually an homogeneous background described by the constant $c$) perturbed by a defect of $L^{p^*}(\mathbb{R}^d)$. Consequently, if $a$ satisfies \eqref{hypothèses1}, \eqref{hypothèses2} and \eqref{hypothèses22}, our problem is equivalent to a perturbed periodic problem introduced in \cite{blanc2018precised, blanc2018correctors}. In this case the existence of an adapted corrector is established, the gradient of which shares the same structure as the coefficient $a$ : it is a gradient of a periodic function perturbed by a function in $L^{p*}(\mathbb{R}^d)$. The homogenization problem can also be addressed : the whole sequence $u^{\varepsilon}$ converges to $u^*$ and the coefficient $a^*$ can be made explicit. 
    
2. \underline{The case $p\geq d$} : This case is characterized by a slow decay of $\nabla a$ at infinity. Contrary to the case $p<d$, we can show here the existence of coefficients $a$ satisfying \eqref{hypothese_continue_ap} and such that it is impossible to split $a$ as in \eqref{Decompistion_homog_coeff_ap}, that is to characterize our particular geometry as a periodic (let alone homogeneous) background perturbed by a local defect. A typical counter example, which will be detailed in Section~\ref{SectionAP_3}, is given by a coefficient $a$ which oscillates very slowly at infinity ($a = 2+ \sin(\ln(1+|x|)$ in dimension $d=1$ for example). Far from the origin, such a coefficient locally looks as constant but does not converge at infinity. In this particular case, we can show that the sequence $u^{\varepsilon}$ itself does not converge. We have only the convergence up to an extraction and the sequence admits an infinite number of converging subsequences. 

In our discrete case, when $\delta a$ belongs to  $\left(L^p(\mathbb{R}^d)\right)^d$, we therefore expect a similar phenomenon~: the convergence of $u^{\varepsilon}$ should depend on the value of the ratio $\frac{p}{d}$, that is, on the type of decay at infinity of the discrete gradient $\delta a$.

\subsection{Main results}

In the sequel, we denote by $B_R$ the ball of radius $R >0$ centered at the origin,  by $B_R(x)$ the ball of radius $R >0$ and center $x \in \mathbb{R}^d$ and by $Q+x :=  \displaystyle \prod_{i=1}^d ]x_i,x_i+1[$, the unit cell translated by a vector $x\in \mathbb{R}^d$. We also denote by $|A|$ the volume of any Borel subset $A\subset \mathbb{R}^d$. In addition, for a normed vector space $(X,\|.\|_X)$ and a matrix-valued function $f \in X^n$, $n\in \mathbb{N}$, we use the notation $\|f\|_{X} \equiv \|f \|_{X^n}$ when the context is clear. 

Assuming that the coefficient $a$ satisfies \eqref{hypothèses1}, \eqref{hypothèses2}, \eqref{hypothèses22} and \eqref{hypothèses3}, the main questions that we examine in this paper are the following : does the whole sequence $u^{\varepsilon}$ converges to $u^*$ (and not only a sub-sequence) ? If it is the case, can the diffusion coefficient $a^*$ of the homogenized equation be made explicit ? Can we establish the existence of a strictly sub-linear corrector solution to \eqref{correctorequationap}~? 

\subsubsection{The case $p<d$}

When $\delta a \in \left(L^p(\mathbb{R}^d)\right)^d$ for $p<d$, our approach is an adaptation of that of the continuous case which we have just sketched above in Section \ref{Subsect1_1} : we show that the coefficient $a$ actually models a periodic geometry perturbed by a local defect which, up to a local averaging, belongs to $L^{p^*}(\mathbb{R}^d)$, for $p^* = \frac{pd}{d-p}$ the Sobolev exponent associated with $p$. To this end, we introduce an operator $\mathcal{M}$ to describe the local averages of a function $f \in L^1_{loc}(\mathbb{R}^d)$ and defined by : 
$$\mathcal{M}(f)(z) = \displaystyle \int_{Q+z} f(x)dx.$$ 
We also introduce the following two functional spaces : 
\begin{equation}
\label{defEp}
\mathcal{E}^p = \left\{f \in L^1_{loc}(\mathbb{R}^d) \ \middle|  \ \mathcal{M}\left(\left|f\right| \right) \in L^{p^*}(\mathbb{R}^d) \right\},
\end{equation}
\begin{equation}
\label{defAp}
    \mathcal{A}^p = \left \{ f \in L^1_{loc}(\mathbb{R}^d) \ \middle| \ \mathcal{M}\left(\left|f\right| \right) \in L^{p^*}(\mathbb{R}^d) \ \text{and} \ \delta f \in (L^p(\mathbb{R}^d))^d\right\},
\end{equation}
equipped with the norms : 
\begin{equation}
\label{normeE_p}
    \left\|f\right\|_{\mathcal{E}^p} = \| \mathcal{M}\left(\left|f\right| \right) \|_{L^{p^*}(\mathbb{R}^d)},
\end{equation}
\begin{equation}
\label{normAp}
    \|f\|_{\mathcal{A}^p} = \|\mathcal{M}\left(\left|f\right| \right)\|_{L^{p^*}(\mathbb{R}^d)} + \|\delta f\|_{(L^{p}(\mathbb{R}^d))^d}.
\end{equation}
We particularly note that the functions in $\mathcal{E}^p$ or $\mathcal{A}^p$ are characterized by the integrability of the local averaging operator $\mathcal{M}$ applied to their absolute value. Our main result regarding the functions of $\mathbf{A}^p$ when $p<d$ is given in the following proposition : 
\begin{prop}
\label{decomposition_fonctions_A_p}
Assume $p<d$. Let $f \in  \mathbf{A}^p$, then there exists a unique periodic function $f_{per}$ such that $f - f_{per} \in \mathcal{E}^p$. In addition, there exists a constant $C>0$ independent of $f$ such that : \begin{equation}
\label{GNS_discret}
    \left\|f - f_{per}\right\|_{\mathcal{E}^p} \leq C \|\delta f \|_{L^p(\mathbb{R}^d)}.
\end{equation}
\end{prop}

Proposition \ref{decomposition_fonctions_A_p} is a discrete adaptation of the Gagliardo-Nirenberg-Sobolev inequality and Section \ref{SectionAp_1} is devoted to its proof. It ensures that every function $f \in \mathbf{A}^p$ is the sum of a periodic function and a "perturbation" of $\mathcal{A}^p$. This result therefore allows to identify a periodic background perturbed by a local defect. Precisely, the coefficient $a$ is of the form $a=a_{per} + \Tilde{a}$, that is, it is the sum of a periodic coefficient $a_{per}$, that will be made explicit in this paper (see Proposition \ref{expliciteperiodiqueap}), and a perturbation denoted by $\Tilde{a}$. To address the homogenization problem in this particular perturbed case, we establish the following result : 
\begin{theorem}
\label{Correcteur_A_P}
Assume $a$ satisfies \eqref{hypothèses1}, \eqref{hypothèses2}, \eqref{hypothèses22} and \eqref{hypothèses3} for $1<p<d$. We denote by $a_{per}$ the unique periodic coefficient given by Proposition \ref{decomposition_fonctions_A_p} such that $\Tilde{a} := a-a_{per} \in \left(\mathcal{A}^p\right)^{d \times d}$. Let $q \in \mathbb{R}^d$. If $w_{per,q}$ is the periodic solution, unique up to an additive constant, to  
$$-\operatorname{div}\left(a_{per} \left( \nabla w_{per,q} + q \right)\right) = 0 \quad \text{on } \mathbb{R}^d.$$
Then, there exists $\Tilde{w}_q\in L^1_{loc}(\mathbb{R}^d)$ solution to 
\begin{equation}
\label{eq_correcteur_A_p}
\left\{
\begin{array}{cc}
     - \operatorname{div}(a (\nabla w_{per,q} + \nabla \Tilde{w}_q +q)) = 0  &   \text{on } \mathbb{R}^d, \vspace{4pt} \\
   \displaystyle \lim_{|x| \to \infty} \dfrac{|\Tilde{w}_q(x)|}{1+ |x|} = 0, & 
\end{array}
\right.
\end{equation}
such that $\nabla \Tilde{w}_q \in \left(\mathcal{A}^p\cap \mathcal{C}^{0, \alpha}(\mathbb{R}^d)\right)^d$. Such a solution $\Tilde{w}_q$ is unique up to an additive constant. 

In addition, the sequence $u^{\varepsilon}$ of solutions to \eqref{equationepsilon_new} converges, strongly in $L^2(\Omega)$ and weakly in $H^1(\Omega)$ to $u^*$ solution to : 
\begin{equation}
\label{equationhomog_new_per}
\left\{
\begin{array}{cc}
   -\operatorname{div}(a_{per}^* \nabla u^{*}) = f   & \text{on } \Omega, \\
    u^{*} = 0 & \text{on } \partial \Omega,
\end{array}
\right.
\end{equation}
where $a_{per}^*$ is defined by \eqref{homogenizedcoeff_ap}.
\end{theorem}

Theorem \ref{Correcteur_A_P} states the existence of an adapted corrector and shows the convergence of $u^{\varepsilon}$ to an homogenized limit $u^*$. Similarly to the case of a periodic geometry perturbed by local defects of $L^r(\mathbb{R}^d)$ studied in \cite{blanc2018precised, blanc2018correctors, blanc2015local, blanc2012possible}, the gradient of our adapted corrector shares the same structure as the coefficient $a$~: it is the sum of a periodic function and a perturbation in $\mathcal{A}^p$. The perturbations of $\mathcal{A}^p$ does not impact the homogenized solution since the homogenized coefficient is the same as in the periodic problem \eqref{equationepsilon_new} when $a = a_{per}$. 
In the sequel, we establish Theorem \ref{Correcteur_A_P} in the case where $p<d$ and $p \neq 1$, the case $p=1$ being specific. Indeed, as we shall see in Section \ref{SectionAP_2}, our approach is based on the study of the general diffusion equation $-\operatorname{div}(a \nabla u) = \operatorname{div}(f)$ on $\mathbb{R}^d$, when $a$ belongs to $\left(L^2_{per}(\mathbb{R}^d) + \mathcal{A}^p\right)^{d\times d}$ (where $L^2_{per}(\mathbb{R}^d)$ denotes the space of locally $L^2$ periodic functions) and $ f$ belongs to $\left(\mathcal{A}^p\right)^d$. A key element of this study is a continuity result from $\left(\mathcal{A}^p\right)^d$ to $\left(\mathcal{A}^p\right)^d$ established in Proposition \ref{existence_cas_periodique_AP} (for periodic coefficients) and in Lemma \ref{estimee_A_priori_Ap} (in the general case) satisfied by the operator $-\nabla\left(-\operatorname{div}a\nabla\right)^{-1}\operatorname{div}$,  which is false when $p=1$ (see Remark \ref{remark_counter_p1} for a counter-example) and, in this case, we are not able to show the existence of $\Tilde{w}_q$ such that $\nabla \Tilde{w}_q \in \left(\mathcal{A}^1\right)^d$. However, since the coefficient $a$ belongs to $\left(L^{\infty}(\mathbb{R}^d)\right)^{d \times d}$, assumption~\eqref{hypothèses3} for $p=1$ implies that the same assumption is true for every $p>1$ and Theorem~\ref{Correcteur_A_P} gives the existence of an adapted corrector such that $\nabla \Tilde{w}_q$ belongs to $\left( \mathcal{A}^p \right)^d$ for every $p>1$. 

The existence of an adapted corrector is actually key to establish an homogenization theory in the context of problem \eqref{equationhomog_new}. We use it first in the proof of Theorem \ref{Correcteur_A_P} in order to identify the homogenized equation \eqref{equationhomog_new_per}. Moreover, if we define an approximation $\displaystyle u^{\varepsilon,1} = u^* + \varepsilon \sum_{i=1}^d \partial_i u^* w_{e_i}(./\varepsilon) $ such as \eqref{approximatesequence} in the periodic case but using our adapted corrector, it is also possible to describe the behavior of $\nabla u^{\varepsilon}$ in several topologies exactly as in the periodic context. If we denote $R^{\varepsilon} := u^{\varepsilon} - u^{\varepsilon,1}$, the results established in the present paper ensure that our setting is covered by the work of \cite{blanc2018precised} which studies problem \eqref{equationepsilon_new} under general assumptions (the existence of a corrector strictly sublinear at infinity in particular) and shows the convergence of $R^{\varepsilon}$ to 0 for the topology of $W^{1,r}$ when $r \geq 2$. Some properties related to the strict sublinearity of our corrector therefore allow to make precise the convergence rate of $\nabla R^{\varepsilon}$ (see estimate \eqref{convergence_rate_ap}). 

We also note that assumption \eqref{hypothèses22} regarding the H\"older continuity of the coefficient together with~\eqref{hypothèses3} implies that $\Tilde{a}$ belongs to $\left(L^q(\mathbb{R}^d)\right)^{d\times d}$ for a given exponent $p^*<q$ as a consequence of Proposition \ref{prop_lebesgue_inclusion} established in Section~\ref{SectionAp_1}. It follows that \cite{blanc2018correctors, blanc2015local, blanc2012possible} actually cover our setting and show the existence of a corrector of the form $w = w_{per} + \Tilde{w}$ where $\Tilde{w}$ a is solution to \eqref{eq_correcteur_A_p} such that $\nabla \Tilde{w} \in \left(L^q(\mathbb{R}^d)\right)^d$ for this particular exponent $q$. However, the results of Theorem \ref{Correcteur_A_P} are stronger in our approach : it ensures that the perturbed part of our corrector has a gradient in $\left(\mathcal{A}^p\right)^d$ and, since $p^* < q$, it provides better properties regarding its integrability at infinity. It is indeed shown in Section \ref{Subsection_homog_ap} that the theoretical convergence rates of $\nabla R^{\varepsilon}$ are improved if we assume $\Tilde{a} \in \left(\mathcal{A}^p\right)^{d \times d}$ rather than only $\Tilde{a} \in \left(L^q(\mathbb{R}^d)\right)^{d\times d}$. Besides proving the stronger results of Theorem \ref{Correcteur_A_P}, one contribution of the present study is also to put in place a whole methodological machinery for functions with integrable discrete gradients that allows to obtain homogenization results, similarly but \emph{independently} from the proofs and the arguments conducted in the context of $L^q$ functions. Our aim is, in particular, to highlight the fact that the methodology employed in \cite{blanc2018correctors, blanc2015local, blanc2012possible} only requires to know the global behavior of $a$ at infinity, and the non-local control of the averages of $\Tilde{a}$ (in contrast to the assumptions of $L^q$ integrability in \cite{blanc2018correctors, blanc2015local, blanc2012possible}) given by Proposition \ref{decomposition_fonctions_A_p} is sufficient to perform the homogenization of problem~\eqref{equationepsilon_new}. 
Although we have not pursued in this direction, we also believe that the so-called large-scale regularity results established in \cite{gloria2020regularity} could possibly be adapted to our setting in order to obtain homogenization results for problem~\eqref{equationepsilon_new}, similar to those of Theorem \ref{Correcteur_A_P} but without assumption of H\"older regularity satisfied by~$a$. 


\subsubsection{The case $p>d$}

When $p\geq d$, we show that the homogenization of problem \eqref{equationepsilon_new} is not always possible. More  precisely, we exhibit a couple of sequences $u^{\varepsilon}$ that have subsequential limits. Our two counter examples slowly oscillate at infinity (see Figure \ref{figf2ap} for examples).\\


Our article is organized as follows. In Section \ref{SectionAp_1}, we study the properties of the space $\mathbf{A}^p$ in the case $p<d$ and we establish the discrete version of the Gagliardo-Nirenberg-Sobolev inequality stated in Proposition \ref{decomposition_fonctions_A_p}. In section \ref{SectionAP_2}, we prove Theorem \ref{Correcteur_A_P}. Finally, in Section~\ref{SectionAP_3}, we study the homogenization problem \eqref{equationepsilon_new} in the case $p\geq d$.

\section{Properties of the functional space $ \mathbf{A}^p$, the case $p<d$}

\label{SectionAp_1}

Throughout this section, we assume that $d\geq 2$ and that $p\in [1,d[$. We study the properties of the space $\mathbf{A}^p$ defined by \eqref{circ_AP}. The main idea is, of course, to see the operator $\delta$ as a discrete gradient operator and to draw an analogy between this discrete gradient and the usual continuous gradient~$\nabla$. We show that the functions of $\mathbf{A}^p$ satisfies several properties similar to those satisfied by the functions $f$ such that $\nabla f \in \left(L^p(\mathbb{R}^d)\right)^d$ and we establish a discrete variant of the Gagliardo-Nirenberg-Sobolev inequality proving that the functions $f \in  \mathbf{A}^p$ satisfy, up to the addition of a periodic function, some properties of integrability. More precisely, we prove that such a function $f$ can be split as the the sum of a periodic function $f_{per}$ and a function $\Tilde{f}\in \mathcal{A}^p$, which belongs, up to a local averaging (made precise in formula \eqref{normeE_p} above), to the particular Lebesgue space $L^{p^*}(\mathbb{R}^d)$.


\subsection{Properties of $\mathcal{E}^p$ and $\mathcal{A}^p$}

To start with, we need to introduce several properties satisfied by the spaces $\mathcal{E}^p$ and $\mathcal{A}^p$, respectively defined in \eqref{defEp} and \eqref{defAp}, and we establish some asymptotic properties regarding the average value and the strict sub-linearity of the functions belonging to $\mathcal{A}^p$.


We first claim that the spaces $\mathcal{A}^p$ and $\mathcal{E}^p$ respectively equipped with the norms \eqref{normeE_p} and \eqref{normAp} are two Banach spaces. The proof is given in \cite{goudey} and consists in considering $\mathcal{E}^p$ as a particular subset of $L^{p^*}\left(\mathbb{R}^d, L^1_{loc}(\mathbb{R}^d)\right)$, we skip it for the sake of brevity. A classical property satisfied by such a subspace of $L^{p^*}\left(\mathbb{R}^d, L^1_{loc}(\mathbb{R}^d)\right)$ and that will be useful in the sequel is given in the following proposition.

\begin{prop}
\label{sous-suite}
Let $\left(f_n\right)_{n \in \mathbb{N}}$ be a sequence of functions of $\mathcal{E}^p$ that converges to $f$ in $\mathcal{E}^p$. Then, there exists a sub-sequence $f_{\varphi(n)}$ that converges to $f$ in $L^{1}_{loc}(\mathbb{R}^d)$.   
\end{prop}





Our next aim is to study the asymptotic behavior of sufficiently regular functions of $\mathcal{A}^p$. We begin by proving that uniformly continuous functions in $\mathcal{A}^p$ vanish at infinity. 

\begin{prop}
\label{prop_comportement_asymptotique_Ep}
Let $f$ be an uniformly continuous function on $\mathbb{R}^d$ such that $\displaystyle  \mathcal{M}\left(\left|f\right| \right)(x)$ vanishes at infinity.
Then $\displaystyle \lim_{|x| \rightarrow \infty} f(x) = 0$.
\end{prop}

\begin{proof}
We argue by contradiction and assume that $f$ does not converge to 0 at infinity. Then, there exists $\varepsilon >0$ such that for every $R>0$, there exists $x_R \in \mathbb{R}^d$ with $|x_R|>R$ and $|f(x_R)| > \varepsilon$. $f$ being uniformly continuous, there exists $\delta>0$ such that for every $R>0$ and for every $y\in B_{\delta}(x_R)$, we have $\displaystyle |f(y)| > \frac{\varepsilon}{2}$. Since $\displaystyle \lim_{|x| \rightarrow \infty} \mathcal{M}\left(\left|f\right| \right)(x) = 0$, there exists $R_0>0$ such that
$|x|>R_0$ implies $\displaystyle \mathcal{M}\left(\left|f\right| \right)(x)< \frac{|B_{\delta}|}{2}\ \varepsilon$. On the other hand, $\displaystyle \mathcal{M}\left(\left|f\right| \right)(x_{R_0}) \geq \int_{B_{\delta}(x_{R_0})} |f(z)| dz \geq \frac{|B_{\delta}| }{2} \ \varepsilon$.
Since $|x_{R_0}|>R_0$, we have a contradiction. 
\end{proof}


From the previous proposition, we deduce the following corollary. 

\begin{corol}
\label{comportement_asymptotique_AP}
Let $f \in \mathcal{E}^p \cap \mathcal{C}^{0, \alpha}(\mathbb{R}^d)$ for $\alpha \in ]0,1[$, then $\displaystyle \lim_{|x| \rightarrow \infty} f(x) = 0$.
\end{corol}

\begin{proof}
Since $f \in \mathcal{C}^{0,\alpha}(\mathbb{R}^d)$, the function $\mathcal{M}(|f|)$ also belongs to $C^{0,\alpha}(\mathbb{R}^d)$ and we have 
$$ \displaystyle \left|\mathcal{M}(|f|)(x) - \mathcal{M}(|f|)(z)\right| \leq \|f\|_{C^{0,\alpha}(\mathbb{R}^d)}|x-z|^{\alpha},$$ for every $x,z \in \mathbb{R}^d$. In addition, since $\mathcal{M}(|f|)\in L^{p^*}(\mathbb{R}^d)$, it follows that $\displaystyle \lim_{|x| \rightarrow \infty} \mathcal{M}\left(\left|f\right| \right)(x) = 0$ and we conclude using Proposition~\ref{prop_comportement_asymptotique_Ep}.
\end{proof}

The next proposition regards the average value of the functions in $\mathcal{A}^p$ and is actually key for the homogenization of problem \eqref{equationepsilon_new}. Indeed, as stated in Corollary \ref{convergenceLinfinistar_AP}, it implies a weak convergence to 0 of the sequence $(|f(./\varepsilon)|)_{\varepsilon>0}$, which means in a certain sense that a perturbation of $\mathcal{A}^p$ has no macroscopic impact on the ambient background. This property shall be particularly useful to identify the homogenized coefficient $a^*$ associated with problem \eqref{equationepsilon_new}. 


\begin{prop}
\label{proposition_moyenne_Ap}
Let $f\in \mathcal{A}^p$. Then, $\displaystyle\lim_{R \rightarrow \infty}  \frac{1}{|B_R|} \int_{B_R(x_0)} |f| = 0$ for every $x_0 \in \mathbb{R}^d$ and we have the following convergence rate :
\begin{equation}
\label{convergence_rate_delta}
    \frac{1}{|B_R|} \int_{B_R(x_0)} |f| \leq  \frac{C}{R^{d/p^*}},
\end{equation}
where $C>0$ is independent of $R$ and $x_0$. 
\end{prop}

\begin{proof}
Let $R>0$ and $x_0 \in \mathbb{R}^d$. Since $|Q|=1$, we have : 
$$ \int_{B_R(x_0)} |f|(x)dx = \int_{Q} \int_{B_R(x_0)} |f|(x) dxdz = \int_{Q} \int_{B_R(x_0)-z} |f|(x+z) dx dz.$$
In addition, for $R$ large enough and for every $z \in Q$, we have $B_R(x_0)-z\subset B_{2R}(x_0)$. Using the Fubini theorem, we therefore obtain 
\begin{align*}
   \int_{Q} \int_{B_R(x_0)-z} |f|(x+z) dx dz & \leq \int_{Q} \int_{B_{2R}(x_0)} |f|(x+z) dx dz = \int_{B_{2R}(x_0)} \mathcal{M}(|f|)(x) dx. 
\end{align*}
The H\"older inequality next gives
\begin{align*}
   \frac{1}{|B_R|} \int_{B_{2R}(x_0)} \mathcal{M}(|f|) & \leq \frac{|B_{2R}|^{1/(p^*)'}}{|B_R|} \|f\|_{\mathcal{E}^p}  = C(d) \frac{1}{R^{d/p^*}}\|f\|_{\mathcal{E}^p},
\end{align*}
where $(p^*)'$ is the conjugate exponent of $p^*$ and $C(d)>0$ depends only on the ambient dimension $d$. We obtain \eqref{convergence_rate_delta}.  
\end{proof}

\begin{corol}
\label{convergenceLinfinistar_AP}
Let $u \in \mathcal{A}^p \cap L^{\infty}(\mathbb{R}^d)$, then the sequence $\left(|u(./\varepsilon)|\right)_{\varepsilon>0}$ converges to 0 for the weak-* topology of $L^{\infty}(\mathbb{R}^d)$ as $\varepsilon$ vanishes.
\end{corol}

\begin{proof}
We fix $R>0$ and we begin by considering $\varphi = 1_{B_R(x_0)}$ for $x_0 \in \mathbb{R}^d$. For every $\varepsilon>0$, we have :
\begin{align*}
\left|\int_{\mathbb{R}^d} |u(x/\varepsilon)|\varphi(x) dx\right| = \int_{B_R(x_0)} \left| u(x/\varepsilon) \right| dx & = \varepsilon^d   \int_{B_{R/\varepsilon}(x_0/\varepsilon)}\left| u(y) \right|dy.
\end{align*}
We therefore use Proposition \ref{proposition_moyenne_Ap} and we obtain $\displaystyle \left|\int_{\mathbb{R}^d} |u(x/\varepsilon)|\varphi(x) dx\right| \underset{\varepsilon\to 0}{\longrightarrow} 0$.
We conclude using the density of simple functions in $L^1(\mathbb{R}^d)$.
\end{proof}

We next show that every function with a gradient in $\left(\mathcal{A}^p \cap L^{\infty}(\mathbb{R}^d)\right)^d$ is strictly sub-linear at infinity. 

\begin{prop}
\label{Sous_linéarite_AP}
Let $\frac{d}{2}< p<d$ and $u\in L^1_{loc}(\mathbb{R}^d)$ such that $\nabla u \in \left(\mathcal{E}^p \cap L^{\infty}(\mathbb{R}^d)\right)^d$. Then $u$ is strictly sub-linear at infinity.
More precisely, we have
\begin{equation}
\label{estimme_ss_linearite_ap}
    \dfrac{|u(x)|}{1+|x|} = O\left(\dfrac{1}{|x|^{d/p^*}}\right).
\end{equation}

\end{prop}

\begin{proof}
Let $x\in \mathbb{R}^d$ such that $x \neq 0$. We denote by $r = |x|$ and we have 
\begin{align}
    |u(x) - u(0)|  &  \displaystyle \leq  \int_{B_r(x)}  \dfrac{|\nabla u(w)|}{|x-w|^{d-1}}dw + \int_{B_r} \dfrac{|\nabla u(w)|}{|w|^{d-1}} dw \\
    \label{sous_lin_morrey_ap}
    & = \displaystyle \int_{Q} \left( \int_{B_r(x-z)} \dfrac{|\nabla u(w+z)|}{|x-w-z|^{d-1}}dw + \int_{B_r(-z)} \dfrac{|\nabla u(w+z)|}{|w+z|^{d-1}} dw\right)dz.
\end{align}
The first inequality above is established for instance in \cite[p.266]{evans10} in the proof of the Morrey's inequality (\cite[Theorem 4 p.266]{evans10}).
We next remark that $B_r(x-z) \subset B_{2r}(x)$ for every $r$ sufficiently large and every $z\in Q$. For $z \in Q$, since $|z|\leq \sqrt{d}$, there also exists a constant $C>0$ independent of $x$ and $z$ such that $\dfrac{1}{|x-w-z|^{d-1}} \leq C \dfrac{1}{|x-w|^{d-1}}$ for every $w \in B_{2r}(x)\setminus{B_{2\sqrt{d}}(x)}$. We deduce 
$$\int_{B_r(x-z)\setminus{B_{2\sqrt{d}}(x)}} \dfrac{|\nabla u(w+z)|}{|x-w-z|^{d-1}}dw \leq C \int_{B_{2r}(x)} \dfrac{|\nabla u(w+z)|}{|x-w|^{d-1}}dw.$$
Since $p>\frac{d}{2}$, we also have $p^*>d$ and $\dfrac{(d-1)p^*}{p^{*}-1}<d$. Using the Fubini theorem and the H\"older inequality, it therefore follows 
\begin{align*}
    \int_{Q} \int_{B_r(x-z)\setminus{B_{2\sqrt{d}}(x)}} \dfrac{|\nabla u(w+z)|}{|x-w-z|^{d-1}}dw \hspace{0.05cm} dz  & \leq C \left(\int_{B_{2r}(x)} \dfrac{1}{|x-w|^{(d-1)\frac{p^*}{p^{*}-1}}} dw \right)^{\frac{p^*-1}{p^*}} \|\nabla u\|_{\mathcal{E}^p} \\
    &= C_1  r^{1 - \frac{d}{p^*}} \|\nabla u\|_{\mathcal{E}^p},
\end{align*}
where $C_1$ is a constant that depends only on $p^*$ and the dimension $d$. We also have
\begin{align}
    \int_{Q} \int_{B_{2\sqrt{d}}(x)} \dfrac{|\nabla u(w+z)|}{|x-w-z|^{d-1}}dw \hspace{0.05cm} dz & \leq \|\nabla u\|_{L^{\infty}(\mathbb{R}^d)} \int_{Q} \int_{B_{2\sqrt{d}}} \dfrac{1}{|w+z|^{d-1}}dw \hspace{0.05cm} dz\\
    &\leq \|\nabla u\|_{L^{\infty}(\mathbb{R}^d)} \int_{B_{3\sqrt{d}}} \dfrac{1}{|w|^{d-1}}dw,
\end{align}
and, if we denote $C_2 = \max\left(C_1, \displaystyle \int_{B_{3\sqrt{d}}} \dfrac{1}{|w|^{d-1}}dw\right)$, we obtain 
\begin{equation}
\label{estime_sous_lin_ap_1}
    \int_{Q} \int_{B_r(x-z)} \dfrac{|\nabla u(w+z)|}{|x-w-z|^{d-1}}dw \hspace*{0.05cm}  dz  \leq C_2 \left( r^{1 - \frac{d}{p^*}} \|\nabla u\|_{\mathcal{E}^p} + \|\nabla u \|_{L^{\infty}(\mathbb{R}^d)}\right).
\end{equation}
We can similarly show that 
\begin{equation}
\label{estime_sous_lin_ap_2}
    \int_{Q} \int_{B_r(-z)} \dfrac{|\nabla u(w+z)|}{|w+z|^{d-1}}dw \hspace*{0.05cm} dz   \leq C_2 \left( r^{1 - \frac{d}{p^*}} \|\nabla u\|_{\mathcal{E}^p} + \|\nabla u \|_{L^{\infty}(\mathbb{R}^d)}\right).
\end{equation}
Since $\dfrac{d}{p^*}<1$, estimates \eqref{sous_lin_morrey_ap}, \eqref{estime_sous_lin_ap_1} and \eqref{estime_sous_lin_ap_2} finally show the existence of a constant $M>0$ which depends only on $u$, $p$ and $d$ such that $\dfrac{|u(x) - u(0)|}{|x|} \leq \dfrac{M}{|x|^{d/p^*}}$,
when $|x|$ is sufficiently large. We both obtain the strict sub-linearity at infinity of $u$ and estimate \eqref{estimme_ss_linearite_ap}. 
\end{proof}

We conclude this section proving that H\"older-continuous functions of $\mathcal{E}^p$ actually belong to $L^q(\mathbb{R}^d)$ for some Lebesgue exponent $q\geq p$. 

\begin{prop}
\label{prop_lebesgue_inclusion}
Let $\alpha \in ]0,1[$ and $p\geq 1$. Then the set $\left\{f \in \mathcal{C}^{0, \alpha}(\mathbb{R}^d) \ \middle| \ \mathcal{M}(|f|) \in L^p(\mathbb{R}^d) \right\}$ is a subset of $L^q(\mathbb{R}^d)$ for every $q \geq \mathbf{q} := \dfrac{p(\alpha+d)-d}{\alpha}$. In addition, the inclusion does not hold if $q<\mathbf{q}$.
\end{prop}

\begin{proof}
We show the proposition in dimension $d=1$ for clarity, the proof for higher dimensions is a simple adaptation. Let $f\in \mathcal{C}^{0,\alpha}(\mathbb{R})$ such that $\mathcal{M}(|f|)\in L^p(\mathbb{R})$. For every $N \in \mathbb{Z}$ and $k\in \left\{0,...,2^{|N|}-1\right\}$, we denote $\displaystyle Q_{N,k} := \left[N+\dfrac{k}{2^{|N|}}\hspace{0.1cm} ;\hspace{0.1cm}N+\dfrac{k+1}{2^{|N|}}\right]$, $\displaystyle \beta_{N,k} := \max_{x \in Q_{N,k}} |f(x)|$ and $x_{N,k} := \underset{x \in Q_{N,k}}{\operatorname{argmax}} \ |f(x)| $. Using Corollary \ref{comportement_asymptotique_AP}, we know that $\displaystyle \sup_{k \in \left\{0,...,2^{|N|}-1\right\}} \beta_{N,k}$ converges to 0 when $N \to \infty$. Since $f \in \mathcal{C}^{0, \alpha}(\mathbb{R})$, there exists $C>0$ such that for every $N$ and every $k \in \left\{0,...,2^{|N|}-1\right\}$, we have $|f(y)|\geq \dfrac{\beta_{N,k}}{2}$ for all $y\in [x_{N,k} - C(\beta_{N,k})^{\frac{1}{\alpha}}, x_{N,k} + C(\beta_{N,k})^{\frac{1}{\alpha}}]$. For every $x\in \left[N,N + \dfrac{1}{2}\right]$, we have  $ \left[N + \dfrac{1}{2},N\right]\subset [x,x+1]$ and we deduce :
$$\int_{x}^{x+1} |f(y)|dy \geq \int^{N}_{N+\frac{1}{2}} |f(y)|dy \geq  C \sum_{k = 2^{|N|-1}}^{2^{|N|}-1} (\beta_{N,k})^{\frac{1}{\alpha}} \beta_{N,k}.$$
Therefore,  
$$\frac{C^p}{2} \sum_{N \in \mathbb{Z}} \sum_{k = 2^{|N|-1}}^{2^{|N|}-1} (\beta_{N,k})^{p\left(1+\frac{1}{\alpha}\right)} \leq \sum_{N \in \mathbb{Z}} \int_{N}^{N+\frac{1}{2}} \left| \int_x^{x+1} |f(y)|dy\right|^p dx \leq \int_{\mathbb{R}} \left| \mathcal{M}(|f|)\right|^p<+\infty.$$ 
We similarly obtain $ \displaystyle \sum_{N \in \mathbb{Z}} \sum_{k = 0}^{2^{|N|-1}} (\beta_{N,k})^{p\left(1+\frac{1}{\alpha}\right)} < + \infty$. For every $q\geq \dfrac{p(1+\alpha) - 1}{\alpha}$, we have 
$$\int_{\mathbb{R}} |f(x)|^q dx \leq \sum_{N \in \mathbb{Z}} \sum_{k = 0}^{2^{|N|}-1} \int_{Q_{N,k}} (\beta_{N,k})^q = \sum_{N \in \mathbb{Z}} \sum_{k = 0}^{2^{|N|}-1} \dfrac{1}{2^{|N|}} (\beta_{N,k})^q.$$
If we distinguish the two cases $(\beta_{N,k})^{\frac{1}{\alpha}} \leq \dfrac{1}{2^{|N|}}$ and  $(\beta_{N,k})^{\frac{1}{\alpha}} \geq \dfrac{1}{2^{|N|}}$, we obtain the following bound : 
$$\int_{\mathbb{R}^d} |f(x)|^q dx \leq \sum_{N \in \mathbb{Z}} \sum_{k = 0}^{2^{|N|}-1} \left(\dfrac{1}{2^{(q \alpha+1) |N|}} + (\beta_{N,k})^{q + \frac{1}{\alpha}}\right) = \sum_{N \in \mathbb{Z}} \left(\dfrac{1}{2^{q \alpha |N|}} + \sum_{k = 0}^{2^{|N|}-1} (\beta_{N,k})^{q + \frac{1}{\alpha}}\right). $$
Since $q+ \dfrac{1}{\alpha} \geq p\left(1+\dfrac{1}{\alpha}\right)$, we conclude that $\displaystyle \int_{\mathbb{R}^d} |f|^q < +\infty$. 

In order to show that the inclusion does not hold when $q< \dfrac{p(1+\alpha) -1}{\alpha}$, we denote $a_{k} = \dfrac{1}{\ln(k)\sqrt{k}}$ for $k\in \mathbb{N}\setminus{\{1,2\}}$ and consider the function $$\displaystyle f(x) = \sum_{k\geq 2} \left( -\dfrac{\sqrt{k}}{k^{\frac{\alpha}{2}}} |x-k|+\dfrac{1}{\ln(k)k^{\frac{\alpha}{2}}}\right)1_{\left[-a_k,  a_k\right]}(x-k),$$
where $1_{A}$ denotes the characteristic function of a subset $A$. We can easily show that $f \in \mathcal{C}^{0,\alpha}(\mathbb{R})$ and $\mathcal{M}(|f|)$ belongs to
$L^p(\mathbb{R})$ for $p = \dfrac{2}{\alpha +1}$ whereas $f \notin L^q(\mathbb{R})$ if $q<\dfrac{1}{\alpha} = \dfrac{p(1+\alpha) - 1}{\alpha}$. 
\end{proof}

\begin{remark}
Without assumption of H\"older continuity, we have the existence of functions $f$ such that $\mathcal{M}(|f|) \in L^p(\mathbb{R}^d)$ for $p>1$ whereas $f$ does not belong to any $L^q(\mathbb{R}^d)$. Consider indeed, for $d=1$,
$$ \displaystyle f(x) = \sum_{k \geq 2} \left( \dfrac{1}{\ln(k)} - \dfrac{\sqrt{k}}{\ln(k)}|x-k|\right) 1_{\left[-\frac{1}{\sqrt{k}}, -\frac{1}{\sqrt{k}}\right]}(x-k).$$
This function, composed of a sum of "bumps" centered at the integers $k\geq 2$, does not belong to any $L^q(\mathbb{R})$ for $q\geq 1$ due to its logarithmic decrease at infinity. However, a simple calculation allows to show that  $\mathcal{M}(|f|)$ belongs to $L^p(\mathbb{R}^d)$, for every $p>2$, because of a classical regularizing property of the local averaging. 
\end{remark}

\subsection{Discrete variant of the Gagliardo-Nirenberg-Sobolev inequality}

In this section we show that, still under the assumption $p<d$, it is possible to obtain a bound on the local average $\mathcal{M}(|f|)$ of a function $f \in \mathcal{A}^p$ using bounds on its discrete derivative $\delta f$. To this end, we establish a discrete variant of the Gagliardo-Nirenberg-Sobolev inequality (see for instance \cite[Section 5.6.1]{evans10} for the classical inequality) adapting its proof in our discrete setting. We begin by showing the result for the functions of $L^p(\mathbb{R}^d)$ in Proposition \ref{Corollaire_GNS_LP} below. Our aim will be next to extend this result to $\mathcal{A}^p$ arguing by density.

\begin{prop}
\label{Corollaire_GNS_LP}
There exists a constant $C>0$ such that for every $f \in L^p(\mathbb{R}^d)$, we have : 
\begin{equation}
\label{GNS_Lp}
    \left\| \mathcal{M}\left(\left|f\right| \right) \right \|_{L^{p^*}(\mathbb{R}^d)} \leq C \| \delta f \|_{L^p(\mathbb{R}^d)}.
\end{equation}
\end{prop}

\begin{proof}
\textbf{Step 1}. We first prove the result for the space of continuous compactly supported functions, denoted by $\mathcal{C}_c^0(\mathbb{R}^d)$ in the sequel. Let $f \in \mathcal{C}_c^0(\mathbb{R}^d)$. For every $i \in \{1,...,d \}$, we remark that : 
$$ \partial_i \mathcal{M}\left(\left|f\right| \right)(x) = \int_{Q_i + \Tilde{x}_i} \delta_i |f(y_1,...,y_{i-1},x_i,y_{i+1},...,y_d)| d\Tilde{y}_i,$$
where $\displaystyle Q_i = \prod_{j \in \{1,...,d\}\setminus{\{i\}}} ]0,1[$ and $\Tilde{x}_i = (x_1,..,x_{i-1}, x_{i+1},...,x_d)$. In addition, using a triangle inequality, we have
$|\delta_i |f(y)|| \leq |\delta_i f(y)|$ for every $i$.
We next denote by $p'$ the conjugate exponent associated with~$p$. Successively using the Hölder inequality, a change of variable and the Fubini theorem, we obtain : 
\begin{align*}
    \int_{\mathbb{R}^d} |\partial_i \mathcal{M}\left(\left|f\right| \right)(x)|^p dx  &  \leq  |Q_i + \Tilde{x_i}|^{p/p'} \int_{\mathbb{R}^d}  \int_{Q_i + \Tilde{x}_i} |\delta_i f(y_1,...,y_{i-1},x_i,y_{i+1},...,y_d)|^p d\Tilde{y}_i dx \\
     & = \int_{\mathbb{R}^d}  \int_{Q_i} |\delta_i f(z_1+x_1,...,z_{i-1}+x_{i-1},x_i,z_{i+1}+x_{i+1},...,z_d+x_d)|^p d\Tilde{z}_i dx \\
     & =  |Q_i| \int_{\mathbb{R}^d}  |\delta_i f(x)|^p dx =  \|\delta_i f\|_{L^p(\mathbb{R}^d)}^p.
\end{align*}
We therefore deduce $\|\nabla \mathcal{M}\left(\left|f\right| \right)\|_{L^p(\mathbb{R}^d)} \leq \| \delta f \|_{L^p(\mathbb{R}^d)}$.
Since $f\in \mathcal{C}^{0}_c(\mathbb{R}^d)$, we have $\mathcal{M}\left(\left|f\right| \right) \in \mathcal{C}^{1}_c(\mathbb{R}^d)$. The classical Gagliardo-Nirenberg-Sobolev inequality allows to conclude to the existence of a constant $C>0$ independent of $f$ such that
$$ \left\| \mathcal{M}\left(\left|f\right| \right) \right \|_{L^{p^*}(\mathbb{R}^d)} \leq C \| \delta f \|_{L^p(\mathbb{R}^d)}.$$  

\textbf{Step 2}. Now that the case of compactly supported function has been dealt with, we can generalize for $L^p(\mathbb{R}^d)$ using a density result. Let $f \in L^p(\mathbb{R}^d)$ and $(f_n)_{n \in \mathbb{N}}$ a sequence of $\mathcal{C}^0_c(\mathbb{R}^d)$ that converges to $f$ in $L^p(\mathbb{R}^d)$. We can easily show that $\delta f_n$ converges to $\delta f$ in this space. In the first step, we have shown the existence of $C>0$ such that for every $n,m \in \mathbb{N}$, we have
\begin{equation}
\label{inégalité_GNS_suite}
    \left\| f_n \right \|_{\mathcal{E}^p} \leq C \| \delta f_n \|_{L^p(\mathbb{R}^d)},
\end{equation}
\begin{equation}
\label{inégalité_GNS_suite2}
    \left\| f_n - f_m \right \|_{\mathcal{E}^p} \leq C \| \delta f_n - \delta f_m\|_{L^p(\mathbb{R}^d)}.
\end{equation}
Since $(\delta f_n)$ converges in $\left(L^p(\mathbb{R}^d)\right)^d$, it is a Cauchy sequence in this space and, using \eqref{inégalité_GNS_suite2}, we deduce that $f_n$ is a Cauchy sequence in the Banach space $\mathcal{E}^p$ and $f_n$ therefore converges to $f$ in this space. We finally take the limit in \eqref{inégalité_GNS_suite} when $n \to \infty$ and we obtain \eqref{GNS_Lp}.
\end{proof}

We next prove a discrete version of Schwarz lemma for the functions of $L^1_{loc}(\mathbb{R}^d)$ (see \cite[Theorem VI, p.59]{schwartz1966theorie} for the classical version). More precisely, we show that if a vectored-valued function $T$ satisfies some discrete Cauchy equations in the sense of \eqref{Cauchy_discret}, then there exists a function $u \in L^1_{loc}$ such that $T$ is the discrete gradient of $u$. In the sequel of this section we shall use this result in order to establish some density properties in $\mathcal{A}^p$. 

\begin{prop}
\label{deRham}
Let $T \in \left(L^1_{loc}(\mathbb{R}^d)\right)^d$ such that for every $i,j\in\{1,...,d\}$, we have : 
\begin{equation}
\label{Cauchy_discret}
    \delta_j T_i = \delta_i T_j.
\end{equation}
Then, there exists $u \in L^1_{loc}(\mathbb{R}^d)$ such that $T = \delta u$. 
\end{prop}

\begin{proof}
For clarity, we only show this result in the case $d=2$, the proof for higher dimensions is similar. In the sequel, for $y \in \mathbb{R}$, we denote by $[y]\in \mathbb{Z}$ the integer part of $y$ and by $\{ y \} \in [0,1[$ its fractional part. We begin by considering $u$ defined by : 
\begin{equation}
    u(x) = \sum_{n = 0}^{[x_1]-1} T_1\left(\left(\{ x_1\}, x_2\right) + n e_1\right) + \sum_{m = 0}^{[x_2]-1} T_2\left(\left(\{ x_1\}, \{ x_2\}\right) + m e_2\right),
\end{equation}
for almost all $x = (x_1, x_2) \in \mathbb{R}^2$. Since $T_i$ belongs to $L^1_{loc}(\mathbb{R}^d)$ for every $i\in \{1,2\}$, we clearly have $u \in L^1_{loc}(\mathbb{R}^d)$. We next show that $\delta u = T$. Indeed, we have
\begin{align*}
    \delta_1 u(x) = u(x+e_1) - u(x)
    & = \sum_{n = 0}^{[x_1]} T_1\left(\left(\{ x_1\}, x_2\right) + n e_1\right) - \sum_{n = 0}^{[x_1]-1} T_1\left(\left(\{ x_1\}, x_2\right) + n e_1\right) \\
    & = T_1\left(\left(\{ x_1\}, x_2\right) + [x_1] e_1\right) = T_1(x).
\end{align*}
We can similarly show that : 
\begin{align*}
    \delta_2 u(x) = \sum_{n = 0}^{[x_1]-1} \delta_2 T_1\left(\left(\{ x_1\}, x_2\right) + n e_1 \right) + T_2\left(\left(\{ x_1\}, x_2\right)\right).
\end{align*}
Since $\delta_2 T_1 = \delta_1 T_2$, we deduce :
$$\sum_{n = 0}^{[x_1]-1} \delta_2 T_1\left(\left(\{ x_1\}, x_2\right) + n e_1 \right) = \sum_{n = 0}^{[x_1]-1} \delta_1 T_2\left(\left(\{ x_1\}, x_2\right) + n e_1 \right) = T_2(x) - T_2\left(\left(\{ x_1\}, x_2\right)\right).$$
We finally obtain that $\delta_2 u(x) = T_2(x)$, and we can conclude that $T = \delta u$. 
\end{proof}

In the sequel, for every $R>0$, we denote
$\displaystyle Q_R = \left\{ x \in \mathbb{R}^d \ \middle| \ \max_{i \in \{1,...,d\}} |x_i| < R \right\}$.
The next lemma is a discrete version of a particular case of the Poincaré-Wirtinger inequality (see for example \cite[Section 5.8.1]{evans10} for the classical version). This result is a technical tool that will allow us to establish our discrete Gagliardo-Nirenberg-Sobolev inequality stated in Proposition~\ref{decomposition_fonctions_A_p}.


\begin{lemme}
\label{Poincare_Wirtinger_discret}
Assume $d\geq 2$. Let $p\in [1,+\infty[$ and $f$ be in $ \mathbf{A}^p$, the set defined by \eqref{circ_AP}. For every $N\in \mathbb{N}^*$, we denote $A_N = Q_{2N} \setminus{Q_{N}}$ and
$\mathcal{I}_N = \left\{ k \in \mathbb{Z}^d \ \middle| \ Q+k \subset A_N \right\}$.
We consider
\begin{equation}
\label{moyenne_periodique_f}
   \displaystyle f_{per,N}(x) = \frac{1}{\sharp \mathcal{I}_N} \sum_{k \in \mathcal{I}_N} f(x + k) \quad \text{for } x\in Q, 
\end{equation}
(where $\sharp B$ denotes the cardinality of a discrete set $B$) which we extend by periodicity. Then there exists a constant $C>0$ independent of $N$ and $f$ such that : 
\begin{equation}
    \|f-f_{per,N}\|_{L^p(A_N)} \leq C N \|\delta f\|_{L^p(Q_{6N}\setminus{Q_N})}.
\end{equation}
\end{lemme}

\begin{proof}
For clarity, we show the result in the case $d=2$, the proof in higher dimensions is similar. 
First of all, we remark that for every $k\in \mathcal{I}_N$, we have 
$N\leq \max(|k_1|, |k_2|)<2N$.
We next estimate the $L^p$-norm of $f - f_{per,N}$ : \begin{align*}
   \|f - f_{per,N}\|_{L^p(A_N)}^p &= \sum_{q\in \mathcal{I}_N} \int_{Q+q} |f - f_{per,N}|^p 
    = \sum_{q\in \mathcal{I}_N} \int_{Q} \left| \frac{1}{\sharp \mathcal{I}_N}\sum_{k\in \mathcal{I}_N} f(x+q) - f(x+k)\right|^p dx.
\end{align*}
Using the Holder inequality, we therefore obtain : 
\begin{equation}
\label{inegalité_GNS_norm_ap}
   \|f - f_{per,N}\|_{L^p(A_N)}^p \leq \frac{1}{\sharp \mathcal{I}_N}\sum_{q\in \mathcal{I}_N}\sum_{k\in \mathcal{I}_N} \int_{Q} \left|f(x+q) - f(x+k)\right|^p dx. 
\end{equation}

\begin{figure}[h!]
    \centering
    \includegraphics[width = 6.5cm]{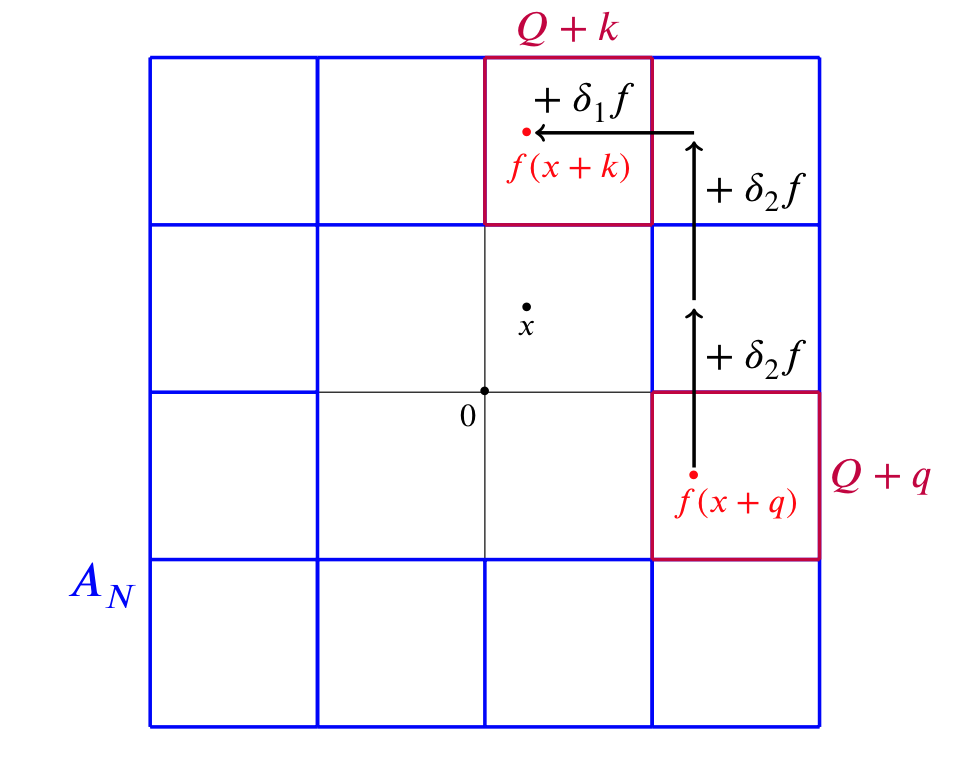}
    \caption{Illustration of the relation between $f(x+q)$, $\delta_i f$ and $f(x+k)$ in dimension $d=2$.}
    \label{fig:illustration_preuve_GNS}
\end{figure}

Our aim is now to study $|f(x+q) - f(x+k)|$ for every $q,k \in \mathcal{I}_N$ and $x \in Q$ splitting $f(x+q)-f(x+k)$ as a sum of several translations of the functions $\delta_i f$, where $i \in \{1,2\}$. Our approach, which is illustrated in figure \ref{fig:illustration_preuve_GNS}, consists in considering a discrete path of $A_N$  connecting $x+k$ and $x+q$ and, next, in iterating the relation~: 
\begin{equation}
\label{argument_it_poincareap}
  f(x+ne_j) = f(x) + \sum_{m=0}^{n-1} \delta_j f(x+ me_j), \quad \forall n\in \mathbb{N}.  
\end{equation}

In the sequel, we use the convention $\displaystyle \sum_{m=0}^{n} := \sum_{m=n}^0$ when $n\leq 0$. We consider two different cases. 

\textbf{Case 1 : There exists $i \in \{1,2\}$ such that $|k_i|\geq N$ and $|q_i|\geq N$}. Without loss of generality, we can assume that $i=1$. For every $x \in Q$,  using \eqref{argument_it_poincareap} for $j=2$ we have : 
\begin{equation}
    \label{equality_poincare_ap_1}
    f(x+q) =  f(x + (q_1,2N-1)) - \sum_{l = 0}^{2N - 2 - q_2} \delta_2 f(x + q + le_2).
\end{equation}
We can again iterate \eqref{argument_it_poincareap} and we have : 
\begin{align}
       \label{equality_poincare_ap_2}
  f\left(x + (q_1,2N-1)\right) & =  f\left(x+(k_1,2N-1)\right) - \sum_{m =0}^{k_1 - q_1-1} \delta_1 f(x+(q_1,2N-1) + m e_1),\\
  \label{equality_poincare_ap_3}
    f(x + (k_1,2N-1)) & = f(x+k) - \sum_{n = 0}^{k_2 - 2N} \delta_2 f(x + (k_1,2N-1) + ne_2).
\end{align}
Using \eqref{equality_poincare_ap_1}, \eqref{equality_poincare_ap_2} and \eqref{equality_poincare_ap_3}, we therefore obtain : 
\begin{align*}
    f(x+k) - f(x+q) & = \sum_{l = 0}^{2N - 2 - q_2} \delta_2 f(x + q + le_2)+ \sum_{m =0}^{k_1 - q_1-1} \delta_1 f(x+(q_1,2N-1) + m e_1) \\
    & + \sum_{n = 0}^{k_2 - 2N} \delta_2 f(x + (k_1,2N-1) + ne_2)
\end{align*}
Since each sum in the right-hand side of the above equality contains at most $4N$ terms, the H\"older inequality implies that  : 
\begingroup
\begin{align*}
    |f(x+q) - f(x+k)|^p & \leq CN^{p-1}\left( \sum_{l = 0}^{2N -2 - q_2} |\delta_2 f(x + q + le_2)|^p 
     + \sum_{m =0}^{k_1 - q_1-1} |\delta_1 f(x+(q_1,2N-1) + m e_1)|^p \right) \\
    & + C N^{p-1} \sum_{n = 0}^{k_2 - 2N} |\delta_2 f(x + (k_1,2N-1) + ne_2)|^p,
\end{align*}
\endgroup
where $C>0$ depends only on $d$ and $p$. 
We now insert our last inequality in \eqref{inegalité_GNS_norm_ap},  for every $|k_1|>N$ and $|q_1|>N$. We consequently have to study the following three sums : 
\begin{align*}
   S_1 & : =  \frac{1}{\sharp \mathcal{I}_N}\sum_{q\in \mathcal{I}_N, \ |q_1|\geq N}\sum_{k\in \mathcal{I}_N, \ |k_1|\geq N} \int_{Q} CN^{p-1} \sum_{l = 0}^{2N-2 - q_2} |\delta_2 f(x + q + le_2)|^p dx,
\end{align*}
\begin{equation}
    S_2 := \frac{1}{\sharp \mathcal{I}_N}\sum_{q\in \mathcal{I}_N, \ |q_1|\geq N}\sum_{k\in \mathcal{I}_N, \ |k_1|\geq N} \int_{Q} CN^{p-1} \sum_{m =0}^{k_1 - q_1-1} |\delta_1 f(x+(q_1,2N-1) + m e_1)|^p dx,
\end{equation}
\begin{equation}
    S_3 := \frac{1}{\sharp \mathcal{I}_N}\sum_{q\in \mathcal{I}_N, \ |q_1|\geq N}\sum_{k\in \mathcal{I}_N, \ |k_1|\geq N} \int_{Q} CN^{p-1} \sum_{n = 0}^{k_2 - 2N} |\delta_2 f(x + (k_1,2N-1) + n e_2)|^p dx.
\end{equation}
Here we only study $S_1$, the method to estimate $S_2$ and $S_3$ being extremely similar. 
Since $|q_2| \leq 2N-1$, we have 
\begin{equation}
\label{majoraition_S1_ap_1}
    \begin{array}{cc}
       S_1   &  \displaystyle \leq  \frac{1}{\sharp \mathcal{I}_N}\sum_{q\in \mathcal{I}_N, \ |q_1|\geq N}\sum_{k\in \mathcal{I}_N, \ |k_1|\geq N} \int_{Q} CN^{p-1} \sum_{l = -4N}^{4N} |\delta_2 f(x + q + le_2)|^p dx \\
     \vspace{-0.3cm}  & \\ 
         &  \displaystyle = \frac{CN^{p-1}}{\sharp \mathcal{I}_N} \sum_{l = -4N}^{4N}   \sum_{k\in \mathcal{I}_N, \ |k_1|\geq N} \sum_{q\in \mathcal{I}_N, \ |q_1|\geq N} \int_{Q+ q + le_2}  |\delta_2 f(x )|^p dx. \hspace{1cm}
    \end{array}
\end{equation}

In addition, for every $q\in \mathcal{I_N}$ such that $|q_1|\geq N$ and every $l \in \{-4N,...,4N\}$, we have that
$\displaystyle N\leq~\max_{i\in \{1,2\}} |(q+le_2)_i|<6N$,
that is : 
\begin{equation}
\label{majoration_S1_ap_2}
  \sum_{q\in \mathcal{I}_N, \ |q_1|\geq N} \int_{Q+ q + l e_2}  |\delta_2 f(x )|^p dx \leq \int_{Q_{6N}\setminus{Q_ N}} |\delta_2 f|^p.  
\end{equation}
Since $\sharp I_N \sim K N^2$, where $K>0$ is a constant independent of $N$, the sum 
$\displaystyle \dfrac{1}{\sharp \mathcal{I}_N} \sum_{l = -4N}^{4N}   \sum_{k\in \mathcal{I}_N, |k_1|\geq N} 1$ is a $O(N)$.
Using \eqref{majoraition_S1_ap_1} and \eqref{majoration_S1_ap_2}, we obtain the existence of a constant $\Tilde{C}>0$ independent of $N$ such that : 
$$ S_1 \leq \Tilde{C}N^p \int_{Q_{6N} \setminus{Q_N}} |\delta_2 f|^p.$$
With exactly the same method, it is possible to establish similar bounds for the sums $S_2$ and $S_3$, which show that
\begin{equation}
\label{inegalite_poincare_ap_1}
   \frac{1}{\sharp \mathcal{I}_N}\sum_{q\in \mathcal{I}_N, \ |q_1|\geq N}\sum_{k\in \mathcal{I}_N, \ |k_1| \geq N} \int_{Q} \left|f(x+q) - f(x+k)\right|^p dx \leq 3 \Tilde{C}N^p  \int_{Q_{6N} \setminus{Q_N}} \sum_{i\in\{1,2\}}|\delta_i f|^p. 
\end{equation}

\textbf{Case 2 : There exists $i,j\in \{1,2\}$, $i\neq j$, such that $|q_i|\geq N$ and $|k_j| \geq N$}. We assume that $i=1$ and $j=2$, the idea being identical if $i=2$ and $j=1$. Proceeding exactly as in the first case, it is possible to show that 
\begin{equation}
 |f(x+q) - f(x+k)|^p   \leq CN^{p-1} \left(\sum_{l = 0}^{k_2 - q_2-1} |\delta_2 f(x + q + le_2)|^p  +\sum_{m = 0}^{k_1 - q_1-1} |\delta_2 f(x + (q_1,k_2) + m e_1)|^p\right).
\end{equation}
Since $|q_1|\geq N$, et $|k_2| \geq N$, we remark that each point of the form $x + q + le_2$ or $x + (q_1,k_2)+me_1$ belongs to $A_N$. Some estimates similar to those established in the first case allow to obtain 
\begin{equation}
\label{inegalité_poincare_ap_2}
    \frac{1}{\sharp \mathcal{I}_N}\sum_{q\in \mathcal{I}_N, \ |q_1|\geq N}\sum_{k\in \mathcal{I}_N, \ |k_2| \geq N} \int_{Q} \left|f(x+q) - f(x+k)\right|^p dx \leq \Tilde{C}N^p \sum_{i\in \{1,2\}} \int_{Q_{6N} \setminus{Q_N}} |\delta_i f|^p.
\end{equation}
Using finally \eqref{inegalité_GNS_norm_ap} and inequalities \eqref{inegalite_poincare_ap_1}, \eqref{inegalité_poincare_ap_2} established in the two different cases, we have 
\begin{align*}
   \|f-f_{per,N}\|_{L^p(A_N)}^p & \leq \sum_{i,j \in \{1,2\}} \frac{1}{\sharp \mathcal{I}_N}\sum_{q\in \mathcal{I}_N, \ |q_i| \geq N}\sum_{k\in \mathcal{I}_N, \ |k_j| \geq N} \int_{Q} \left|f(x+q) - f(x+k)\right|^p dx  \\
   & \leq \Tilde{C} N^p  \sum_{i \in \{1,2\}} \int_{Q_{6N} \setminus{Q_N}} |\delta_i f|^p.
\end{align*}
We have therefore established the existence of a constant $C>0$ independent of $f$ and $N$ such that : 
$$\|f-f_{per,N}\|_{L^p(A_N)} \leq CN\|\delta f \|_{L^p(Q_{6N} \setminus{Q_N})}.$$
\end{proof}

\begin{remark}
For clarity, we have chosen to show the inequality of Lemma \ref{Poincare_Wirtinger_discret} only for the particular sets $A_N$ but the proof could be adapted to any connected set. On the other hand, for non-connected set, this result does not hold. In particular Lemma \ref{Poincare_Wirtinger_discret} is not true in dimension $d=1$. As a counter-example, we can consider the function $f$ such that $f(x)=2$ if $x<0$ and $f(x)=1$ else. Such a function satisfies $\delta f = 0$ on  $\mathbb{R}\setminus{[-1,1]}$ but, for every $N\in \mathbb{N}^*$, $f - f_{N,per}$ 
does not vanish on $A_N = ]-2N,-N]\cup [N,2N[$. Indeed, $f - f_{N,per}(x)$ is equal to $\dfrac{1}{2}$ if $x\in ]-2N,-N]$ and $-\dfrac{1}{2}$ if $x \in [N,2N[$. 
\end{remark}

We are finally able to prove the discrete version of the Gagliardo-Nirenberg-Sobolev inequality stated in Proposition \ref{decomposition_fonctions_A_p}. 


\begin{proof}[Proof of Proposition \ref{decomposition_fonctions_A_p}]
We begin by establishing the density of  $L^p(\mathbb{R}^d)$ in $ \mathbf{A}^p$, equipped with the semi-norm $\|f\| =~\|\delta f\|_{\left(L^p(\mathbb{R}^d)\right)^d}$. We adapt step by step the method used in \cite[Theorem 2.1]{ortner2012note} which studies the continuous case $\nabla f \in \left(L^p(\mathbb{R}^d)\right)^d$. We fix $f \in L^1_{loc}(\mathbb{R}^d)$ such that $\delta f \in \left(L^p(\mathbb{R}^d)\right)^d$. For every $R>0$, we consider $\chi_R \in \mathcal{D}(\mathbb{R}^d)$, a positive function such that 
\begin{equation}
\label{hyp_chi_n}
    Supp(\chi_R) \subset Q_{\frac{5R}{3}}, \quad \chi_R \equiv 1 \text{ in } Q_{\frac{4R}{3}}, \quad \|\chi_R\|_{L^{\infty}(\mathbb{R}^d)} = 1, \quad \| \nabla \chi_R \|_{L^{\infty}(\mathbb{R}^d)} \leq \frac{1}{R}.
\end{equation}
Let $N \in \mathbb{N}^*$. We denote $A_N = Q_{2N} \setminus{Q_N}$
and we consider $f_{per,N}$ the $Q$-periodic function defined by \eqref{moyenne_periodique_f} in Lemma \ref{Poincare_Wirtinger_discret}. We introduce the sequence $f_N = \chi_N \left(f - f_{per,N}\right)$.
For every $N\in \mathbb{N}$ and $i \in \{1,...,d\}$, using a triangle inequality, we have 
$$ \|\delta_i f_N - \delta_i f\|_{L^p(\mathbb{R}^d)} \leq \|\delta_i \chi_N \left(f - f_{per,N}\right)\|_{L^p(\mathbb{R}^d)} + \|(1-\chi_N) \delta_i f \|_{L^{p}(\mathbb{R}^d)}.$$

Since $\delta f \in \left(L^{p}(\mathbb{R}^d)\right)^d$, we have $$\|(1-\chi_N)\delta_i f\|_{L^p(\mathbb{R}^d)} \leq \|\delta_i f \|_{L^p(\mathbb{R}^d \setminus{Q_N})} \underset{N\to+\infty}{\longrightarrow} 0.$$
Since $\chi_N$ is supported in $Q_{\frac{5N}{3}}$ and $\chi_N \equiv 1$ on  $Q_{\frac{4N}{3}}$, $\delta \chi_N$ is supported in $A_N$ for $N$ sufficiently large. 
Lemma \ref{Poincare_Wirtinger_discret} yields the existence of a constant $C>0$ such that for every $N\in \mathbb{N}^*$, we have : 
\begin{equation}
\label{Majoration_A_N_1}
   \|f - f_{per,N}\|_{L^p(A_N)} \leq C N\|\delta f \|_{L^p(\mathbb{R}^d \setminus{Q}_N)}. 
\end{equation}
We next use the mean-value inequality so that, for every $i\in \{1,...,d\}$ and $N\in \mathbb{N}$ : 
\begin{equation}
\label{accroissements_finis_chi}
  \|\delta_i \chi_N\|_{L^{\infty}(\mathbb{R}^d)} \leq \|\nabla \chi_N \|_{\left(L^{\infty}(\mathbb{R}^d)\right)^d} \leq \frac{1}{N}.  
\end{equation}
Using \eqref{Majoration_A_N_1} and \eqref{accroissements_finis_chi}, we obtain : 
\begin{align*}
   \|\delta_i \chi_N \left(f - f_{per,N}\right)\|_{L^p(\mathbb{R}^d)} \leq C \|\delta f \|_{L^p(\mathbb{R}^d \setminus{Q_N})} \underset{N\to+\infty}{\longrightarrow} 0.
\end{align*}
We conclude that the sequence $\delta f_N$ converges to $\delta f$ in $L^p(\mathbb{R}^d)$.

We next show that the sequence $f_N$ is a Cauchy sequence in $\mathcal{E}^p$. Let $M$ and $N$ in $\mathbb{N}$. Since $f_N - f_M \in L^p(\mathbb{R}^d)$, we know from Proposition \ref{Corollaire_GNS_LP} the existence of a constant $C>0$ independent of $N$ and $M$ such that : 
$$\|f_N - f_M\|_{\mathcal{E}^p} \leq C \|\delta f_N - \delta f_M\|_{L^p(\mathbb{R}^d)}.$$
We have proved that $\delta f_N$ converges in $\left(L^p(\mathbb{R}^d)\right)^d$, it is therefore a Cauchy sequence for the $L^p$-norm and we can deduce that $f_N$ is also a Cauchy sequence in the Banach space $\mathcal{E}^p$. Thus, there exists $g\in\mathcal{E}^p$ such that $f_N$ converges to $g$ in $\mathcal{E}^p$. We use Proposition \ref{sous-suite} and we obtain, up to an extraction,  that $f_N$ also converges to $g$ in $L^1_{loc}(\mathbb{R}^d)$. Since $\delta f_N$ converges to $\delta f$ in $L^p$, the uniqueness of the limit in $L^1_{loc}(\mathbb{R}^d)$ shows that 
$\displaystyle \delta g = \lim_{N \to \infty} \delta f_N = \delta f$.
We deduce that $\delta(f-g) = 0$, that is, $f_{per} := f-g$ is a $Q$-periodic function.  
To conclude, we again use Proposition \ref{Corollaire_GNS_LP} to obtain
 $$\|f_N\|_{\mathcal{E}^p}\leq C \|\delta f_N\|_{L^p(\mathbb{R}^d)}.$$
We pass to the limit in this inequality and we obtain  $$\|f-f_{per}\|_{\mathcal{E}^p} = \|g\|_{\mathcal{E}^p}\leq C \|\delta f\|_{L^p(\mathbb{R}^d)}.$$

There remains to show the uniqueness of $f_{per}$. We assume there exists $f_{per,1}$ and $f_{per,2}$ two periodic functions such that both $\Tilde{f}_1 := f-f_{per,1}$ and $\Tilde{f}_2 := f - f_{per,2}$ belong to $\mathcal{E}^p$. Thus, we have $\mathcal{M}\left(|f_{per,1} - f_{per,2}|\right) = \mathcal{M}\left(|\Tilde{f}_2 - \Tilde{f}_1|\right) \in L^{p^*}(\mathbb{R}^d)$. Since $|f_{per,1} - f_{per,2}|$ is a periodic function, $\mathcal{M}\left(|f_{per,1} - f_{per,2}|\right)$ is constant and is therefore equal to 0. 
\end{proof}

From the previous proof, we deduce the next corollary, which will be useful in the next Section.

\begin{corol}
\label{densité_Ap}
Let $f \in \mathcal{A}^p$, then there exists a sequence $(f_n)_{n\in \mathbb{N}}$ of $L^p(\mathbb{R}^d)$ such that $\displaystyle \lim_{n \to \infty} \|f_n-f\|_{\mathcal{A}^p} =0$. In addition, if $f \in L^{\infty}(\mathbb{R}^d)$, the sequence $(f_n)_{n\in \mathbb{N}}$ can be choosen such that for every $n\in \mathbb{N}$, $\displaystyle \|f_n\|_{L^{\infty}(\mathbb{R}^d)} \leq 2 \|f\|_{L^{\infty}(\mathbb{R}^d)}$.
\end{corol}

\begin{proof}
In the proof of Proposition \ref{decomposition_fonctions_A_p}, we have established the existence of a sequence $f_n$ of $L^p(\mathbb{R}^d)$-functions and the existence of a $Q$-periodic function $f_{per}$ such that $f_n$ converges to $f-f_{per}$ in $\mathcal{E}^p$. Since $f$ and $f - f_{per}$ both belong to $\mathcal{E}^p$, we clearly have $f_{per}=0$. We can conclude that $f_n$ converges to $f$ in $\mathcal{E}^p$.
Now, we assume that $f \in L^{\infty}(\mathbb{R}^d)$. We have shown in the proof of Proposition \ref{decomposition_fonctions_A_p}, that the sequence $(f_n)_{n\in \mathbb{N}^*}$ can be defined by
$f_n = \chi_n\left(f - f_{per,n}\right)$, 
where $f_{per,n}$ is the periodic function given by \eqref{moyenne_periodique_f} and $\chi_n$ satisfies \eqref{hyp_chi_n}. We clearly have $\|f_{per,n}\|_{L^{\infty}(\mathbb{R}^d)} \leq \|f\|_{L^{\infty}(\mathbb{R}^d)}$. Finally, since $\|\chi_n\|_{L^{\infty}(\mathbb{R}^d)} \leq 1$, we obtain $\|f_n\|_{L^{\infty}(\mathbb{R}^d)} \leq 2 \|f\|_{L^{\infty}(\mathbb{R}^d)}$ using a triangle inequality.
\end{proof}

To conclude this section, we show that it is possible to describe the periodic function $f_{per}$ given in Proposition \ref{decomposition_fonctions_A_p} when the function $f \in \mathcal{A}^p$ is assumed uniformly H\"older-continuous. 


\begin{prop}
\label{expliciteperiodiqueap}
Let $f \in \mathbf{A}^p \cap L^{\infty}(\mathbb{R}^d)$ for $1\leq p<d$. Then the unique $Q$-periodic function $f_{per}$ such that $f-f_{per} \in \mathcal{E}^p$ given by Proposition \ref{decomposition_fonctions_A_p} is equal to
\begin{equation}
\label{expressionfper}
    f_{per} = \lim_{N \to \infty} \dfrac{1}{\sharp I_N} \sum_{k \in I_N} f(.+k), 
\end{equation}
where $I_N= \left\{ k \in \mathbb{Z}^d \ \middle| \ |k| \leq N \right\}$ for every $N \in \mathbb{N}^*$. In addition, if there exists $\alpha \in ]0,1[$ such that $f\in \mathcal{C}^{0,\alpha}(\mathbb{R}^d)$, then $f_{per} \in \mathcal{C}^{0, \alpha}(\mathbb{R}^d)$. 
\end{prop}

\begin{proof}
We define $\Tilde{f} = f -f_{per}$. We first show that if $f \in L^{\infty}(\mathbb{R}^d)$, then $\mathcal{M}(|\Tilde{f}|)$ is Lipschitz continuous on $\mathbb{R}^d$. To this end, we remark that for every $i \in \{1,...,d\}$, we have 
$$ \partial_i \mathcal{M}(|\Tilde{f}|)(x) = \int_{Q_i + \Tilde{x}_i} \delta_i |\Tilde{f}(y_1,...,y_{i-1},x_i,y_{i+1},...,y_d)| d\Tilde{y}_i, $$
where $\displaystyle Q_i = \prod_{j \in \{1,...,d\}\setminus{\{i\}}} ]0,1[$ and $\Tilde{x}_i = (x_1,..,x_{i-1}, x_{i+1},...,x_d)$. Using a triangle inequality we have $\delta_i |\Tilde{f}| \leq |\delta_i \Tilde{f}|$ and, since $\delta_i \Tilde{f} = \delta_i f$ and $\|\delta f\|_{L^{\infty}(\mathbb{R}^d)} \leq 2 \|f\|_{L^{\infty}(\mathbb{R}^d)}$, we can bound the integral uniformly with respect to $x$ in the previous equality and we have : 
$$\left\|\partial_i \mathcal{M}(| \Tilde{f}| )\right\|_{L^{\infty}(\mathbb{R}^d)} \leq 2 \|f\|_{L^{\infty}(\mathbb{R}^d)}.$$
It follows that $\nabla  \mathcal{M}(| \Tilde{f}|)$ belongs to $\left( L^{\infty}(\mathbb{R}^d)\right)^d$ and we deduce that $\mathcal{M}(| \Tilde{f}|)$ is Lipschitz-continuous. Moreover, since $\mathcal{M}(| \Tilde{f}|)$ belongs to $L^{p^*}(\mathbb{R}^d)$, we have that 
$\displaystyle \lim_{|x| \to \infty} \mathcal{M}(| \Tilde{f}|)(x) = 0$, 
and for every $x\in \mathbb{R}^d$, the Ces\`aro mean of the sequence $\left(\mathcal{M}(| \Tilde{f}|)(x+k)\right)_{k \in \mathbb{Z}^d}$ is equal to 0. We therefore obtain that $\displaystyle \lim_{N\to \infty} \dfrac{1}{\sharp I_N} \sum_{k \in I_N} \mathcal{M}(| \Tilde{f}|)(x+k) = 0$.
Consequently, for every $x \in \mathbb{R}^d$ and $N \in \mathbb{N}^*$, we have 
\begin{align*}
    \bigintss_{Q+x} \left| \dfrac{1}{\sharp I_N} \sum_{k\in I_N} f(y+k) - f_{per}(y)\right| dy & = \bigintss_{Q+x} \left| \dfrac{1}{\sharp I_N} \sum_{k\in I_N} \Tilde{f}(y+k)\right| dy \\
   &  \leq \dfrac{1}{\sharp I_N} \sum_{k \in I_N} \mathcal{M}(| \Tilde{f}|)(x+k) \stackrel{N \to \infty}{\longrightarrow} 0,
\end{align*}
and deduce $\displaystyle \lim_{N \to \infty} \dfrac{1}{\sharp I_N} \sum_{k\in I_N} f(.+k) = f_{per}$ in $L^1_{loc}(\mathbb{R}^d)$. If we now assume that $f \in \mathcal{C}^{0,\alpha}(\mathbb{R}^d)$ for $\alpha \in ]0,1[$, we have for every $N \in \mathbb{N}^*$ and $x,y \in \mathbb{R}^d$ : 
\begin{align}
\label{Cesaro_holder_continuity}
   \left|\dfrac{1}{\sharp I_N} \sum_{k\in I_N} f(x+k)\right| & \leq \|f\|_{L^{\infty}(\mathbb{R}^d)},\\ 
   \label{Cesaro_uniform_bound}
    \left|\dfrac{1}{\sharp I_N} \sum_{k\in I_N} f(x+k) - \dfrac{1}{\sharp I_N} \sum_{k\in I_N} f(y+k)\right| & \leq \|f\|_{\mathcal{C}^{0, \alpha}(\mathbb{R}^d)} |x-y|^{\alpha}.
\end{align}
In addition, up to an extraction, the sequence $\left(\dfrac{1}{\sharp I_N} \sum_{k\in I_N} f(.+k)\right)_{N \in \mathbb{N}^*}$ converges to $f_{per}$ almost everywhere and, considering the limit when $N \to \infty$ in \eqref{Cesaro_uniform_bound} and \eqref{Cesaro_holder_continuity}, we obtain for almost all $x,y \in \mathbb{R}^d$~: 
$$\left|f_{per}(x)\right| \leq \|f\|_{L^{\infty}(\mathbb{R}^d)} \quad \text{and} \quad \left|f_{per}(x) - f_{per}(y)\right| \leq \|f\|_{\mathcal{C}^{0, \alpha}(\mathbb{R}^d)} |x-y|^{\alpha}.$$
\end{proof}

\begin{remark}
More generally, in the proof of Proposition \ref{expliciteperiodiqueap} we have actually shown that if $f = f_{per} + \Tilde{f}$ where $f_{per}$ is periodic and $\displaystyle \lim_{|x| \to \infty} \mathcal{M}(|\Tilde{f}|) = 0$, then $f_{per}$ is necessarily given by~\eqref{expressionfper}. 
\end{remark}

\begin{remark}
\label{remarque_ap_period_T}
All the results established in this section can be easily adapted in a context of $T$-periodicity at infinity, for any period $T=(T_1,T_2,...,T_d)$, considering $ \displaystyle \delta_T f := \left(f(.+T_i) - f\right)_{i \in \{1,...,d\}}$ instead of $\delta f$ and $ \displaystyle \mathcal{M}_T(|f|)(x) := \int_{\prod_{i=1}^d ]0,T_i[} |f(y+x)| dy $ instead of $\mathcal{M}(|f|)$.
\end{remark}

\section{The homogenization problem when $p<d$}

\label{SectionAP_2}

In this section we study homogenization problem \eqref{equationepsilon_new} when the coefficient $a$ satisfies assumptions \eqref{hypothèses1}, \eqref{hypothèses2}, \eqref{hypothèses22} and \eqref{hypothèses3} for $1<p<d$ and we prove Theorem \ref{Correcteur_A_P} in this case. As in the previous section, our assumption $p<d$ of course requires that $d\geq 2$. Since $p<d$, Proposition \ref{decomposition_fonctions_A_p} gives the existence of two matrix-valued functions $a_{per} \in \left(L^2_{per}\right)^{d \times d}$ and $\Tilde{a} \in \left(\mathcal{A}^p\right)^{d \times d}$ such that $a = a_{per} + \Tilde{a}$ and $\|\Tilde{a}\|_{\mathcal{E}^p} \leq C \|\delta a \|_{L^p(\mathbb{R}^d)}$,
where $C>0$ is a constant independent of $a$.
We are therefore indeed studying a problem of perturbed periodic geometry in the presence of a local defect $\Tilde{a}$, which is, up to a local averaging, a matrix-valued function with components in $L^{p^*}(\mathbb{R}^d)$. We note that assumptions \eqref{hypothèses1}, \eqref{hypothèses22} and Proposition \ref{expliciteperiodiqueap} ensure that the coefficients $a_{per}$ and $\Tilde{a}$ also satisfy the following two properties of ellipticity and regularity :
\begin{align}
  \label{hypothèse_a_elliptic_AP}
   \exists \lambda> 0, \text{  $\forall x$, $\xi \in \mathbb{R}^d$} \quad  \lambda |\xi|^2 \leq \langle a_{per}(x)\xi, \xi\rangle, \\  
   \label{hypothèse_a_holder_AP}
   a_{per}, \Tilde{a} \in \left(\mathcal{C}^{0,\alpha}(\mathbb{R}^d)\right)^{d\times d}.
\end{align}


In order to study the corrector equation, we adapt the method introduced in \cite{ blanc2018correctors}. We remark that \eqref{eq_correcteur_A_p} is equivalent to $- \operatorname{div}(a\nabla \Tilde{w}_q) = \operatorname{div}(\Tilde{a}(\nabla w_{per,q} + q))$.
Under assumption \eqref{hypothèse_a_holder_AP}, elliptic regularity theory (see for instance \cite[Theorem 5.19 p.87]{MR3099262}) also implies that $\nabla w_{per,q} \in \left(\mathcal{C}^{0, \alpha}(\mathbb{R}^d)\right)^d$. Thus, $f = \Tilde{a}(\nabla w_{per} + q)$ belongs to $\left(\mathcal{E}^p \cap \mathcal{C}^{0, \alpha}(\mathbb{R}^d)\right)^d$ and, since $\nabla w_{per,q}$ is periodic, we have $\delta f \in \left(L^p(\mathbb{R}^d)\right)^{d \times d}$. To prove Theorem~\ref{Correcteur_A_P}, it is therefore sufficient to study the more general problem :
\begin{equation}
    \label{equation_generale_Ap}
    - \operatorname{div}(a \nabla u) = \operatorname{div}(f) \quad \text{on } \mathbb{R}^d,
\end{equation}
for every $ f\in \left(\mathcal{A}^p \cap \mathcal{C}^{0, \alpha}(\mathbb{R}^d)\right)^d$. 


\subsection{Preliminary regularity result}

We begin by establishing a regularity result for the solutions $u$ to \eqref{equation_generale_Ap} such that $\nabla u \in \left(\mathcal{A}^p\right)^d$. We need to introduce the space 
$$\displaystyle L^2_{unif}(\mathbb{R}^d) = \left\{f \in L^2_{loc}(\mathbb{R}^d) \ \middle| \ \sup_{x \in \mathbb{R}^d} \|f\|_{L^2(B_1(x))} < \infty \right\},$$
equipped with $\displaystyle \|f\|_{L^2_{unif}} = \sup_{x \in \mathbb{R}^d} \|f\|_{L^2(B_1(x))}.$ We have :

\begin{lemme}
\label{lemme_majoration_L2_unif}
There exists $C>0$ such that, for every $0<R<1$ and $f \in \mathcal{C}^{0,\alpha}(\mathbb{R}^d)$, 
\begin{equation}
\label{born_norm_unif_ap}
  \|f\|_{L^2_{unif}(\mathbb{R}^d)} \leq C\left(\dfrac{1}{|B_R|}\|\mathcal{M}(|f|)\|_{L^2_{unif}(\mathbb{R}^d)} + R^{\alpha}\|f\|_{C^{0,\alpha}(\mathbb{R}^d)}\right).  
\end{equation}

\end{lemme}

\begin{proof}
We denote $B_R^+(x)  = B_R(x) \cap \left(Q+x\right)$ and $\displaystyle \fint_{B_R^+(x)} = \dfrac{1}{|B_R^+(x)|} \int_{B_R^+(x)}$, 
for $R>0$ and $x\in \mathbb{R}^d$. For every $x_0 \in \mathbb{R}^d$ and $0<R<1$, we have :
\begin{align*}
    \int_{B_1(x_0)}\left| f(x) - \fint_{B_R^+(x)} f(y) dy \right|^2 dx & \leq  \int_{B_1(x_0)}\left| \fint_{B_R^+(x)} |f(x) - f(y)| dy \right|^2 dx \\
    & \leq \|f\|_{\mathcal{C}^{0, \alpha}(\mathbb{R}^d)}^2 \int_{B_1(x_0)}\left| \fint_{B_R^+(x)} |x-y|^{\alpha} dy \right|^2 dx \\
    & \leq \|f\|_{\mathcal{C}^{0, \alpha}(\mathbb{R}^d)}^2 R^{2 \alpha} |B_1|.
\end{align*}
Using a triangle inequality, we therefore deduce : 
\begin{align*}
    \|f\|_{L^2(B_1(x_0))} &\leq \left\|f - \fint_{B_R^+(.)} f \right\|_{L^2(B_1(x_0))} + \left\|\fint_{B_R^+(.)} |f| \right\|_{L^2(B_1(x_0))} \\
    &\leq R^{\alpha} |B_1|^{1/2} \|\nabla u\|_{\mathcal{C}^{0, \alpha}(\mathbb{R}^d)} + \left\|\fint_{B_R^+(.)} |f| \right\|_{L^2(B_1(x_0))}.
\end{align*}
Since $B_R^+(x)\subset Q+x$, we obtain 
\begin{align*}
    \|f\|_{L^2(B_1(x_0))} 
    & \leq R^{\alpha} |B_1|^{1/2} \|f\|_{\mathcal{C}^{0, \alpha}(\mathbb{R}^d)} + \dfrac{1}{|B_R|} \left\| \mathcal{M}(|f|)\right\|_{L^2_{unif}(\mathbb{R}^d)}.
\end{align*}
Taking the supremum over all $x_0$ yields \eqref{born_norm_unif_ap} which concludes the proof. 
\end{proof}

Lemma \ref{lemme_majoration_L2_unif} is now useful to establish : 

\begin{prop}
\label{regularite_A_P}
There exists a constant $C>0$ such that for every $f \in \left(\mathcal{C}^{0,\alpha}(\mathbb{R}^d)\right)^d$ and $u \in H^1_{loc}(\mathbb{R}^d)$ solution to \eqref{equation_generale_Ap}
with $\nabla u \in \left(\mathcal{C}^{0,\alpha}(\mathbb{R}^d)\right)^d$, we have :
\begin{equation}
    \|\nabla u \|_{\mathcal{C}^{0, \alpha}(\mathbb{R}^d)} \leq C\left( \|\mathcal{M}\left(\left|\nabla u \right| \right) \|_{L^2_{unif}} + \|f\|_{\mathcal{C}^{0, \alpha}(\mathbb{R}^d)} \right). 
\end{equation}
\end{prop}

\begin{proof}
Since $u$ is a solution to equation \eqref{equation_generale_Ap} where the coefficient $a$ satisfies assumption \eqref{hypothèse_a_holder_AP}, we know from \cite[Theorem 5.19 p.87]{MR3099262} (see also \cite[Theorem 3.2 p.88]{giaquinta1983multiple}) there exists a constant $C>0$, such that for every $x_0\in \mathbb{R}^d$, we have : 
\begin{align*}
   \|\nabla u\|_{\mathcal{C}^{0,\alpha}(B_{1}(x_0))} & \leq C \left( \|\nabla u\|_{L^2(B_4(x_0))} + \|f\|_{\mathcal{C}^{0,\alpha}(\mathbb{R}^d)}\right) \leq C \left( C_1\|\nabla u\|_{L^2_{unif}(\mathbb{R}^d)} + \|f\|_{\mathcal{C}^{0,\alpha}(\mathbb{R}^d)}\right),
\end{align*}
where $C_1>0$ depends only on the dimension $d$. Since the right-hand side in the previous inequality is independent of $x_0$, there exists a constant $C_2>0$ such that 
$$\|\nabla u\|_{\mathcal{C}^{0,\alpha}(\mathbb{R}^d)} \leq C_2 \left( \|\nabla u\|_{L^2_{unif}(\mathbb{R}^d)} + \|f\|_{\mathcal{C}^{0,\alpha}(\mathbb{R}^d)}\right).$$

For every $0<R<1$, we use this inequality and Lemma \ref{lemme_majoration_L2_unif} to obtain the existence of a constant $C_3>0$, independent of $R$, u and $f$, such that, 
$$ \|\nabla u\|_{\mathcal{C}^{0,\alpha}(\mathbb{R}^d)} \leq C_3\left(\dfrac{1}{|B_R|}\|\mathcal{M}(|\nabla u|)\|_{L^2_{unif}(\mathbb{R}^d)} + R^{\alpha}\|\nabla u\|_{C^{0,\alpha}(\mathbb{R}^d)} + \|f\|_{C^{0,\alpha}(\mathbb{R}^d)} \right).$$
It remains to choose $R$ such that $R^{\alpha} C_3 < 1$ to conclude.
\end{proof}

\subsection{Well-posedness for (\ref{equation_generale_Ap}) when the coefficient is periodic}
 
We next study equation \eqref{equation_generale_Ap} when the coefficient $a$ is periodic, that is when $\Tilde{a}=0$. For every $f\in \mathcal{A}^p$, we prove the existence and uniqueness of a solution $u$ to :
\begin{equation}
\label{equation_coeff_periodique_A_P}
 -\operatorname{div}(a_{per} \nabla u) = \operatorname{div}(f) \quad \text{on } \mathbb{R}^d,   
\end{equation}
such that $\nabla u \in \left(\mathcal{A}^p\right)^d$. Adapting a method introduced in \cite{blanc2018correctors}, this result is the first step to study \eqref{equation_generale_Ap}. We begin with existence of the solution.

\begin{prop}
\label{existence_cas_periodique_AP}
Assume $a_{per}$ satisfies \eqref{hypothèse_a_elliptic_AP} and \eqref{hypothèse_a_holder_AP}. Let $f\in \left(\mathcal{A}^p\right)^d$ for $p\in]1,d[$. Then, there exists a solution $u \in L^1_{loc}(\mathbb{R}^d)$ to \eqref{equation_coeff_periodique_A_P}
such that $\nabla u \in \left(\mathcal{A}^p\right)^d$. In addition, there exists a constant $C_1>0$ independent of $u$ and $f$ such that 
\begin{equation}
\label{estim_per_AP}
    \|\nabla u\|_{\mathcal{A}^p}  \leq C_1 \|f \|_{\mathcal{A}^p}.
\end{equation}
If we additionally assume that $f \in \left(L^{\infty}(\mathbb{R}^d)\right)^d$, then $\nabla u \in \left(L^2_{unif}(\mathbb{R}^d)\right)^d$ and there exists $C_2>0$ independent of $f$ and $u$ such that : 
\begin{equation}
\label{estim_per_unif_AP}
    \|\nabla u\|_{L^2_{unif}(\mathbb{R}^d)} \leq C_2 \left(\|f\|_{\mathcal{E}^p} + \|f\|_{L^{\infty}(\mathbb{R}^d)}\right).
\end{equation}
\end{prop}

We will need to introduce the Green function $G_{per}$ associated with  $-\operatorname{div}(a_{per} \nabla.)$ on $\mathbb{R}^d$ defined as the unique solution to 
\begin{equation}
    \label{defGreen_rappel}
    \left\{ 
\begin{array}{cc}
   - \operatorname{div}_x \left(a_{per}(x) \nabla_x G_{per}(x,y) \right) = \delta_y (x)  &  \text{in } \mathcal{D}'(\mathbb{R}^d), \\
    \displaystyle \lim_{|x-y| \rightarrow \infty} G_{per}(x,y) = 0. &
\end{array}  
\right.
\end{equation}
In order to define a solution to \eqref{equation_coeff_periodique_A_P}, we will use several pointwise estimates established in \cite[Section 2]{avellaneda1991lp} and satisfied by $G_{per}$ on the whole space $\mathbb{R}^d$. Indeed, we know there exist $C_1>0$ and $C_2>0$ such that for every $x,y \in \mathbb{R}^d$ with $x\neq y$, it holds :
\begin{align}
\label{estimgreen1}
|\nabla_y G_{per}(x,y)| & \leq C_1 \dfrac{1}{|x-y|^{d-1}},\\
\label{estimgreen2}
|\nabla_x \nabla_y G_{per}(x,y)| & \leq C_3 \dfrac{1}{|x-y|^{d}}.
\end{align}

\begin{proof}

\textbf{Step 1 : Existence of a solution}. Corollary \ref{densité_Ap} gives the existence of a sequence $(f_n)_{n \in \mathbb{N}}$ of functions in $\left(L^p(\mathbb{R}^d)\right)^d$ that converges to $f$ in $\left(\mathcal{A}^p\right)^d$. For every $n \in \mathbb{N}$, the results of \cite{avellaneda1991lp} establish the existence of a solution, unique up to an additive constant, $u_n$ in $L^1_{loc}(\mathbb{R}^d)$ to :
\begin{equation}
\label{equation_periodique_L2}
  - \operatorname{div}(a_{per} \nabla u_n) = \operatorname{div}(f_n) \quad \text{on } \mathbb{R}^d,  
\end{equation}
and such that $\nabla u_n \in \left(L^p(\mathbb{R}^d)\right)^d$.
In addition, using the periodicity of $a_{per}$, we apply the operator $\delta_i$ to equation \eqref{equation_periodique_L2}, and we obtain $- \operatorname{div}(a_{per} \delta_i\nabla u_n) = \operatorname{div}(\delta_if_n)$ for every $i \in \{1,...,d\}$.
Since $\delta_i \nabla u_n$ and $\delta_i f_n$ both belong to $\left(L^p(\mathbb{R}^d)\right)^d$, the continuity result of \cite[Theorem A]{avellaneda1991lp} yields the existence of a constant $C_1>0$ independent of $n$ such that $\|\delta \nabla u_n \|_{L^p(\mathbb{R}^d)} \leq C_1 \| \delta f_n \|_{L^p(\mathbb{R}^d)}$.
From Proposition \ref{Corollaire_GNS_LP} we infer the existence of $C_2>0$ independent of $n$ such that : 
\begin{equation}
\label{estimée_suite_approchée}
    \| \nabla u_n \|_{\mathcal{A}^p} \leq C_2 \|\delta \nabla u_n \|_{L^p(\mathbb{R}^d)} \leq C_1C_2 \| \delta f_n \|_{L^p(\mathbb{R}^d)}.
\end{equation}
Likewise, for every $m,n \in \mathbb{N}$, the function $u_n - u_m$ is a solution to $- \operatorname{div}(a_{per}(\nabla u_n - \nabla u_m)) = \operatorname{div}(f_n - f_m)$.
Since $f_n - f_m$ and $\nabla(u_n - u_m)$ both belong to $\left(L^p(\mathbb{R}^d)\right)^d$, we similarly obtain : 
$$\|\nabla u_n - \nabla u_m \|_{\mathcal{A}^p}\leq C_1C_2 \|\delta f_n - \delta f_m\|_{L^p(\mathbb{R}^d)}.$$
Since $(f_n)_{n \in \mathbb{N}}$ converges to $f$ in $\mathcal{A}^p$, it is a Cauchy sequence in this space and the previous inequality shows $(\nabla u_n)_{n \in \mathbb{N}}$ is also a Cauchy sequence in $\left(\mathcal{A}^p\right)^d$. We therefore obtain the existence of $T \in \left(\mathcal{A}^p\right)^d$ such that $\nabla u_n$ converges to $T$ in $\left(\mathcal{A}^p\right)^d$. Using Proposition \ref{sous-suite}, we have that,  up to an extraction, $\nabla u_n$ also converges to $T$ in $\left(L^1_{loc}(\mathbb{R}^d)\right)^d$
and the Schwarz Lemma shows the existence of $u \in L^1_{loc}(\mathbb{R}^d)$ such that $T = \nabla u$. 
Finally, taking the limit when $n\to \infty$ in \eqref{equation_periodique_L2} and \eqref{estimée_suite_approchée}, we obtain that $u$ is a solution to \eqref{equation_coeff_periodique_A_P} such that : 
$$\|\nabla u \|_{\mathcal{A}^p} \leq C_2 \|\delta f \|_{L^p(\mathbb{R}^d)} \leq C_2 \|f\|_{\mathcal{A}^p}.$$ 

\textbf{Step 2 : Proof of estimate \eqref{estim_per_unif_AP}}. We now additionally assume that $f \in \left(L^{\infty}(\mathbb{R}^d)\right)^d$. Corollary \ref{densité_Ap} gives the existence of a sequence $(f_n)_{n \in \mathbb{N}}$ of $\left(L^p(\mathbb{R}^d)\right)^d$ such that $f_n$ converges to $f$ in $\left(\mathcal{A}^p\right)^d$ and such that for every $n\in \mathbb{N}$, we have :
\begin{equation}
\label{majoration_unif_fn}
  \|f_n\|_{L^{\infty}(\mathbb{R}^d)} \leq 2 \|f\|_{L^{\infty}(\mathbb{R}^d)}.  
\end{equation}
Exactly as in step 1, we denote by $u_n$, the unique solution (up to an additive constant) to \eqref{equation_periodique_L2} such that $\nabla u_n \in \left(L^p(\mathbb{R}^d)\right)^d$.  
We fix $x_0 \in \mathbb{R}^d$ and our aim is to show that the norm of $\nabla u_n$ in $L^2(B_1(x_0))$ is uniformly bounded with respect to $n$ and $x_0$. We begin by splitting $\nabla u_n$ in two parts. We write  $\nabla u_n = \nabla u_{n,1} + \nabla u_{n,2}$ where $u_{n,1}$ is the unique solution (up to an additive constant) to $- \operatorname{div}(a_{per}\nabla u_{n,1}) = \operatorname{div}\left(f_n 1_{B_{4 \sqrt{d}}(x_0)}\right)$ on $\mathbb{R}^d$,
such that $\nabla u_{n,1}\in \left(L^2(\mathbb{R}^d)\cap L^p(\mathbb{R}^d)\right)^d$ and $u_{n,2}$ is the unique solution (again up to an additive constant) to $- \operatorname{div}(a_{per}\nabla u_{n,2}) = \operatorname{div}\left(f_n\left(1- 1_{B_{4 \sqrt{d}}(x_0)}\right)\right)$ on $\mathbb{R}^d$
such that $\nabla u_{n,2}\in \left(L^p(\mathbb{R}^d)\right)^d$.
Since $f_n 1_{B_{4 \sqrt{d}}(x_0)}$ belongs to $\left(L^2(\mathbb{R}^d) \cap L^p(\mathbb{R}^d)\right)^d$, the existence of $u_{n,1}$ is established in \cite[Theorem A]{avellaneda1991lp}, and we have the existence of a constant $C_1>0$ independent of $n$ and $x_0$ such that : 
\begin{equation}
\label{L2_estimate_nabla_u_n_delta_a}
    \|\nabla u_{n,1}\|_{L^2(\mathbb{R}^d)} \leq C_1 \|f_n 1_{B_{4 \sqrt{d}}(x_0)}\|_{L^2(\mathbb{R}^d)}. 
\end{equation}
Similarly, since $f_n\left(1- 1_{B_{4 \sqrt{d}}(x_0)}\right)$ belongs to $\left(L^p(\mathbb{R}^d)\right)^d$, the existence of $u_{n,2}$ is also given in \cite[Theorem A]{avellaneda1991lp} and we have
\begin{equation}
  \nabla u_{n,2} = \int_{\mathbb{R}^d} \nabla_x \nabla_y G_{per}(.,y)\left(f_n\left(1- 1_{B_{4 \sqrt{d}}(x_0)}\right)\right)(y)dy.  
\end{equation}
We note that the equality $\nabla u_n = \nabla u_{n,1} + \nabla u_{n,2}$ holds as a consequence of the uniqueness of a solution to \eqref{equation_periodique_L2} with a gradient in $\left( L^p(\mathbb{R}^d) \right)^d$. 
We next respectively estimate the norm of $\nabla u_{n,1}$ and $\nabla u_{n,2}$ in $L^2(B_1(x_0))$.
First, using \eqref{L2_estimate_nabla_u_n_delta_a}, we have 
\begin{equation}
    \label{majoration_L2_Un1}
     \|\nabla u_{n,1}\|_{L^2(B_1(x_0))} \leq C_1 \left\|f_n 1_{B_{4 \sqrt{d}}(x_0)}\right\|_{L^2(\mathbb{R}^d)}\leq C_1 |B_{4 \sqrt{d}}|^{1/2} \|f_n\|_{L^{\infty}(\mathbb{R}^d)}.
\end{equation}

In order to estimate the $L^2$-norm of $\nabla u_{n,2}$, we use the behavior \eqref{estimgreen2} of the Green function $G_{per}$ and we obtain the existence of a constant $C>0$ such that for every $x \in B_{1}(x_0)$, we have 
$$|\nabla u_{n,2}(x)| \leq C \int_{\mathbb{R}^d\setminus{B_{4 \sqrt{d}}(x_0)}} \frac{1}{|x-y|^d} |f_n(y)| dy.$$
Since $|Q|=1$, using a change of variables, we have
\begin{align*}
I(x):= \int_{\mathbb{R}^d \setminus{B_{4\sqrt{d}}(x_0)}} \frac{1}{|x-y|^d} |f_n(y)| dy 
    & = \int_Q \int_{\mathbb{R}^d \setminus{B_{4\sqrt{d}}(z +x-x_0)}} \frac{1}{|y+z|^d} |f_n(x-y-z)| dy dz.
\end{align*}
We note that for every $z \in Q$, $|z|\leq \sqrt{d}$ and $|x-x_0|\leq 1\leq \sqrt{d}$, and it follows
$$I(x) \leq \int_Q \int_{\mathbb{R}^d \setminus{B_{2\sqrt{d}}}} \frac{1}{|y+z|^d} |f_n(x-y-z)| dy dz.$$
Next, for every $y \in \mathbb{R}^d\setminus{B_{2\sqrt{d}}}$, we have $\displaystyle |z| \leq \frac{1}{2} |y|$ and we use a triangle inequality to deduce $|z + y| \geq |y|-|z| \geq \frac{1}{2} |y|$, and 
\begin{align*}
    \int_Q \int_{\mathbb{R}^d \setminus{B_{2\sqrt{d}}}} \frac{1}{|y+z|^d} |f_n(x-y-z)| dy dz 
    & \leq 2^d \int_{\mathbb{R}^d \setminus{B_{2\sqrt{d}}}} \frac{1}{|y|^d} \int_{Q+y} |f_n(x-z)| dz dy.
\end{align*}
We use the H\"older inequality and obtain, 
\begin{align*}
  I(x) 
    & \leq  2^d \left(\int_{\mathbb{R}^d\setminus{B_{2 \sqrt{d}}}} \frac{1}{|y|^{(p^*)' d}}dy \right)^{1/(p^*)'} \|\mathcal{M}(|f_n|)\|_{L^{p^*}(\mathbb{R}^d)}.
\end{align*}
Here, we have denoted by $(p^*)'$ the conjugate Lebesgue exponent associated with $p^*$.
We have finally proved that for every  $x \in B_1(x_0)$, 
\begin{equation}
\label{inegalite_u_n_2_delta_a}
    |\nabla u_{n,2}(x)| \leq C I(x) \leq A \|f_n\|_{\mathcal{E}^p},
\end{equation}
where $\displaystyle A = C 2^d \left(\int_{\mathbb{R}^d\setminus{B_{2 \sqrt{d}}}} \frac{1}{|y|^{(p^*)' d}}dy \right)^{1/(p^*)'}$ is clearly independent of $x_0$ and $f$. We integrate \eqref{inegalite_u_n_2_delta_a} on $B_1(x_0)$ and we obtain the existence of a constant $C_2>0$, independent of $n$, $x_0$ and $f$, and such that :
\begin{equation}
\label{majoration_unif_L2_un2}
    \|\nabla u_{n,2}\|_{L^2(B_1(x_0))} \leq C_2 \|f_n\|_{\mathcal{E}^p}.
\end{equation}
Since $\nabla u_n = \nabla u_{n,1} + \nabla u_{n,2}$, we use \eqref{majoration_L2_Un1} and  \eqref{majoration_unif_L2_un2} : 
$$\|\nabla u_{n}\|_{L^2(B_1(x_0))} \leq \|\nabla u_{n,1}\|_{L^2(B_1(x_0))} + \|\nabla u_{n,2}\|_{L^2(B_1(x_0))} \leq C \left(\|f_n\|_{L^{\infty}(\mathbb{R}^d)} + \|f_n\|_{\mathcal{E}^p}\right),$$
where $C>0$ is independent of $n$, $x_0$ and $f$. 
Since the previous inequality holds for every $x_0 \in \mathbb{R}^d$, $\nabla u_n$ is bounded in $\left(L^2_{loc}(\mathbb{R}^d)\right)^d$ and, up to an extraction, it weakly converges to a function $v \in \left(L^2_{loc}(\mathbb{R}^d)\right)^d$. We recall that $\nabla u_n$ also converges to $\nabla u$ in $\left(L^1_{loc}(\mathbb{R}^d)\right)^d$, it follows that $v = \nabla u$. In addition the $L^2$ norm being lower semi-continuous, for every $x_0 \in \mathbb{R}^d$ we obtain : 
$$ \|\nabla u\|_{L^2(B_1(x_0))} \leq \liminf_{n \to \infty} \|\nabla u_n\|_{L^2(B_1(x_0))}\leq C \left(\|f\|_{L^{\infty}(\mathbb{R}^d)} + \|f\|_{\mathcal{E}^p}\right).$$
We finally take the supremum over all the points $x_0 \in \mathbb{R}^d$ and we obtain \eqref{estim_per_unif_AP}. 
\end{proof}

\begin{remark}
\label{remark_counter_p1}
For $p=1$, estimate \eqref{estim_per_AP} stated in Proposition \ref{existence_cas_periodique_AP} does not hold. For $d=2$ say, we can consider $a_{per}=I_2$ and, $ \displaystyle f(x):= \left(-\dfrac{x_2}{|x|^2\ln(|x|)^2}\arctan\left(\frac{x_2}{x_1}\right), \dfrac{x_1}{|x|^2\ln(|x|)^2}\arctan\left(\frac{x_2}{x_1}\right) \right)$, where we have denoted $x= (x_1, x_2)$.
A solution to \eqref{equation_coeff_periodique_A_P} is $u(x) = \ln(\ln(|x|))$. When $|x| \to \infty$  we can show that $|f(x)| = O\left(\dfrac{1}{|x| \ln(|x|)^2}\right)\in L^2(\mathbb{R}^2) = L^{1^*}(\mathbb{R}^2)$ and
$$|\delta f_1(x)| = O\left(\dfrac{1}{|x|^2 \ln(|x|)^2}\right)\in L^1(\mathbb{R}^2), \qquad   |\delta f_2(x)| = O\left(\dfrac{1}{|x|^2 \ln(|x|)^2}\right) \in L^1(\mathbb{R}^2). $$
Consequently $f$ belongs to $\left(\mathcal{A}^1\right)^2$. However $|\delta \partial_1 u(x_1,x_2)| \sim \dfrac{C}{|x|^2\ln{|x|}}\notin L^1(\mathbb{R}^2)$ and $\nabla u$ does not belong to $\left(\mathcal{A}^1\right)^2$. This is of course related to the fact that the operator $-\nabla \Delta^{-1} \operatorname{div}$ is not continuous from $\left(L^1(\mathbb{R}^d)\right)^d$ to $\left(L^1(\mathbb{R}^d)\right)^d$. We only have continuity from $L^1$ to $\operatorname{weak}-L^1$ (see \cite[Section 7.3]{meyer1997wavelets} for the details). 
\end{remark}

\begin{remark}
The property $f \in \left(\mathcal{E}^p\right)^d$ in \eqref{estim_per_unif_AP} is required to obtain the uniform estimate in $L^2_{unif}(\mathbb{R}^d)$ satisfied by $\nabla u$. When $f$ only belongs to $\left(L^{\infty}(\mathbb{R}^d)\right)^d$, \eqref{equation_coeff_periodique_A_P} may possibly have no solution with a gradient in $L^2_{unif}(\mathbb{R}^d)$. Consider indeed for $d=2$, $f(x) = f(x_1,x_2) = \left(-2 \dfrac{x_1^2}{|x|^2}, -2 \dfrac{x_1x_2}{|x|^2}\right) \in \left(L^{\infty}(\mathbb{R}^2)\right)^2$. Then $u(x) =  x_1\ln(|x|)$ satisfies $-\Delta u = \operatorname{div}(f)$, while $\nabla u \notin \left(L^2_{unif}(\mathbb{R}^2)\right)^2$ and has logarithmic growth.
\end{remark}

We next deal with uniqueness of the solution.

\begin{lemme}
\label{unicité_periodique_AP}
Assume $a_{per}$ satisfies \eqref{hypothèse_a_elliptic_AP} and \eqref{hypothèse_a_holder_AP}. Let $u \in L^1_{loc}(\mathbb{R}^d)$ be a solution in $\mathcal{D}'(\mathbb{R}^d)$ to : 
\begin{equation}
\label{equation_unicite_per}
    - \operatorname{div}(a_{per} \nabla u) = 0 \quad \text{on } \mathbb{R}^d,
\end{equation}
such that $\nabla u \in \left(\mathcal{A}^p\right)^d$ for $p\in ]1,d[$. Then $\nabla u = 0$. 
\end{lemme}

\begin{proof}
For every $i \in \{1,...,d\}$, we consider a translation by $e_i$ of equation \eqref{equation_unicite_per} and we subtract it from the original equation. The periodicity of $a_{per}$ implies
$ - \operatorname{div}(a_{per}  \nabla \delta_i u) = 0$. Since we have assumed $\nabla u \in \left(\mathcal{A}^p\right)^d$, we know that $\nabla \delta_i  u$ belongs to $\left(L^p(\mathbb{R}^d)\right)^d$ and the uniqueness result established in \cite[Proposition 2.1]{blanc2018correctors} for solutions with gradient in $ \left(L^p(\mathbb{R}^d)\right)^d$ therefore implies that $\nabla \delta_i u  = 0 $. It follows that $\nabla u$ is  $Q$-periodic and, consequently, the function $\mathcal{M}(|\nabla u|)$ is constant. Since by assumption $\mathcal{M}\left(\left|\nabla u\right| \right)$ belongs to $L^{p^*}(\mathbb{R}^d)$, we obtain $\mathcal{M}\left(\left|\nabla u\right| \right) = 0$ which shows that $\nabla u = 0$. 
\end{proof}


\begin{corol}
\label{régularité_AP_coeff_periodique}
Let $p\in ]1,d[$. There exists $C>0$ such that for every $f\in \left(\mathcal{A}^p\cap \mathcal{C}^{0,\alpha}(\mathbb{R}^d)\right)^d$ and $u$ solution to \eqref{equation_coeff_periodique_A_P}
such that $\nabla u \in \left(\mathcal{A}^p\right)^d$, we have : 
\begin{equation}
    \|\nabla u\|_{\mathcal{C}^{0, \alpha}(\mathbb{R}^d)} \leq C \left(\| \mathcal{M}(|f|)\|_{L^{p^*}(\mathbb{R}^d)} +  \|f\|_{\mathcal{C}^{0, \alpha}(\mathbb{R}^d)}\right). 
\end{equation}
\end{corol}

\begin{proof}
Using both Proposition \ref{existence_cas_periodique_AP} and the uniqueness of Lemma \ref{unicité_periodique_AP}, we know that $\nabla u$ belongs to $\left(L^2_{unif}(\mathbb{R}^d)\right)^d$ and the elliptic regularity theory (see for instance the results of \cite[Theorem 5.19 p.87]{MR3099262}) implies that $\nabla u \in \left(\mathcal{C}^{0,\alpha}(\mathbb{R}^d)\right)^d.$ The inequality is therefore a direct consequence of Proposition~\ref{regularite_A_P} and estimate \eqref{estim_per_unif_AP}.
\end{proof}

In the sequel of this section, we study the specific case where $1<p<\frac{d}{2}$, that is when $p^*<d$. We can then show some additional properties satisfied by $u$ solution to \eqref{equation_coeff_periodique_A_P}. We successively show, respectively in Lemma \ref{Lemme_forme_u_per_ap} and Lemma \ref{lemme_borne_u_per_ap} that, up to an additive constant $u$ belongs to $\mathcal{E}^{p^*}$ and it is uniformly bounded as soon as $f$ belongs to $\left(L^{\infty}(\mathbb{R}^d)\right)^d$. To this end, we first need to recall the Hardy-Littlewood-Sobolev (see for instance \cite[Theorem 7.25 p. 162]{MR3099262}).

\begin{prop}[Hardy-Littlewood-Sobolev inequality]
Let $0<\alpha<d$. We define $I(f) := |x|^{-\alpha} \ast f$ (where $\ast$ is the convolution operator).
Let $p,q >1$ such that $\displaystyle 1 + \frac{1}{q} = \frac{\alpha}{d} + \frac{1}{p}$. Then, there exists $C>0$ such that for every $f\in L^p(\mathbb{R}^d)$ we have : 
 $$ \|I(f)\|_{L^q(\mathbb{R}^d)} \leq C \|f\|_{L^p(\mathbb{R}^d)}.$$
\end{prop}

We may now prove that the unique (up to an additive constant) solution $u$ to \eqref{equation_coeff_periodique_A_P} such that $\nabla u \in \left(\mathcal{A}^p\right)^d$ can be made explicit using the Green function $G_{per}$. 

\begin{lemme}
\label{Lemme_forme_u_per_ap}
Assume that $1<p<\frac{d}{2}$ and let $f \in \left(\mathcal{A}^p\right)^d$, then the solution $u$ (unique up to an additive constant) to \eqref{equation_coeff_periodique_A_P} such that $\nabla u \in \left(\mathcal{A}^p\right)^d$ is given by 
\begin{equation}
\label{def_u_integrale_AP}
    u = Tf := \int_{\mathbb{R}^d} \nabla_y G_{per}(.,y) f(y) dy.
\end{equation}
In addition, we have $u \in \mathcal{E}^{p^*}$.
\end{lemme}

\begin{proof}
We begin by showing that the function $Tf$ given by \eqref{def_u_integrale_AP} is well-defined in $L^1_{loc}(\mathbb{R}^d)$ if $f \in \left(\mathcal{A}^p\right)^d$ and satisfies $\mathcal{M}(|Tf|) \in L^{p^{**}}(\mathbb{R}^d)$. Using estimate \eqref{estimgreen1}, we know there exists a constant  $C>0$ such that for every $x \in \mathbb{R}^d$, we have : 
\begin{equation*}
    |Tf(x)| \leq C \int_{\mathbb{R}^d} \frac{1}{|y|^{d-1}} |f|(x-y)dy. 
\end{equation*}
For every $z\in \mathbb{R}^d$, we integrate the previous inequality with respect to $x \in Q+z$ and we use the Fubini Theorem to obtain : 
$$\mathcal{M}\left(\left|Tf\right| \right)(z) \leq C \int_{\mathbb{R}^d} \frac{1}{|y|^{d-1}} \int_{Q + z} |f|(x-y) dx dy = C \dfrac{1}{|.|^{d-1}}*\mathcal{M}(|f|).$$
Since $\displaystyle 1 +  \frac{1}{p^{**}} = \frac{d-1}{d} + \frac{1}{p^*}$, the Hardy-Littlewood-Sobolev inequality therefore shows the existence of a constant $C>0$, such that : 
\begin{equation}
\label{hardy_inequality_mean}
    \|\mathcal{M}\left(\left|Tf\right| \right)\|_{L^{p^{**}}(\mathbb{R}^d)} \leq C \|\mathcal{M}\left(\left|f\right| \right)\|_{L^{p^{*}}(\mathbb{R}^d)}.
\end{equation}
In particular $\mathcal{M}\left(\left|Tf\right| \right)$ belongs to $L^{p^{**}}(\mathbb{R}^d)$ and is therefore finite for almost every $x \in \mathbb{R}^d$. We deduce that $Tf$ is well-defined in $L^1_{loc}(\mathbb{R}^d)$.

We next show that, up to an additive constant, we have $u=Tf$. We first recall that in the proof of Proposition \ref{existence_cas_periodique_AP}, we have considered a sequence $(f_n)_{n \in \mathbb{N}}$ of functions in $\left(L^p(\mathbb{R}^d)\right)^d$ that converges to $f$ in $\mathcal{A}^p$ and an associated sequence of functions $(u_n)_{n \in \mathbb{N}}$, solutions to \eqref{equation_periodique_L2} with a gradient in $\left(L^p(\mathbb{R}^d)\right)^d$, that converges in $L^1_{loc}(\mathbb{R}^d)$ to $u$ solution to \eqref{equation_coeff_periodique_A_P} such that $\nabla u \in \left(\mathcal{A}^p\right)^d$. We claim that $u_n$ is actually defined, up to an additive constant, by $\displaystyle u_n = Tf_n := \int_{\mathbb{R}^d} \nabla _y G_{per}(.,y).f_n(y) dy$.
Using estimate \eqref{estimgreen1}, we indeed know there exists a constant $C>0$ such that for every $x \in \mathbb{R}^d$,
$$\left|Tf_n\right| \leq C \int_{\mathbb{R}^d} \frac{1}{|x-y|^{d-1}} |f_n|(y)dy.$$
Since $f_n$ belongs to $\left(L^p(\mathbb{R}^d)\right)^d$ for every $n \in \mathbb{N}$ and $\displaystyle 1 + \frac{1}{p^{*}} = \frac{d-1}{d} + \frac{1}{p}$, we know from the Hardy-Littlewood-Sobolev inequality that $Tf_n$ belongs to $L^{p^*}(\mathbb{R}^d)$. In addition, the results established in \cite[Theorem A]{avellaneda1991lp} implies that $Tf_n$ is a solution to \eqref{equation_periodique_L2} such that $\nabla Tf_n \in \left(L^p(\mathbb{R}
^d)\right)^d$. A solution to \eqref{equation_periodique_L2} with a gradient in $\left(L^p(\mathbb{R}^d)\right)^d$ being unique up to an additive constant, we conclude that $\nabla u_n = \nabla Tf_n$. We therefore obtain that $u_n = Tf_n$ up to an additive constant. Exactly as in the proof of inequality \eqref{hardy_inequality_mean}, the Hardy-Littlewood-Sobolev inequality gives the existence of a constant $C>0$ independent of $n$ such that  
$$\|\mathcal{M}\left(\left|u_n - Tf\right| \right)\|_{L^{p^{**}}(\mathbb{R}^d)} = \|\mathcal{M}\left(\left|Tf_n - Tf\right| \right)\|_{L^{p^{**}}(\mathbb{R}^d)} \leq C \|\mathcal{M}\left(\left|f_n-f\right| \right)\|_{L^{p^{*}}(\mathbb{R}^d)} \underset{n\rightarrow \infty}{\longrightarrow} 0.$$ 

Finally, using Proposition \ref{sous-suite} we know that $u_n$ converges to $T_f$ in $L^1_{loc}(\mathbb{R}^d)$ up to an extraction. The uniqueness of the limit in $L^1_{loc}(\mathbb{R}^d)$ allows to conclude that $u=T_f$. 
\end{proof}

\begin{lemme}
\label{lemme_borne_u_per_ap}
Assume $1<p< \frac{d}{2}$ and let $f \in \left(\mathcal{A}^p\cap L^{\infty}(\mathbb{R}^d)\right)^d$. Then, the function $u$ defined by \eqref{def_u_integrale_AP} belongs to $L^{\infty}(\mathbb{R}^d)$. 
\end{lemme}

\begin{proof}
We begin by considering $x_0 \in \mathbb{R}^d$ and we split $u(x_0)$ in two parts as follows :
\begin{align*}
    u(x_0) & = \int_{\mathbb{R}^d\setminus{B_{3 \sqrt{d}}(x_0)}} \nabla_y G_{per}(x_0,y) f(y) dy + \int_{B_{3 \sqrt{d}}(x_0)} \nabla_y G_{per}(x_0,y) f(y) dy  = I_1(x_0) + I_2(x_0).
\end{align*}
We want to bound both $|I_1(x_0)|$ and $|I_2(x_0)|$ uniformly with respect to $x_0$. Estimate~\eqref{estimgreen1} gives $C>0$ independent of $x_0$ such that : 
$$|I_1(x_0)| \leq C \int_{\mathbb{R}^d\setminus{B_{3 \sqrt{d}}(x_0)}} \frac{1}{|x_0-y|^{d-1}} |f(y)| dy.$$
Since $|Q|=1$, we have by integrating the previous inequality : 
\begin{align*}
  |I_1(x_0)| & \leq  C \int_{Q} \int_{\mathbb{R}^d\setminus{B_{3 \sqrt{d}}(x_0-z)}} \frac{1}{|x_0-y-z|^{d-1}} |f(y+z)| dy dz.
\end{align*}
For every $z\in Q$ and $y\in \mathbb{R}^d \setminus{B_{3 \sqrt{d}}(x_0-z)}$, since $|z| < \sqrt{d}$ and $|x_0 - y - z| \geq 3\sqrt{d}$, we have $|x_0-y| = |x_0 - y-z +z| \geq 2\sqrt{d}$.
It follows that $\mathbb{R}^d\setminus{B_{3\sqrt{d}}(x_0-z)} \subset{\mathbb{R}^d\setminus{B_{2\sqrt{d}}(x_0)}}$. We also have $\displaystyle |z| \leq \sqrt{d} \leq \frac{1}{2} |x_0-y|$ which gives  $|x_0 - y - z| \geq \frac{1}{2}|x_0 - y|$.
Using respectively the Fubini theorem and the H\"older inequality, we deduce : 
\begin{align*}
   |I_1(x_0)| & \leq 2^{d-1} C \int_{Q} \int_{\mathbb{R}^d\setminus{B_{2\sqrt{d}}(x_0)}} \frac{1}{|x_0-y|^{d-1}} |f(y+z)| dy dz \\ 
   &\leq2^{d-1}C  \left(\int_{\mathbb{R}^d\setminus{B_{2\sqrt{d}}}}  \frac{1}{|y|^{(d-1)(p^*)'}}dy\right)^{1/(p^*)'}\|\mathcal{M}\left(\left|f\right| \right)\|_{L^{p^*}(\mathbb{R}^d)}. 
\end{align*}
Here we have denoted by $\displaystyle(p^*)'=\frac{pd}{d(p-1)+p}$, the conjugate exponent associated with $p^*$. 
The integral of the right-hand term being finite as soon as $\displaystyle (d-1)\frac{pd}{d(p-1) + p}>d$, that is as soon as $p<\dfrac{d}{2}$, we have finally bounded $|I_1(x_0)|$ uniformly with respect to $x_0$. 

Next, in order to bound $|I_2(x_0)|$, we again use \eqref{estimgreen1} : 
\begin{align*}
    |I_2(x_0)| & \leq C \int_{B_{3 \sqrt{d}}(x_0)} \frac{1}{|x_0-y|^{d-1}} |f(y)| dy \leq C \left(\int_{B_{3 \sqrt{d}}} \frac{1}{|y|^{d-1}} dy\right) \|f\|_{L^{\infty}(\mathbb{R}^d)} < \infty.
\end{align*}
The right-hand side in the latter inequality being independent of $x_0$, we conclude the proof. 
\end{proof}

\subsection{Well posedness in the non-periodic setting}

In this section we return to the non-periodic problem \eqref{equation_generale_Ap}, 
when $a=a_{per} + \Tilde{a}$ and the perturbation $\Tilde{a} \in \left(\mathcal{A}^p\right)^d$ of the periodic geometry does not necessarily vanish.  We assume it satisfies the regularity assumption \eqref{hypothèse_a_holder_AP}. We again adapt a method introduced in \cite{blanc2018correctors} which consists to, first, establish the continuity of operator $\nabla\left(-\operatorname{div}a\nabla\right)^{-1}\operatorname{div}$ from $\left(\mathcal{A}^p\cap \mathcal{C}^{0,\alpha}(\mathbb{R}^d)\right)^d$ to $\left(\mathcal{A}^p\cap \mathcal{C}^{0,\alpha}(\mathbb{R}^d)\right)^d$, and, second, to use both this continuity result and a connectedness argument to extend the results established in the periodic case $a = a_{per}$ to the general case. In order to show the continuity result (established in Lemma~\ref{estimee_A_priori_Ap} below), we need to first introduce a preliminary result when the perturbation $\Tilde{a}$ is sufficiently small and next a uniqueness result regarding the solutions $u$ to \eqref{equation_generale_Ap} such that $\nabla u \in \left(\mathcal{A}^p\right)^d$, respectively in Lemma \ref{point_fixe_LP} and Lemma \ref{unicite_generale_AP}.

\begin{lemme}
\label{point_fixe_LP}
Let $a_{per}$ be a $Q$-periodic matrix-valued function satisfying \eqref{hypothèse_a_elliptic_AP} and \eqref{hypothèse_a_holder_AP}. Then, for every $r \in ]1, +\infty[$, there exists $\varepsilon_0>0$ such that for every $0<\varepsilon<\varepsilon_0$, every $f \in \left(L^r(\mathbb{R}^d)\right)^d$ and every matrix-valued coefficient $\Tilde{a} \in \left(L^{\infty}(\mathbb{R}^d)\right)^{d \times d}$ satisfying $\|\Tilde{a}\|_{L^{\infty}(\mathbb{R}^d)} < \varepsilon$, equation \eqref{equation_generale_Ap} with $a = a_{per} + \Tilde{a}$ admits a unique (up to an additive constant) solution $u$ such that $\nabla u \in \left(L^r(\mathbb{R}^d)\right)^d$. 
\end{lemme}

\begin{proof}
We begin by remarking that the existence and uniqueness of such a solution is equivalent to the existence and uniqueness of a solution $u$ to  $- \operatorname{div}(a_{per}\nabla u) = \operatorname{div}(f + \Tilde{a}\nabla u)$.
We apply a fixed-point method on $\mathbb{R}^d$, considering $(u_n)_{n \in \mathbb{N}}$ defined by $u_0 = 0$ and for every $n\in \mathbb{N}$, $u_{n+1}$ is solution to : 
\begin{equation}
\label{equation_point_fixe_AP}
 - \operatorname{div}(a_{per}\nabla u_{n+1}) = \operatorname{div}(f + \Tilde{a}\nabla u_{n}) \quad \text{on } \mathbb{R}^d,  
\end{equation}
such that $\nabla u_{n+1} \in \left(L^r(\mathbb{R}^d)\right)^d$. Since, for every $n \in \mathbb{N}$, the function $F_n := f + \Tilde{a}\nabla u_{n}$ belongs to $\left(L^r(\mathbb{R}^d)\right)^d$, the results of \cite[Theorem A]{avellaneda1991lp} show the sequence $u_n$ is well-defined. Since, likewise for every $n\in \mathbb{N}^*$, the function $u_{n+1} - u_n$ is solution to $- \operatorname{div}(a_{per}( \nabla (u_{n+1} - u_n))) = \operatorname{div}(\Tilde{a}(\nabla u_n - \nabla u_{n-1}))$,
the result of \cite[Theorem A]{avellaneda1991lp} also yields a constant $C>0$ independent of $\Tilde{a}$ and $n$ such that 
\begin{equation}
\label{majoration_contractante}
   \|\nabla u_{n+1} - \nabla u_n\|_{L^r(\mathbb{R}^d)} \leq C \|\Tilde{a}(\nabla u_n - \nabla u_{n-1})\|_{L^r(\mathbb{R}^d)} \leq C \|\Tilde{a}\|_{L^{\infty}(\mathbb{R}^d)} \|\nabla u_{n} - \nabla u_{n-1}\|_{L^r(\mathbb{R}^d)}. 
\end{equation}
Therefore, if 
\begin{equation}
\label{majoration_a_petit}
    \|\Tilde{a}\|_{L^{\infty}(\mathbb{R}^d)} < \frac{1}{C},
\end{equation}
the sequence $\nabla u_n$ is a Cauchy sequence in $\left(L^r(\mathbb{R}^d)\right)^d$ and it converges to a gradient $\nabla u$ in $\left(L^r(\mathbb{R}^d)\right)^d$. Passing to the limit in the distribution sense in \eqref{equation_point_fixe_AP}, we obtain that $\nabla u$ is solution to \eqref{equation_generale_Ap}. 
To prove uniqueness, we consider $u^1$ and $u^2$ two solutions to \eqref{equation_generale_Ap} such that $\nabla u^1$ and $\nabla u^2$ belongs to $\left(L^r(\mathbb{R}^d)\right)^d$ and we have that $u^1-u^2$ is solution to
$-\operatorname{div}(a_{per} (\nabla( u^1-u^2))) = \operatorname{div}(\Tilde{a}(\nabla u^1 - \nabla u^2)).$
Estimate \eqref{majoration_contractante} implies  
$$ \|\nabla u^1 - \nabla u^2\|_{L^r(\mathbb{R}^d)} \leq C \|\Tilde{a}\|_{L^{\infty}(\mathbb{R}^d)} \|\nabla u^1 - \nabla u^2\|_{L^r(\mathbb{R}^d)},$$
which, given \eqref{majoration_a_petit}, shows $\nabla u^1 = \nabla u^2$. 
\end{proof}

\begin{lemme}
\label{unicite_generale_AP}
Let $a_{per} \in  \left(L^2_{per}(\mathbb{R}^d)\right)^{d \times d}$ and $\Tilde{a} \in \left(\mathcal{A}^p\right)^{d\times d}$ for $p\in]1,d[$.  Assume that $a = a_{per} + \Tilde{a}$ satisfies \eqref{hypothèses1}, \eqref{hypothèse_a_elliptic_AP} and \eqref{hypothèse_a_holder_AP}. Let $u\in L^1_{loc}(\mathbb{R}^d)$ solution in $\mathcal{D}'(\mathbb{R}^d)$ to : 
\begin{equation}
\label{equation_unicité_general}
    - \operatorname{div}(a \nabla u) = 0 \quad \text{on } \mathbb{R}^d,
\end{equation}
such that $\nabla u \in \left(\mathcal{A}^p\cap L^{\infty}(\mathbb{R}^d)\right)^d$. Then $\nabla u = 0$.
\end{lemme}


\begin{proof}

\textbf{Step 1 : Truncation of $\Tilde{a}$}. For every $R>0$, we consider $\chi_R \in \mathcal{D}'(\mathbb{R}^d)$ a non-negative function such that $\text{Supp} (\chi_R) \subset B_{R+1}$, $\chi_{R_{|B_{R}}} \equiv 1$, $\|\chi_R \|_{L^{\infty}(\mathbb{R}^d)} =1$ and $\|\nabla \chi_R \|_{L^{\infty}(\mathbb{R}^d)} \leq C_0$,
where $C_0>0$ is a constant independent of $R$. In the sequel, we denote $\Tilde{a}_R = \chi_R \Tilde{a}$ and $\Tilde{a}_R^C = (1 - \chi_R) \Tilde{a}$.
We next consider the following equation : 
\begin{equation}
    \label{equationliouvillemodifiée}
    -\operatorname{div}((a_{per} + \Tilde{a}_R^C) \nabla v) = \operatorname{div}(\Tilde{a}_R \nabla u) \quad \text{on } \mathbb{R}^d.
    \end{equation}
Since $u$ is solution to \eqref{equation_unicité_general}, $v \equiv u$ is clearly solution to \eqref{equationliouvillemodifiée}. 

\textbf{Step 2 : Study of a particular solution to \eqref{equationliouvillemodifiée}}. Since $\Tilde{a}_R$ is compactly supported and $\nabla u \in \left(L^{\infty}(\mathbb{R}^d)\right)^d$, the function $\Tilde{a}_R \nabla u$ belongs to $\left(L^p(\mathbb{R}^d)\cap L^2(\mathbb{R}^d)\right)^d$. In addition, using Corollary \ref{comportement_asymptotique_AP}, we know that $\Tilde{a}(x)$ converges to 0 when $|x|\to \infty$ and for every $\varepsilon>0$, there exists $R_0>0$ such that for every $R>R_0$, we have : 
\begin{equation}
\label{hyp_rayon_a_petit}
   \left\|\Tilde{a}_R^C\right\|_
{L^{\infty}(\mathbb{R}^d)} < \varepsilon. 
\end{equation}
Thus, using Lemma \ref{point_fixe_LP}, we obtain that for every $R$ large enough, there exists a solution $v_R$ to~\eqref{equationliouvillemodifiée} such that $\nabla v_R \in \left(L^p(\mathbb{R}^d) \cap L^2(\mathbb{R}^d)\right)^d$. In addition, since $\nabla v_R$ belongs to $\left(L^2(\mathbb{R}^d)\right)^d$, we have for every $x\in \mathbb{R}^d$ :
$$\int_{Q+x} |\nabla v_R| \leq \left(\int_{Q+x} |\nabla v_R|^2\right)^{1/2} \leq \|\nabla v_R\|_{L^2(\mathbb{R}^d)}.$$
Consequently, $\mathcal{M}(|\nabla v_R|)$ is uniformly bounded with respect to $x$ and belongs to $L^2_{unif}(\mathbb{R}^d)$. Since $\Tilde{a}_R \nabla u \in \left(\mathcal{C}^{0, \alpha}(\mathbb{R}^d)\right)^d$, the regularity result of Proposition \ref{regularite_A_P} gives that $\nabla v_R$ belongs to $\left(\mathcal{C}^{0,\alpha}(\mathbb{R}^d)\right)^d$. 
We next prove the existence of $R>0$ such that $\nabla u = \nabla v_{R}$.  

\textbf{Step 3 : Existence of $R$ such that $\nabla u = \nabla v_{R}$}. We know that $\nabla u \in \left(L^{\infty}(\mathbb{R}^d)\right)^d$ and Proposition \ref{regularite_A_P} therefore shows that $\nabla u$ belongs to $\left(\mathcal{C}^{0,\alpha}(\mathbb{R}^d)\right)^d$.
In the sequel, we denote $w = u-v_R$. Since $\nabla v_R \in \left(L^p(\mathbb{R}^d) \right)^d \subset \left( \mathcal{A}^p \right)^d$, we have $\nabla w \in \left(\mathcal{A}^p\right)^d$. In addition, $w$ is solution to 
$- \operatorname{div}((a_{per} + \Tilde{a}_R^C) \nabla w) = 0$
or equivalently, a solution to : 
\begin{equation}
\label{equation_liouville2_R}
    - \operatorname{div}(a_{per} \nabla w) = \operatorname{div}(\Tilde{a}^C_R \nabla w) \quad \text{on } \mathbb{R}^d.
\end{equation}
We have $\Tilde{a}^C_R \in \left(\mathcal{A}^p \cap \mathcal{C}^{0, \alpha}(\mathbb{R}^d)\right)^{d\times d}$ and $\nabla w \in \left(\mathcal{A}^p \cap \mathcal{C}^{0,\alpha}(\mathbb{R}^d)\right)^d$, a short calculation allows to show that $\Tilde{a}^C_R \nabla w$ also belongs to $\left(\mathcal{A}^p\cap \mathcal{C}^{0, \alpha}(\mathbb{R}^d)\right)^d$. We next remark that for every $\alpha \in ]0,1[$, we have $\mathcal{C}^{0, \alpha}(\mathbb{R}^d) \subset \mathcal{C}^{0,\alpha/2}(\mathbb{R}
^d)$. We apply the estimate of Corollary~\ref{régularité_AP_coeff_periodique} to equation~\eqref{equation_liouville2_R} and we obtain the existence of a constant $C>0$ independent of $w$, $R$ and $\Tilde{a}$ such that : 
\begin{equation}
\label{estim_holder_w_AP}
    \|\nabla w\|_{\mathcal{C}^{0, \alpha/2}(\mathbb{R}^d)} \leq C \left(\left\|\mathcal{M}\left(\left|\Tilde{a}^C_R \nabla w\right|\right)\right\|_{L^{p^*}(\mathbb{R}^d)}  + \|\Tilde{a}^C_R \nabla w\|_{\mathcal{C}^{0, \alpha/2}(\mathbb{R}^d)}\right).
\end{equation}
Our aim is now to estimate each norm of the right-hand side in the previous inequality. 
Let $\varepsilon>0$. Since $\Tilde{a} \in \left(\mathcal{A}^p \cap \mathcal{C}^{0,\alpha}(\mathbb{R}^d)\right)^{d\times d}$ and $\mathcal{M}(|\Tilde{a}|) \in L^{p^*}(\mathbb{R}^d)$, there exists $R_1>0$ such that for every $R>R_1$ we have $\|\mathcal{M}(|\Tilde{a}|)\|_{L^{p^*}(\mathbb{R}^d \setminus{B_R})} \leq \varepsilon$. It follows : 
\begin{align*}
\left\|\mathcal{M}\left(\left|\Tilde{a}^C_R \nabla w\right|\right)\right\|_{L^{p^*}(\mathbb{R}^d)} & \leq \|\mathcal{M}\left(\left|\Tilde{a} \nabla w\right|\right)\|_{L^{p^*}(\mathbb{R}^d\setminus{B_R})} \leq \|\mathcal{M}\left(\left|\Tilde{a} \right|\right)\|_{L^{p^*}(\mathbb{R}^d\setminus{B_R})} \|\nabla w\|_{L^{\infty}(\mathbb{R}^d)}. 
\end{align*}
If we therefore consider $R \geq R_1$, we obtain : 
\begin{equation}
\label{estim_norm1_deltap}
\left\|\mathcal{M}\left(\left|\Tilde{a}^C_R \nabla w\right|\right)\right\|_{L^{p^*}(\mathbb{R}^d)}
\leq \varepsilon \|\nabla w\|_{\mathcal{C}^{0, \alpha/2}(\mathbb{R}^d)}.
\end{equation}

In the sequel, for $\beta \in ]0,1[$ we denote 
$\displaystyle [f]_{\mathcal{C}^{0,\beta}(\mathbb{R}^d)} = \sup_{x,y \in \mathbb{R}^d, x \neq y} \dfrac{|f(x)-f(y)|}{|x-y|^{\beta}}$.
We next remark that : 
\begin{equation}
\label{estim_norm_2_CO_ap}
    \left\|\Tilde{a}^C_R \nabla w\right\|_{\mathcal{C}^{0, \alpha/2}(\mathbb{R}^d)} \leq 2 \left\|\Tilde{a}^C_R \right\|_{\mathcal{C}^{0, \alpha/2}(\mathbb{R}^d)} \left\| \nabla w \right\|_{\mathcal{C}^{0, \alpha/2}(\mathbb{R}^d)}.
\end{equation}

Since $\Tilde{a} \in \left(\mathcal{C}^{0, \alpha}(\mathbb{R}^d)\right)^{d\times d}$, for every $x,y \in \mathbb{R}^d$ we have using \eqref{hyp_rayon_a_petit} : 
\begin{align*}
  \left|\Tilde{a}^C_R(x) - \Tilde{a}^C_R(y)\right| & \leq \sqrt{2} \|\Tilde{a}^C_R\|_{L^{\infty}(\mathbb{R}^d)}^{1/2} \left|\Tilde{a}^C_R(x) - \Tilde{a}^C_R(y)\right|^{1/2}
   \leq \sqrt{2 \varepsilon}\left\|\Tilde{a}^C_R\right\|_{\mathcal{C}^{0,\alpha}(\mathbb{R}^d)}^{1/2}|x-y|^{\alpha/2}.
\end{align*}
In addition, for every $R>0$, we have : 
\begin{align*}
  \left|\Tilde{a}^C_R(x) - \Tilde{a}^C_R(y)\right| & \leq \left|\Tilde{a}(x) - \Tilde{a}(y)\right| |1-\chi_R(x)| + \left|\chi_R(x) - \chi_R(y)\right||\Tilde{a}(y)|\\
  & \leq \left(\|\Tilde{a}\|_{\mathcal{C}^{0, \alpha}(\mathbb{R}^d)} \|1-\chi_R\|_{L^{\infty}(\mathbb{R}^d)} + \|\chi_R\|_{\mathcal{C}^{0, \alpha}(\mathbb{R}^d)} \|\Tilde{a}\|_{L^{\infty}(\mathbb{R}^d)}\right)|x-y|^{\alpha}.
\end{align*}
Since  $\chi_R$ and $\nabla \chi_R$ are uniformly bounded with respect to $R$ in $L^{\infty}(\mathbb{R}^d)$, we deduce there exists a constant $C>0$ independent of $R$ such that $\left\|\Tilde{a}^C_R\right\|_{\mathcal{C}^{0,\alpha}(\mathbb{R}^d)} \leq C$. It follows that, if $R>R_0$, where $R_0$ is such that \eqref{hyp_rayon_a_petit} is satisfied for every $R>R_0$, then $[\Tilde{a}_R^C]_{\mathcal{C}^{0,\alpha/2}(\mathbb{R}^d)} \leq C \sqrt{\varepsilon}$.
Using \eqref{hyp_rayon_a_petit}, we deduce that $ \|\Tilde{a}_R^C\|_{\mathcal{C}^{0,\alpha/2}(\mathbb{R}^d)} \leq C \sqrt{\varepsilon}$,
and as a consequence of \eqref{estim_norm_2_CO_ap},
\begin{equation}
\label{estim_norm_2_deltap}
    \left\|\Tilde{a}^C_R \nabla w\right\|_{\mathcal{C}^{0, \alpha/2}(\mathbb{R}^d)} \leq 2 C \sqrt{\varepsilon}\left\| \nabla w \right\|_{\mathcal{C}^{0, \alpha/2}(\mathbb{R}^d)}.
\end{equation}

Finally, if $R>\max(R_0,R_1)$, we can insert \eqref{estim_norm1_deltap} and \eqref{estim_norm_2_deltap} in \eqref{estim_holder_w_AP}, and we obtain the existence of a constant $C>0$ independent of $R$ and $\varepsilon$ such that : 
$$ \|\nabla w\|_{\mathcal{C}^{0, \alpha/2}(\mathbb{R}^d)} \leq C \sqrt{\varepsilon} \|\nabla w\|_{\mathcal{C}^{0, \alpha/2}(\mathbb{R}^d)}.$$
If $\varepsilon$ is small enough, we have $C\sqrt{\varepsilon}<1$, and we obtain 
$\|\nabla w\|_{\mathcal{C}^{0, \alpha/2}(\mathbb{R}^d)} = 0$.
We conclude that $\nabla w = 0$, that is $\nabla u = \nabla v_R \in \left(L^2(\mathbb{R}^d)\right)^d$. 

\textbf{Step 4 : Conclusion}. In the previous step we have established that $\nabla u = \nabla v_R \in \left(L^2(\mathbb{R}^d)\right)^d$. Since $p<d$ and $a$ is uniformly bounded and elliptic according to assumptions \eqref{hypothèses1}-\eqref{hypothèses2}, the result of uniqueness of \cite[Lemma 1]{blanc2012possible} for solution $u$ to \eqref{equation_unicité_general} with a gradient in $\left(L^p(\mathbb{R}^d)\right)^d$ shows that $\nabla u =0$.
\end{proof}

We are now in position to establish the continuity of the operator $\nabla \left(- \operatorname{div}(a \nabla .)\right) ^{-1} \operatorname{div}$ from $\left(\mathcal{A}^p \cap~\mathcal{C}^{0,\alpha}(\mathbb{R}^d)\right)^d$ to $\left(\mathcal{A}^p \cap \mathcal{C}^{0,\alpha}(\mathbb{R}^d)\right)^d$.

\begin{lemme}
\label{estimee_A_priori_Ap}
Let $a_{per} \in  \left(L^2_{per}(\mathbb{R}^d)\right)^{d \times d}$ and $\Tilde{a} \in \left(\mathcal{A}^p\right)^{d\times d}$ for $p\in ]1,d[$.  Assume that $a = a_{per} + \Tilde{a}$ satisfies \eqref{hypothèses1}, \eqref{hypothèse_a_elliptic_AP} and \eqref{hypothèse_a_holder_AP}. There exists a constant $C>0$ such that for every $f \in \left(\mathcal{A}^p\cap \mathcal{C}^{0, \alpha}(\mathbb{R}^d)\right)^d$ and $u$ solution to
$ - \operatorname{div}(a \nabla u) = \operatorname{div}(f)$ on $\mathbb{R}^d$ 
with $\nabla u \in \left(\mathcal{A}^p\cap \mathcal{C}^{0, \alpha}(\mathbb{R}^d)\right)^d$, we have : 
\begin{equation}
    \|\nabla u\|_{\mathcal{A}^p} + \|\nabla u\|_{\mathcal{C}^{0, \alpha}(\mathbb{R}^d)} \leq C \left(  \|f\|_{\mathcal{A}^p} + \|f\|_{\mathcal{C}^{0, \alpha}(\mathbb{R}^d)}\right).
\end{equation}

\end{lemme}

\begin{proof}
We argue by contradiction. We assume the existence of two sequences $u_n$ and $f_n$ such that for every $n\in \mathbb{N}$ we have $\nabla u_n, f_n \in \left(\mathcal{A}^p \cap \mathcal{C}^{0, \alpha}(\mathbb{R}^d)\right)^d$ and : 
\begin{equation}
\label{equation_suite_lemme_estimation}
    - \operatorname{div}(a \nabla u_n) = \operatorname{div}(f_n),
\end{equation}
\begin{equation}
\label{hypothèse_borne_un_lemme_estimee}
   \|\nabla u_n\|_{\mathcal{A}^p} + \|\nabla u_n\|_{\mathcal{C}^{0, \alpha}(\mathbb{R}^d)}  = 1,
\end{equation}
\begin{equation}
\label{hyplimite_fn_AP}
    \lim_{n \rightarrow \infty} \|f_n\|_{\mathcal{A}^p} + \|f_n\|_{\mathcal{C}^{0, \alpha}(\mathbb{R}^d)} = 0.
\end{equation}

Since $\nabla u_n$ is bounded uniformly with respect to $n$ for the topology of $\mathcal{C}^{0, \alpha}(\mathbb{R}^d)$, the Arzela-Ascoli theorem shows the uniform convergence of $\nabla u_n$ (up to an extraction) on every compact of $\mathbb{R}^d$ to a gradient $\nabla u \in \left(L^{\infty}(\mathbb{R}^d)\right)^d$. Consequently, if we consider the limit in \eqref{equation_suite_lemme_estimation} when $n\to \infty$, we obtain that $\nabla u$ is solution to $-\operatorname{div}(a \nabla u) = 0 \quad \text{on } \mathbb{R}^d$.

We next claim that $\nabla u$ belongs to $\left(\mathcal{A}^p\right)^d$. Indeed, since $\delta \nabla u_n$ is uniformly bounded with respect to $n$ in $\left(L^p(\mathbb{R}^d)\right)^{d\times d}$ for $p \in ]1,d[$, it weakly converges (up to an extraction) in this space and its limit is equal to $\delta \nabla u$ due to the uniqueness of the limit in the distribution sense. Moreover, since $p>1$, the $L^p$ norm is lower semi-continuous and we have 
$\displaystyle \|\delta \nabla u\|_{L^p(\mathbb{R}^d)} \leq \liminf_{n \rightarrow \infty} \|\delta \nabla u_n \|_{L^p(\mathbb{R}^d)} = 1.$
We also know that $\nabla u_n$ uniformly converges on every compact of $\mathbb{R}^d$, and consequently, that $\mathcal{M}(|\nabla u_n|)$ converges pointwise to $\mathcal{M}(|\nabla u|)$. Using the Fatou lemma, we obtain : 
\begin{align*}
    \int_{\mathbb{R}^d} \left|\mathcal{M}\left(\left|\nabla u\right| \right)(z)\right|^{p^*}dz & = \int_{\mathbb{R}^d} \liminf_{n \rightarrow \infty} \left|\mathcal{M}\left(\left|\nabla u_n\right| \right)(z)\right|^{p^*}dz \leq \liminf_{n \rightarrow \infty} \int_{\mathbb{R}^d}  \left|\mathcal{M}\left(\left|\nabla u_n\right| \right)(z)\right|^{p^*}dz \leq 1.
\end{align*}
It follows that $\nabla u$ belongs to $\left(\mathcal{A}^p \cap L^{\infty}(\mathbb{R}^d)\right)^d$ and the uniqueness result of Lemma~\ref{unicite_generale_AP} implies that $\nabla u = 0$. Our aim is now to prove that $\displaystyle \lim_{n \rightarrow \infty} \|\nabla u_n\|_{\mathcal{A}^p} + \|\nabla u_n\|_{\mathcal{C}^{0,\alpha}(\mathbb{R}^d)} =0$,
in order to reach a contradiction. We first remark that \eqref{equation_suite_lemme_estimation} is equivalent to :
\begin{equation}
\label{eq_equivalente_lemme_estimee}
    - \operatorname{div}(a_{per} \nabla u_n) = \operatorname{div}(\Tilde{a} \nabla u_n + f_n) \quad \text{on } \mathbb{R}^d.
\end{equation}
Since $\nabla u_n$ and $\Tilde{a}$ both belong to $L^{\infty}$, we can easily show that $\Tilde{a} \nabla u_n$ belongs to $\mathcal{A}^p$. Consequently, Proposition \ref{existence_cas_periodique_AP} gives the existence of a constant $C>0$ independent of $n$ such that : 
\begin{equation}
\label{majoration_un_Ap}
    \|\nabla u_n\|_{\mathcal{A}^p}\leq C\left(\|\Tilde{a}\nabla u_n\|_{\mathcal{A}^p} + \|f_n\|_{\mathcal{A}^p}\right). 
\end{equation}
We fix $\varepsilon>0$. Since $\Tilde{a}$ is H\"older-continuous according to assumption \eqref{hypothèse_a_holder_AP}, Corollary~\ref{comportement_asymptotique_AP} gives the existence of $R_1>0$ such that $\|\Tilde{a}\|_{L^{\infty}(\mathbb{R}^d \setminus{B_{R_1}})} < \frac{\varepsilon}{2}$. 
In addition, since $\delta \Tilde{a} \in \left(L^p(\mathbb{R}^d)\right)^{d\times d\times d}$ and $\nabla u_n$ is uniformly bounded for the norm of $L^{\infty}$ with respect to $n$, there exists $R_2>0$ such that for every $n\in \mathbb{N}$ :
\begin{equation}
    \| \delta_i \Tilde{a}\|_{L^p(\mathbb{R}^d \setminus{B_{R_2}})}\|\nabla u_n\|_{L^{\infty}(\mathbb{R}^d)} < \frac{\varepsilon}{2}.
\end{equation}
We next denote $R = \max(R_1,R_2)$. We have proved that $\nabla u_n$ uniformly converges to 0 on every compact of $\mathbb{R}^d$. Therefore, there exists $N \in \mathbb{N}$ such that for every $n \geq N$, we have : 
\begin{equation}
    \|\nabla u_n \|_{L^{\infty}(B_{R+1})} \left(\|\Tilde{a}\|_{L^p(B_R)} + 2\|\Tilde{a}\|_{\mathcal{A}^p}\right) < \frac{\varepsilon}{2}.
\end{equation}
For every $n \geq N$ and $i \in \{1,...,d\}$, we have : 
\begin{align}
\label{majoration_deltaa_nablau_n}
  \|\delta_i \left(\Tilde{a}\nabla u_n\right)\|_{L^p(\mathbb{R}^d)}&  \leq  \|\delta_i\Tilde{a} \tau_{e_i} \nabla u_n\|_{L^p(\mathbb{R}^d)} + \|\Tilde{a} \delta_i \nabla u_n\|_{L^p(\mathbb{R}^d)}.
\end{align}
We next prove that the right-hand side of the previous inequality converges to 0 when $n \to \infty$. We have~: 
\begin{align}
    \|\delta_i\Tilde{a} \tau_{e_i} \nabla u_n\|_{L^p(\mathbb{R}^d)} & \leq \|\delta_i\Tilde{a} \tau_{e_i} \nabla u_n\|_{L^p(B_R)} + \|\delta_i\Tilde{a} \tau_{e_i} \nabla u_n\|_{L^p(\mathbb{R}^d\setminus{B_R})}\\
    & \leq \| \Tilde{a} \|_{\mathcal{A}^p}\|\nabla u_n \|_{L^{\infty}(B_{R+1})} + \| \delta_i \Tilde{a}\|_{L^p(\mathbb{R}^d \setminus{B_R})}\|\nabla u_n\|_{L^{\infty}(\mathbb{R}^d)} 
    < \varepsilon.
\end{align}
Since the parameter $\varepsilon$ can be chosen arbitrarily small, it follows that $\displaystyle \lim_{n \rightarrow \infty} \|\delta_i\Tilde{a} \tau_{e_i} \nabla u_n\|_{L^p(\mathbb{R}^d)} = 0$.
Similarly: 
\begin{align*}
    \|\Tilde{a} \delta_i \nabla u_n\|_{L^p(\mathbb{R}^d)} 
    & \leq 2 \|\Tilde{a}\|_{L^p(B_R)}\|\nabla u_n\|_{L^{\infty}(B_{R+1})} + \|\Tilde{a}\|_{L^{\infty}(\mathbb{R}^d \setminus{B_R})} \|\nabla u_n\|_{\mathcal{A}^p}
    < \varepsilon,
\end{align*}
and $ \displaystyle \lim_{n \rightarrow \infty}  \|\Tilde{a} \delta_i \nabla u_n\|_{L^p(\mathbb{R}^d)} = 0$. 
Using \eqref{majoration_deltaa_nablau_n}, we therefore obtain that $\|\Tilde{a}\nabla u_n\|_{\mathcal{A}^p}$ converges to $0$ when $n \to \infty$. In addition, assumption \eqref{hyplimite_fn_AP} ensures that $f_n$ converges to $0$ in $\left(\mathcal{A}^p\right)^d$ and, according to inequality \eqref{majoration_un_Ap}, we obtain that $\displaystyle \lim_{n \rightarrow \infty} \|\nabla u_n\|_{\mathcal{A}^p} = 0$. 

The last step of the proof consists in showing that $\displaystyle \lim_{n\to \infty} \|\nabla u_n\|_{\mathcal{C}^{0,\alpha}(\mathbb{R}^d)} = 0$. Since $\nabla u_n$ is solution to \eqref{eq_equivalente_lemme_estimee}, the estimate established in Proposition \ref{existence_cas_periodique_AP} shows the existence of $C>0$ such that for every $n$, 
$$\| \nabla u_n \|_{L^2_{unif}} \leq C \left(\|\Tilde{a}\nabla u_n\|_{\mathcal{A}^p} + \|\Tilde{a}\nabla u_n\|_{L^{\infty}(\mathbb{R}^d)}\right).$$
A method similar to that presented above allows us to show that
$\| \nabla u_n \|_{L^2_{unif}}$ converges to $0$ when $n \to \infty$. The regularity estimate of Corollary \ref{régularité_AP_coeff_periodique} and assumption \eqref{hyplimite_fn_AP} therefore show that $\displaystyle \lim_{n\to \infty} \|\nabla u_n\|_{\mathcal{C}^{0,\alpha}(\mathbb{R}^d)} = 0$. We finally reach a contradiction with \eqref{hypothèse_borne_un_lemme_estimee} and we conclude the proof. 
\end{proof}

We are in position to prove the main result of this section, that is the existence and uniqueness of a solution to \eqref{equation_generale_Ap}.


\begin{lemme}
\label{lemme_existence_general_ap}
Let $a_{per} \in  \left(L^2_{per}(\mathbb{R}^d)\right)^{d \times d}$ and $\Tilde{a} \in \left(\mathcal{A}^p\right)^{d\times d}$ for $p\in ]1,d[$.  Assume that $a = a_{per} + \Tilde{a}$ satisfies \eqref{hypothèses1}, \eqref{hypothèse_a_elliptic_AP} and \eqref{hypothèse_a_holder_AP}. Let $f \in \left(\mathcal{A}^p\cap \mathcal{C}^{0,\alpha}(\mathbb{R}^d)\right)^d$, then, there exists a unique, up to an additive constant, function $u \in L^1_{loc}(\mathbb{R}^d)$ solution in $\mathcal{D}'(\mathbb{R}^d)$ to \eqref{equation_generale_Ap}
such that $\nabla u \in \left(\mathcal{A}^p\cap \mathcal{C}^{0, \alpha}(\mathbb{R}^d)\right)^d$. 
\end{lemme}

\begin{proof}
We use here a method introduced in \cite[proof of Proposition 2.1]{blanc2018correctors} using the connectedness of the set $[0,1]$. In the sequel, we denote $a_s = a_{per} + s \Tilde{a}$ for every $s\in [0,1]$ and we consider the following assertion $\mathcal{P}(s) = $ "for every $f \in \left(\mathcal{A}^p\cap \mathcal{C}^{0, \alpha}(\mathbb{R}^d)\right)^d$, there exists a unique, up to an additive constant, function $u \in L^1_{loc}(\mathbb{R}^d)$ solution to $- \operatorname{div}(a_s \nabla u) = \operatorname{div}(f)$,
in $\mathcal{D}'(\mathbb{R}
^d)$ and such that $\nabla u \in \left(\mathcal{A}^p\cap \mathcal{C}^{0,\alpha}(\mathbb{R}^d)\right)^d$". We define the set $\mathcal{I} = \left\{ s \in [0,1] \ \middle| \ \mathcal{P}(s) \text{ is true}\right\}$.
Our aim is to prove that $s=1$ belongs to $\mathcal{I}$ establishing that $\mathcal{I}$ is non empty, open and closed for the topology of $[0,1]$.

\textbf{Step 1 : $\mathcal{I}$ is non empty}. The results of Proposition \ref{existence_cas_periodique_AP} and Corollary \ref{régularité_AP_coeff_periodique} show that $s=0$ belongs to $\mathcal{I}$.  

\textbf{Step 2 : $\mathcal{I}$ is open}. We assume there exists $s\in \mathcal{I}$ and we will show the existence of $\varepsilon>0$ such that $[s,s+ \varepsilon]$ is included in $\mathcal{I}$, that is that there exists $u$ solution to : 
\begin{equation}
\label{equation_ouvert}
  -\operatorname{div}((a_{per} + (s+ \varepsilon)\Tilde{a})\nabla u) = \operatorname{div}(f),  
\end{equation}
when $\varepsilon$ is sufficiently small. We first remark that the above equation is equivalent to : 
$$ - \operatorname{div}((a_{per}+ s \Tilde{a})\nabla u) = \operatorname{div}(\varepsilon \Tilde{a} \nabla u + f).$$
A simple calculation allows to show that if $\nabla u$ belongs to $\left(\mathcal{A}^p\cap \mathcal{C}^{0, \alpha}(\mathbb{R}^d)\right)^d$, then $\varepsilon \Tilde{a} \nabla u$ also belongs to this space. Therefore, the existence and uniqueness of a solution $u$ to \eqref{equation_ouvert} is equivalent to the existence and uniqueness of a solution to the fixed-point problem $ \nabla u = \Phi_s (\varepsilon \Tilde{a}\nabla u + f)$,
where $\Phi_s$ is the linear application $\nabla \left(-\operatorname{div}(a_s \nabla.)\right)^{-1} \operatorname{div}$ from $\left(\mathcal{A}^p\cap \mathcal{C}^{0,\alpha}(\mathbb{R}^d)\right)^d$ to $\left(\mathcal{A}^p \cap \mathcal{C}^{0,\alpha}(\mathbb{R}^d)\right)^d$. Since $s \in \mathcal{I}$, $\Phi_s$ is well defined and the result of Lemma \ref{estimee_A_priori_Ap} ensures that it is continuous for the norm $\|.\|_{\mathcal{A}^p} + \|.\|_{\mathcal{C}
^{0,\alpha}(\mathbb{R}^d)}$. 
We claim that the application $g \rightarrow \Phi_s(\varepsilon\Tilde{a}g + f)$ is a contraction if $\varepsilon$ is sufficiently small. First, if $g_1$ and $g_2$ are two functions of $\left(\mathcal{A}^p\cap \mathcal{C}^{0, \alpha}(\mathbb{R}^d)\right)^d$ and if we denote $\nabla v_1 = \Phi_s(\varepsilon\Tilde{a}g_1 + f)$ and $\nabla v_2 = \Phi_s(\varepsilon\Tilde{a}g_2 + f)$, $v_1 - v_2$ is solution to
$-\operatorname{div}(a_s \nabla (v_1 - v_2)) = \operatorname{div}(\varepsilon \Tilde{a}(g_1 - g_2))$.
The continuity estimate of Lemma \ref{estimee_A_priori_Ap} therefore shows the existence of a constant $C>0$ independent of $g_1$, $g_2$ and $\varepsilon$ such that : 
$$\|\nabla v_1 - \nabla v_2\|_{\mathcal{A}^p} + \|\nabla v_1 - \nabla v_2\|_{\mathcal{C}^{0, \alpha}(\mathbb{R}^d)} \leq C \varepsilon \left(\|\Tilde{a}(g_1 - g_2)\|_{\mathcal{A}^p} + \| \Tilde{a}(g_1 - g_2)\|_{\mathcal{C}^{0, \alpha}(\mathbb{R}^d)}\right).$$
In addition, we have : 
$$ \|\Tilde{a}(g_1 - g_2)\|_{\mathcal{A}^p} \leq \left(\|\Tilde{a}\|_{\mathcal{A}^p} + \|\Tilde{a}\|_{\mathcal{C}^{0, \alpha}(\mathbb{R}^d)}\right)\left(\|g_1 - g_2\|_{\mathcal{A}^p} + \|g_1 - g_2\|_{\mathcal{C}^{0, \alpha}(\mathbb{R}^d)}\right),$$
and : 
$$\| \Tilde{a}(g_1 - g_2)\|_{\mathcal{C}^{0, \alpha}(\mathbb{R}^d)} \leq 2 \|\Tilde{a}\|_{\mathcal{C}^{0, \alpha}(\mathbb{R}^d)}\|g_1 - g_2\|_{\mathcal{C}^{0, \alpha}(\mathbb{R}^d)}.$$
Thus, if $\varepsilon$ satisfies
$3 C \left(\|\Tilde{a}\|_{\mathcal{A}^p} + \|\Tilde{a}\|_{\mathcal{C}^{0, \alpha}(\mathbb{R}^d)}\right) \varepsilon < 1$,
the operator $g \rightarrow \Phi_s(\varepsilon\Tilde{a}g + f)$ is a contraction. Since 
$\left(\mathcal{A}^p \cap \mathcal{C}^{0,\alpha}(\mathbb{R}^d)\right)^d$ equipped with the associated norm is a Banach space, we can use the Banach fixed-point theorem. We deduce the existence and the uniqueness of a solution to \eqref{equation_ouvert}.  

\textbf{Step 3 : $\mathcal{I}$ is closed}. We assume the existence of a sequence $(s_n)$ of $\mathcal{I}$ that converges to $s \in [0,1]$. We want to show that $s$ belongs to $\mathcal{I}$. Let $f \in \left(\mathcal{A}^p\cap \mathcal{C}^{0, \alpha}(\mathbb{R}^d)\right)^d$. By assumption, for every $n \in \mathbb{N}$, there exists $u_n \in L^1_{loc}(\mathbb{R}^d)$ solution to :
\begin{equation}
\label{equation_I_fermé_Ap}
    - \operatorname{div}(a_{s_n}\nabla u_n) = \operatorname{div}(f),
\end{equation}
such that $\nabla u_n \in \left(\mathcal{A}^p \cap \mathcal{C}^{0, \alpha}(\mathbb{R}^d)\right)^d$. For every $n \in \mathbb{N}$, we use Lemma \ref{estimee_A_priori_Ap} to obtain the existence of a constant $C_n>0$ such that : 
$$\|\nabla u_n\|_{\mathcal{A}^p} + \|\nabla u_n\|_{\mathcal{C}^{0, \alpha}(\mathbb{R}^d)}\leq C_n \left( \|f\|_{\mathcal{A}^p} + \|f\|_{\mathcal{C}^{0, \alpha}(\mathbb{R}^d)}\right).$$
We first assume that $C_n$ is uniformly bounded with respect to $n$. in this case, $\nabla u_n$ is uniformly bounded with respect to $n$ in $\left(L^{\infty}(\mathbb{R}^d)\right)^d$. Up to an extraction, the sequences $\nabla u_n$ converges to a gradient $\nabla u$ for the weak-* topology of $L^{\infty}$. In addition, $a_{s_n} = a_{per} + s_n \Tilde{a}$ uniformly converges to $a_s$. We can consider the limit when $n \to \infty$ in \eqref{equation_I_fermé_Ap} and we obtain that $u$ is solution to $-\operatorname{div}(a_s \nabla u) = \operatorname{div}(f)$.

We next show that $\nabla u_n$ is a Cauchy sequence in $\left(\mathcal{A}^p\cap \mathcal{C}^{0,\alpha}(\mathbb{R}^d)\right)^d$ in order to conclude that $\nabla u$ also belongs to this space. Indeed, for every $m,n \in \mathbb{N}$, $u_n - u_m$ is solution to : 
$$-\operatorname{div}(a_s (\nabla u_n -\nabla u_m)) = \operatorname{div}((a_{s_n} - a_{s})\nabla u_n - (a_{s_m} - a_{s})\nabla u_m).$$
Since for every $n \in \mathbb{N}$, we have $(a_{s_n} - a_s)\nabla u_n \in \left(\mathcal{A}^p\cap \mathcal{C}^{0, \alpha}(\mathbb{R}^d)\right)^d$, Lemma \ref{estimee_A_priori_Ap} gives the existence of $\Tilde{C}>0$ independent of $n$ and $m$ such that : 
\begin{align*}
  \|\nabla u_n - \nabla u_m\|_{\mathcal{A}^p\cap \mathcal{C}^{0, \alpha}(\mathbb{R}^d)} & \leq \Tilde{C} \left( \|(a_{s_n} - a_{s})\nabla u_n\|_{\mathcal{A}^p\cap \mathcal{C}^{0, \alpha}(\mathbb{R}^d)} +  \|(a_{s_m} - a_{s})\nabla u_m\|_{\mathcal{A}^p\cap \mathcal{C}^{0, \alpha}(\mathbb{R}^d)}\right) \\
  & = \left( |s_n - s| \|\Tilde{a} \nabla u_n\|_{\mathcal{A}^p\cap \mathcal{C}^{0, \alpha}(\mathbb{R}^d)} +  |s_m -s| \|\Tilde{a}\nabla u_m\|_{\mathcal{A}^p\cap \mathcal{C}^{0, \alpha}(\mathbb{R}^d)}\right),
\end{align*}
where we have denoted
$\|.\|_{\mathcal{A}^p\cap \mathcal{C}^{0, \alpha}(\mathbb{R}^d)} = \|.\|_{\mathcal{A}^p} + \|.\|_{\mathcal{C}^{0,\alpha}(\mathbb{R}^d)}$.
Since for every $n$, $\nabla u_n$ is uniformly bounded in $\left(\mathcal{A}^p\cap \mathcal{C}^{0,\alpha}(\mathbb{R}^d)\right)^d$  with respect to $n$ and since $s_n$ converges to $s$, we deduce that $\nabla u_n$ is a Cauchy sequence in $\left(\mathcal{A}^p\cap\mathcal{C}^{0,\alpha}(\mathbb{R}^d)\right)^d$. This space being a Banach space, we have $\nabla u \in \left(\mathcal{A}^p\cap\mathcal{C}^{0,\alpha}(\mathbb{R}^d)\right)^d$. The uniqueness result being established in Lemma~\ref{unicite_generale_AP}, we therefore obtain that $s\in \mathcal{I}$. 

To conclude this step, we have to show that $C_n$ is uniformly bounded with respect to $n$. To this end, we assume the existence of two sequences $(f_n)$ and $(\nabla u_n)$ of $\left(\mathcal{A}^p\cap \mathcal{C}^{0, \alpha}(\mathbb{R}^d)\right)^d$ such that $ \displaystyle \lim_{n \rightarrow \infty} \|f_n\|_{\mathcal{A}^p\cap \mathcal{C}^{0, \alpha}(\mathbb{R}^d)} = 0$,  $\|\nabla u_n\|_{\mathcal{A}^p\cap \mathcal{C}^{0, \alpha}(\mathbb{R}^d)} = 1$ and $- \operatorname{div}(a_{s_n} \nabla u_n) = \operatorname{div}(f_n)$ on $\mathbb{R}^d$.
We remark that for every $n$,  we have $- \operatorname{div}(a_s \nabla u_n) = \operatorname{div}((s-s_n) \Tilde{a} \nabla u_n + f_n)$.
Since $\nabla u_n$ is bounded for the norm of $\left(\mathcal{A}^p\cap \mathcal{C}^{0, \alpha}(\mathbb{R}^d)\right)^d$ and that $s_n$ converges to $s$, we deduce that $\displaystyle \lim_{n \rightarrow \infty} \|(s-s_n) \Tilde{a} \nabla u_n\|_{\mathcal{A}^p\cap \mathcal{C}^{0, \alpha}(\mathbb{R}^d)} = 0$. We conclude the proof exactly as in the proof of Lemma \ref{estimee_A_priori_Ap}. 

\textbf{Step 4 : Conclusion}. We have established that $\mathcal{I}$ is non-empty, open, and closed for the topology of $[0,1]$. The connectedness of this set therefore gives $\mathcal{I} = [0,1]$. In particular, $1\in \mathcal{I}$ and we can conclude.  
\end{proof}

\begin{prop}
\label{correcteur_adapte_borne_ap}
Assume $1<p<\frac{d}{2}$. Then, under the assumptions of Lemma \ref{lemme_existence_general_ap}, the unique solution $u$ to \eqref{equation_generale_Ap} such that $\nabla u \in \left(\mathcal{A}^p\cap \mathcal{C}^{0,\alpha}(\mathbb{R}^d)\right)^d$ satisfies $u \in L^{\infty}(\mathbb{R}^d)$.
\end{prop}

\begin{proof}
We remark that \eqref{equation_generale_Ap} is equivalent to $-\operatorname{div}(a_{per} \nabla u) = \operatorname{div}(g)$ where $g:= f + \Tilde{a}\nabla u$. Since $\Tilde{a}$ and $\nabla u$ both belongs to $\left(\mathcal{A}^p\cap \mathcal{C}^{0,\alpha}(\mathbb{R}^d)\right)^d$, we can easily show that $g$ also belongs to $\left(\mathcal{A}^p\cap \mathcal{C}^{0,\alpha}(\mathbb{R}^d)\right)^d$. Given $1<p<\frac{d}{2}$, we conclude that $u \in L^{\infty}(\mathbb{R}^d)$ using both Lemma~\ref{lemme_borne_u_per_ap} and the uniqueness result of Lemma~\ref{unicité_periodique_AP}.  
\end{proof}

\subsection{Existence of an adapted corrector and homogenization results}

\label{Subsection_homog_ap}

The well-posedness of \eqref{equation_generale_Ap} now allows for a proof of Theorem \ref{Correcteur_A_P}. We first establish the existence of a corrector adapted to our particular problem \eqref{eq_correcteur_A_p} and we next use it to identify the limit of the sequence $u^{\varepsilon}$, solution to \eqref{equationepsilon_new}.

\begin{proof}[Proof of Theorem \ref{Correcteur_A_P}]
As a consequence of Proposition \ref{expliciteperiodiqueap}, we have that $a_{per}$ and $\Tilde{a}$ satisfy the properties of ellipticity \eqref{hypothèse_a_elliptic_AP} and regularity \eqref{hypothèse_a_holder_AP}. We next remark that \eqref{eq_correcteur_A_p} is equivalent to
$-\operatorname{div}(a \nabla \Tilde{w}_q) = \operatorname{div}(\Tilde{a}(q + \nabla w_{per,q}))$ and
we denote $f = \Tilde{a}(q + \nabla w_{per,q})$. Since $a_{per}$ belongs to $\left(\mathcal{C}^{0, \alpha}(\mathbb{R}^d)\right)^{d\times d}$, a classical regularity property of elliptic equations shows that $\nabla w_{per,q}$ belongs to $\left(\mathcal{C}^{0, \alpha}(\mathbb{R}^d)\right)^d$. Using the periodicity of $\nabla w_{per,q}$, we also have $f \in \left(\mathcal{A}^p\right)^d$. In addition, since $\Tilde{a}$ belongs to $\left(\mathcal{C}^{0,\alpha}(\mathbb{R}^d)\right)^{d\times d}$, we deduce that $f \in \left(\mathcal{C}^{0,\alpha}(\mathbb{R}^d)\right)^d$. Existence and uniqueness (up to an additive constant) of $\Tilde{w}_q$ solution to \eqref{eq_correcteur_A_p} such that $\nabla \Tilde{w}_q \in \left( \mathcal{A}^p \cap \mathcal{C}^{0, \alpha}(\mathbb{R}^d)\right)^d$ are therefore implied by Lemma~\ref{lemme_existence_general_ap}. The strict sub-linearity at infinity of $ \Tilde{w}_q$ is a consequence of Proposition \ref{Sous_linéarite_AP}.

Next we denote $w = (w_{per, e_i} + \Tilde{w}_{e_i})_{i \in \{1,...,d\}}$, where $\Tilde{w}_{e_i}$ is the corrector solution to \eqref{eq_correcteur_A_p} when $q=e_i$. The general homogenization theory for equations in divergence form (see for example \cite[Chapter 6, Chapter 13]{tartar2009general}) shows that, up to an extraction, the sequence $u^{\varepsilon}$ converges (strongly in $L^2$, weakly in $H^1$) to a function $u^* \in H^1_0(\Omega)$ solution to
$-\operatorname{div}(a^* \nabla u^*) = f$.
For every $1 \leq i,j \leq d$, the homogenized matrix-valued coefficient $a^*$ associated with $a$ is given by $\displaystyle \left[ a^* \right]_{i,j} = \operatorname{weak} \lim_{\varepsilon \rightarrow 0} a(./\varepsilon)(I_d + \nabla w(./\varepsilon))$,
where the weak limit is considered in $L^{2}(\Omega)^{d \times d}$. Since $\Tilde{a}$ and $\nabla w_{e_i}$ both belong to $\mathcal{A}^p \cap \mathcal{C}^{0, \alpha}(\mathbb{R}^d)$, Corollary  \ref{convergenceLinfinistar_AP} implies the convergence to 0 of $|\Tilde{a}|(./\varepsilon)$ and  $|\nabla \Tilde{w}_{e_i}|(./ \varepsilon)$ when $\varepsilon \to 0$ for the weak-* topology of $L^{\infty}(\Omega)$. In particular, we have for every $\varphi \in \left(\mathcal{D}(\mathbb{R^d})\right)^d$ : 
$$\left|\int_{\Omega} \Tilde{a}(x/\varepsilon) \nabla \Tilde{w}_{e_i}(x/\varepsilon).\varphi(x) dx \right| \leq \|\Tilde{a} \|_{L^{\infty}(\mathbb{R}^d)} \int_{\Omega}  |\nabla \Tilde{w}_{e_i}(x/\varepsilon)|.|\varphi(x)| dx \stackrel{\varepsilon \to 0}{\longrightarrow} 0. $$
It follows that $ \displaystyle \operatorname{weak}\lim_{\varepsilon \rightarrow 0} \Tilde{a}(./\varepsilon) \nabla \Tilde{w}_{e_i}(./\varepsilon) = 0$.
We similarly have $\displaystyle \operatorname{weak}\lim_{\varepsilon \rightarrow 0} \Tilde{a}(./\varepsilon) \nabla w_{per, e_i}(./\varepsilon) = 0$ and $\displaystyle \operatorname{weak}\lim_{\varepsilon \rightarrow 0} a_{per}(./\varepsilon) \nabla \Tilde{w}_{e_i}(./\varepsilon) = 0$. Since $a = a_{per} + \Tilde{a}$ and $w_{e_i} = w_{per,e_i} + \Tilde{w}_{e_i}$, we obtain : 
$$ \left[ a^* \right]_{i,j} = \operatorname{weak} \lim_{\varepsilon \rightarrow 0} a_{per}(./\varepsilon)(I_d + \nabla w_{per}(./\varepsilon)) =  \left[ a_{per}^* \right]_{i,j}.$$
This limit being independent of the extraction, we deduce that the whole sequence $u^{\varepsilon}$ converges to $u^*$ and $a^* = a^*_{per}$.
\end{proof}

We conclude this section with a discussion regarding the rates of convergence of $u^{\varepsilon}$ to $u^*$. Similarly to the periodic case and in order to make precise the behavior of $\nabla u^{\varepsilon}$, we can consider a sequence of remainders $R^{\varepsilon}$ using the adapted corrector of Theorem~\ref{Correcteur_A_P} and defined by $\displaystyle R^{\varepsilon}(x) = u^{\varepsilon}(x) - u^*(x) - \varepsilon \sum_{j=1}^d \partial_j u^*(x) w_{e_j}\left(\dfrac{x}{\varepsilon}\right)$.
The homogenization results we have established in Theorem \ref{Correcteur_A_P} and, more generally, the results of Section~\ref{SectionAP_2}, allow to use the general results of \cite{blanc2018precised}, which performs a study of homogenization problem \eqref{equationepsilon_new} under general assumptions, provided one has sufficient regularity of the coefficient $a$ and the existence of a corrector with a prescribed rate of strict sub-linearity at infinity. More precisely, if $r\geq 2$, $f \in L^r(\Omega)$ and $\Omega$ is a $\mathcal{C}^{1,1}$ domain, \cite[Theorem 1.5]{blanc2018precised} shows estimates of the form :
\begin{equation}
\label{convergence_rate_ap_general}
    \|\nabla R^{\varepsilon}\|_{L^r(\Omega_1)} \leq C \varepsilon^{\beta} \|f\|_{L^r(\Omega)}, 
\end{equation}
for every $\Omega_1 \subset \subset \Omega$ and where the value of $\beta$ is related to the decreasing rate of $\varepsilon w_{e_i}\left(\frac{.}{\varepsilon}\right)$. In our particular setting, we obtain \eqref{convergence_rate_ap_general} with $\beta = \mu$, where
\begin{equation}
\label{convergence_rate_ap}
   \mu  := \left\{ \begin{array}{cc}
    \dfrac{d}{p^*} & \text{if } p > \dfrac{d}{2},  \\
    1 & \text{if } p < \dfrac{d}{2},
\end{array}
\right.
\end{equation}
is obtained as a direct consequence of Propositions~\ref{Sous_linéarite_AP} and~\ref{correcteur_adapte_borne_ap}.
On the other hand, We recall that Proposition~\ref{prop_lebesgue_inclusion} shows that $\Tilde{a}$ also belongs to $\left(L^q(\mathbb{R}^d)\right)^{d \times d}$ for $q = \dfrac{p^*(\alpha +d) -d}{\alpha}$ and the results of \cite{blanc2018correctors, blanc2015local, blanc2012possible} and \cite[Theorem 1.2]{blanc2018precised} give \eqref{convergence_rate_ap_general} with
\begin{equation}
\label{convergence_rate_ap_lq}
\beta = \nu  := \left\{ \begin{array}{cc}
    \dfrac{d}{q} & \text{if } q > d,  \\
    1 & \text{if } q < d.
\end{array}
\right.
\end{equation}
Since $q > p^*$, a simple calculation shows that $\mu$ is always larger than $\nu$ and the theoretical convergence rates are significantly improved when $\delta a \in \left(L^p(\mathbb{R}^d)\right)^{d\times d}$. We point out that this improvement is particularly relevant if one is interested in fine convergence properties of $u^{\varepsilon}$, that is for the topology of $W^{1,r}$ when $r$ is large, at which scale the local perturbations of $\mathcal{A}^p$ affect the periodic background. This comparison therefore shows the interest of the specific study performed in the present article.

\section{The homogenization problem when $p\geq d$}

\label{SectionAP_3}

We devote this section to the homogenization problem \eqref{equationepsilon_new} when $p \geq d$. 
In this case, the behavior of the functions of $\mathbf{A}^p$ can be very different from the case $p<d$. A Gagliardo-Nirenberg-Sobolev type inequality such as that of Proposition \ref{decomposition_fonctions_A_p} does not hold. We exhibit two counter-examples of coefficient $a$ satisfying assumptions \eqref{hypothèses1}, \eqref{hypothèses2}, \eqref{hypothèses22} and \eqref{hypothèses3}, respectively for $p>d$ and $p=d$, but for which $a$ can not be split as the sum of a periodic coefficient and a perturbation integrable at infinity. The reason is, the decay of $\delta a$ at infinity may be too slow to ensure the existence of a periodic limit of $a$ at infinity. We illustrate the phenomenon with two coefficients $a$ respectively in dimension $d=1$ and $d=2$ which slowly oscillate at infinity, typically as $\sin(\ln(x))$ or $\sin(\ln(\ln(x))$. For such coefficients, we show that the homogenization of problem \eqref{equationepsilon_new} is not possible, since such $u^{\varepsilon}$ has subsequences converging to different limits.

\begin{figure}[h!]
\centering
\includegraphics[scale=0.32]{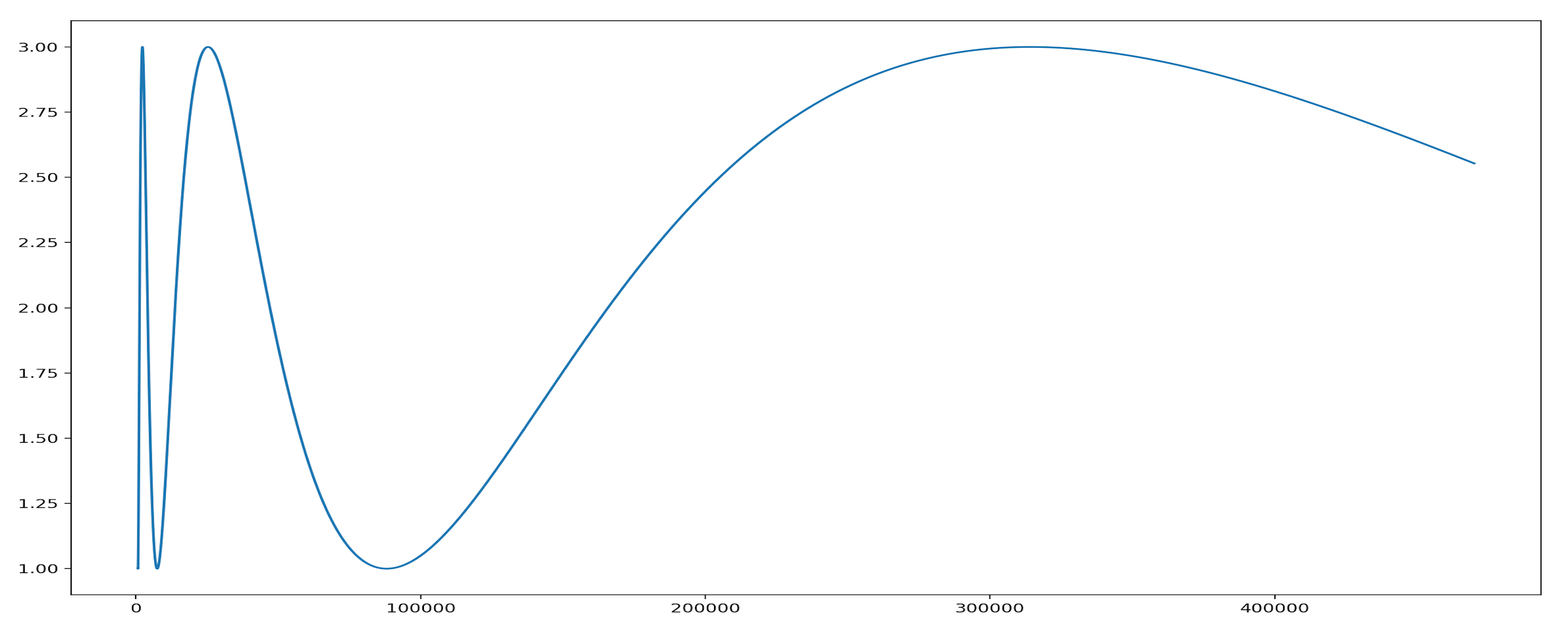}

\caption
{Example of coefficient $a$ with slow oscillations at infinity in dimension $d=1$.}
\label{figf2ap}
\end{figure}

\subsection{Counter-example for $d=1$, $p>1$}

To start with, we study a case where $a \in \mathbf{A}^p$ for $d=1$ and $p>1$. We define
$a(x) =  2 + \sin(\ln(1+|x|))$,
for every $x \in \mathbb{R}$. It is clear that this coefficient satisfies assumptions \eqref{hypothèses1}, \eqref{hypothèses2} and \eqref{hypothèses22}. We claim that $\delta a \in L^p(\mathbb{R})$ for every $p>1$. There indeed exists a constant $C$ such that for every $x\in \mathbb{R}$ with $|x|>1$,
$$|a'(x)| = \left|\dfrac{\cos(\ln(1+|x|))}{1+|x|}\right| \leq \dfrac{C}{|x|}.$$
The mean value theorem then shows that, for $|x|>1$, 
$$|\delta a(x)| = \left|\sin(\ln(1+|x|)) - \sin(\ln(1+|x+1|))\right| \leq \frac{C}{|x|},$$
from which, we deduce, as announced above, that $\delta a \in L^p(\mathbb{R})$ for every $p>1$.

We then consider the homogenization problem \eqref{equationepsilon_new} for
$\Omega = ]1,2[$, that is 
\begin{equation}
\label{eqhomog1D}
\left\{
\begin{array}{cc}
    \dfrac{d}{dx} \left(  a(x/\varepsilon) \dfrac{d}{dx}u^{\varepsilon} \right) = f & \text{on } ]1,2[, \\
    u^{\varepsilon}(1) = u^{\varepsilon}(2) =0. & 
\end{array}
\right.
\end{equation}
Our aim is to establish the existence of two sub-sequences $\left(\varepsilon_n^1\right)_{n \in \mathbb{N}}$ and $\left(\varepsilon_n^2\right)_{n \in \mathbb{N}}$ such that $u^{\varepsilon_n^1} \rightarrow u^{*,1}$ and $u^{\varepsilon_n^2} \rightarrow u^{*,2}$ in $L^2(\Omega)$ when $\varepsilon^1_n, \varepsilon^2_n \to 0$ and such that $u^{*,1} \neq u^{*,2}$. To this end, for $n \in \mathbb{N}$, we define 
$\varepsilon^1_n = \exp(-2 n \pi)$ and  $\varepsilon^2_n = \exp(-(2 n +1) \pi)$.
For $x \in ]1,2[$, we have
\begin{align*}
    a \left( \dfrac{x}{\varepsilon^1_n}  \right)  = 2 + \sin\left( \ln \left( 1 + \dfrac{x}{\varepsilon^1_n} \right) \right) 
    & = 2 + \sin\left( 2 \pi n +  \ln \left( \varepsilon^1_n + x \right) \right) \\
    & = 2 + \sin\left(  \ln \left( \varepsilon^1_n + x \right) \right). 
\end{align*}
Therefore, since $\varepsilon^1_n$ converges to 0 when $n \to \infty$, $a(x/\varepsilon^1_n) = 2 + \sin\left(  \ln \left( \varepsilon^1_n + x \right) \right)$ converges uniformly to $a^{*,1}(x) = 2 + \sin(\ln(x))$ on $]1,2[$. Since $a$ satisfies \eqref{hypothèses1} and \eqref{hypothèses2}, $u^{\varepsilon_n^1}$ is bounded in $H^1(\Omega)$ and, up to an extraction, it weakly converges to a function $u^{*,1}$ in $H^1(\Omega)$. Thus, for every $\varphi \in \mathcal{D}(\Omega)$, we have 
$$ \lim_{n \to \infty} \int_{\Omega}a(x/\varepsilon^1_n) \dfrac{du^{\varepsilon^{1}_n}}{dx} (x) \dfrac{d\varphi}{dx} (x)
dx = \int_{\Omega}a^{1,*}(x) \dfrac{du^{*,1}}{dx}(x) \dfrac{d\varphi}{dx} (x)
dx.$$
We obtain that $u^{1,*}$ is solution in $H^1_0(\Omega)$ to 
\begin{equation}
\label{eqhomog1D1}
\begin{array}{cc}
    -\dfrac{d}{dx} \left(  a^{*,1} \dfrac{d}{dx}u^{*,1} \right) = f & \text{on } ]1,2[.
\end{array}
\end{equation}
We may similarly show that $a(x/\varepsilon^2_n) = 2 - \sin\left(  \ln \left( \varepsilon^2_n + x \right) \right)$ converges uniformly to $a^{*,2}(x) = 2 - \sin(\ln(x))$ on $]1,2[$ and that $u^{\varepsilon^2_n}$ weakly converges in $H^1(\Omega)$ (up to an extraction) to $u^{2,*}$, solution in $H^1_0(\Omega)$ to 
\begin{equation}
\label{eqhomog1D2}
\begin{array}{cc}
    -\dfrac{d}{dx} \left(  a^{*,2} \dfrac{d}{dx}u^{*,2} \right) = f & \text{on } ]1,2[. 
\end{array}
\end{equation}

To conclude, we show that $u^{*,1} \neq u^{*,2}$. Indeed, if we assume that $u^{*,1}=u^{*,2} = u^*$ we have, 
\begin{equation}
\label{eqhomog1D0}
\begin{array}{cc}
    \dfrac{d}{dx} \left(  (a^{*,1}-a^{*,2}) \dfrac{d}{dx}u^{*} \right) = 0 & \text{on } ]1,2[.
\end{array}
\end{equation}
We use $u^*$ as a test function and, since $a^{*,1}-a^{*,2} = 2 \sin(\ln(.))$, we obtain : 
$$2\int_{\Omega} \sin(\ln(x)) \left| \dfrac{d}{dx} u^*\right|^2 =0.$$
We remark that for every $x \in]1,2[$, $\sin(\ln(x)) >0$ and obtain that $\dfrac{d}{dx} u^* = 0$ on $]1,2[$. Since $u^* \in H^1_0(\Omega)$, it follows that $u^* = 0$ and we reach a contradiction as soon as $f \neq 0$. 

We conclude with the following three remarks : 

\textbf{1.} For every $y \in [0,2\pi]$, we could also consider the sub-sequence $\varepsilon_n = \exp(- 2n \pi - y)$ and, as above, we could obtain that $u^{\varepsilon_n}$ converges to $u^*$ solution to $\displaystyle -\dfrac{d}{dx} \left(  a^{*} \dfrac{d}{dx}u^{*} \right) = f$ on $]1,2[$,
where $a^* = 2 +\sin(y + \ln(x))$. Therefore $u^*$ actually has an infinite number of adherent values. 

\textbf{2.} Unlike the periodic case, that is when $a=a_{per}$ is periodic, the coefficients $a^*$ of the homogenized equation that we obtain here (which depends on a chosen extraction) are not constant.  

\textbf{3.} For the specific coefficient $a$ chosen, a property similar to that of Proposition \ref{decomposition_fonctions_A_p} cannot hold. Actually, if $a$ were on the form $a = a_{per} + \Tilde{a}$ with $a_{per}$ periodic and $\Tilde{a}$ a function vanishing at infinity, we would be able to homogenize problem \eqref{eqhomog1D}. Indeed, some explicit calculations give 
$$u^{\varepsilon}(x) = - \int_{1}^x \dfrac{1}{a(y/\varepsilon)}F(y)dy + C^{\varepsilon} \int_{1}^x \dfrac{1}{a(y/\varepsilon)} dy,$$
where $\displaystyle F(x) = \int_{1}^x f(x) dx$ and  $\displaystyle C^{\varepsilon} = \left(\int_{1}^2 \dfrac{1}{a(y/\varepsilon)} dy\right)^{-1} \int_{1}^2 \dfrac{1}{a(y/\varepsilon)}F(y)dy$.

Since $\displaystyle \lim_{|x| \to \infty}|\Tilde{a}(x)| = 0$, we can show that $|\Tilde{a}(./\varepsilon)| \stackrel{\varepsilon \to 0}{\longrightarrow} 0$ in $L^{\infty}(\Omega)-*$. 
If we remark that $\dfrac{1}{a} = \dfrac{1}{a_{per}} - \dfrac{\Tilde{a}}{a_{per}(\Tilde{a}+a_{per})}$,
it follows that $\displaystyle \frac{1}{a}$ converges to $\langle \displaystyle \frac{1}{a_{per}} \rangle = (a_{per}^*)^{-1}$ for the weak-* topology of $L^{\infty}(\Omega)$ (where $\langle . \rangle$ denotes the average value of a periodic function). Therefore the limit $u^{*}$ of $u^{\varepsilon}$ can be made explicit : 
$$u^*(x) = - (a_{per}^*)^{-1} \int_{1}^x F(y)dy + (x-1)(a_{per}^*)^{-1} \int_{1}^2 F(y) dy,$$ 
which is the unique solution in $H^1_0(\Omega)$ to 
$\displaystyle - a_{per}^* (u^*)'' = f$.
In this case, the convergence of the whole sequence $u^{\varepsilon}$ to $u^*$ would be in contradiction with the results obtained above.

\subsection{Counter-example for $d=2$, $p=2$}

We next study the case where $a \in \mathbf{A}^p$ for $p=d$, more specifically when $p=d=2$. We define $a(x) =  2 + \sin(\ln(1+\ln(1+|x|))$,
for every $x \in \mathbb{R}^d$, which satisfies assumptions \eqref{hypothèses1}, \eqref{hypothèses2} and \eqref{hypothèses22}. We remark there exists a constant $C$ such that for every $x \in \mathbb{R}^2$ with $|x|>1$,
$|\nabla a(x)|  \leq \dfrac{C}{\ln(|x|)|x|}$ and the mean value theorem implies
$|\delta a(x)| \leq \dfrac{C}{\ln(|x|)|x|}$ which provides that $\delta a \in L^2(\mathbb{R}^2)^2$.

We then consider the homogenization problem \eqref{equationepsilon_new} on the annulus
$\Omega = \left\{x \in \mathbb{R}^d \ \middle| \ 1<|x|<2\right\} $. We again intend to establish the existence of two sub-sequences $\left(\varepsilon_n^1\right)_{n \in \mathbb{N}}$ and $\left(\varepsilon_n^2\right)_{n \in \mathbb{N}}$ such that $u^{\varepsilon_n^1} \rightarrow u^{*,1}$ and $u^{\varepsilon_n^2} \rightarrow u^{*,2}$ in $L^2(\Omega)$ when $\varepsilon^1_n, \varepsilon^2_n \to 0$ and such that $u^{*,1} \neq u^{*,2}$ thereby proving that homogenization does not hold in this setting. For $n \in \mathbb{N}$, we define $\varepsilon^1_n = \exp\left(-\left(\exp(2 n \pi\right)\right)$ and $\varepsilon^2_n = \exp\left(-\exp\left(\left((4 n +1\right) \dfrac{\pi}{2}\right)\right)$. We have for every $x\in \Omega$ : 
\begin{align*}
    a \left( \dfrac{x}{\varepsilon^1_n}  \right) & = 2 + \sin\left( \ln\left(1+\ln \left( 1 + \dfrac{x}{\varepsilon^1_n} \right) \right)\right)
    = 2 + \sin\left(  \ln \left( 1 + \exp(-2 n \pi)\left(1 + \ln \left( \varepsilon^1_n + x \right)\right) \right) \right), 
\end{align*}
and we can deduce that $a \left( \dfrac{.}{\varepsilon^1_n}  \right)$ uniformly converges to $a^{*,1}\equiv 2$ on $\Omega$. Since $a$ satisfies \eqref{hypothèses1} and \eqref{hypothèses2}, $u^{\varepsilon_n^1}$ is bounded in $H^1(\Omega)$ and, up to an extraction, it weakly converges to a function $u^{*,1}$ in $H^1(\Omega)$. Thus, for every $\varphi \in \mathcal{D}(\Omega)$, we have $\displaystyle  \lim_{n \to \infty} \int_{\Omega}a(x/\varepsilon^1_n) \nabla u^{\varepsilon^{1}_n} (x) \nabla \varphi (x)
dx = \int_{\Omega}2 \nabla u^{*,1}(x) \nabla \varphi (x)
dx$.
We obtain that $u^{1,*}$ is solution in $H^1_0(\Omega)$ to 
\begin{equation}
\label{eqhomog2D1}
\left\{
\begin{array}{cc}
   - 2 \Delta u^{*,1} = f & \text{on } \Omega, \\
    u^{*,1} =0 &  \text{on } \partial \Omega.
\end{array}
\right.
\end{equation}
We similarly have $a \left( \dfrac{x}{\varepsilon^2_n}  \right)  = 2 + \cos\left(  \ln \left( 1 + \exp\left((-(4n +1) \dfrac{\pi}{2}\right)\left(1 + \ln \left( \varepsilon^2_n + x \right)\right) \right) \right) $
and $a \left( \dfrac{.}{\varepsilon^2_n}  \right) $ converges uniformly to $a^{*,2} \equiv 3$ on $\Omega$. Therefore, $u^{\varepsilon^2_n}$ weakly converges in $H^1(\Omega)$ (up to an extraction) to $u^{2,*}$, solution in $H^1_0(\Omega)$ to 
\begin{equation}
\label{eqhomog2D2}
\left\{
\begin{array}{cc}
    -3 \Delta u^{*,2} = f & \text{on } \Omega, \\
    u^{*,2} = 0 & \text{on } \partial \Omega.
\end{array}
\right.
\end{equation}
Clearly $u^{*,1} \neq u^{*,2}$ as soon as $f\neq 0$ and we can conclude that $u^{\varepsilon}$ does not converge in $L^2(\Omega)$. 


\end{document}